\documentclass[11pt,reqno]{amsart}
\pdfoutput=1
\usepackage{calc,amsfonts,amsthm,amscd,amsmath,amssymb,enumerate,graphicx}
\usepackage[numeric,initials,nobysame]{amsrefs}
\usepackage[showlabels,sections,floats,textmath,displaymath]{}
\reversemarginpar
\newlength\fullwidth
\numberwithin{equation}{section}

\DeclareMathSymbol{\leqslant}{\mathalpha}{AMSa}{"36} 
\DeclareMathSymbol{\geqslant}{\mathalpha}{AMSa}{"3E} 
\DeclareMathSymbol{\eset}{\mathalpha}{AMSb}{"3F}     

\def\1{\ifmmode {1\hskip -3pt \rm{I}} \else {\hbox {$1\hskip -3pt \rm{I}$}}\fi}

\newcommand{\var}{\operatorname{Var}}

\newcommand{\given}{\;\big|\;}
\renewcommand{\epsilon}{\varepsilon}
 
\newcommand{\be}{\begin{equation} } \newcommand{\tmix}{T_{\textsc{mix}}}

\newtheorem{Theorem}{Theorem}[section]
\newtheorem{Lemma}[Theorem]{Lemma}
\newtheorem{Proposition}[Theorem]{Proposition}
\newtheorem{Corollary}[Theorem]{Corollary}

\newtheorem{claim}[Theorem]{Claim}
\newtheorem{definition}[Theorem]{Definition}

\newtheorem{maintheorem}{Theorem}
\newtheorem{maincoro}[maintheorem]{Corollary}

\theoremstyle{definition}{

\newtheorem*{remark*}{Remark}

}

\newcommand{\N}{\mathbb N}
\newcommand{\R}{\mathbb R}
\newcommand{\Z}{\mathbb Z}
\newcommand{\bP}{{\bf P}}
\newcommand{\bE}{{\bf E}}
\newcommand{\cA}{\ensuremath{\mathcal A}}
\newcommand{\cB}{\ensuremath{\mathcal B}}

\newcommand{\cE}{\ensuremath{\mathcal E}}

\newcommand{\cG}{\ensuremath{\mathcal G}}

\newcommand{\cL}{\ensuremath{\mathcal L}}

\newcommand{\cR}{\ensuremath{\mathcal R}}

\newcommand{\bbH}{{\ensuremath{\mathbb H}} }

\newcommand{\bbN}{{\ensuremath{\mathbb N}} }

\newcommand{\bbR}{{\ensuremath{\mathbb R}} }

\newcommand{\bbZ}{{\ensuremath{\mathbb Z}} }

\let\a=\alpha \let\b=\beta   \let\d=\delta  \let\gep=\varepsilon
 \let\g=\gamma \let\h=\eta    \let\k=\kappa  \let\l=\lambda
            
  \let\s=\sigma \let\t=\tau   
 \let\x=\xi 
\let\D=\Delta   \let\G=\Gamma  \let\L=\Lambda 
\let\O=\Omega      

\def\\{\hfill\break}

\def\thsp{\thinspace}
\def\x{\thinspace}
\def\tthsp{\kern .083333 em}

\def\?{\mskip -10mu}
\newfam\msafam
\newfam\msbfam
\newfam\eufmfam
\def\hexnumber#1{%
  \ifcase#1 0\or 1\or 2\or 3\or 4\or 5\or 6\or 7\or 8\or
  9\or A\or B\or C\or D\or E\or F\fi}
\def\({\left(}
\def\){\right)}

\let\neper=e
\let\ii=i

\def\ie{\hbox{\it i.e.\ }}

\let\sset=\subset

\def\nep#1{ \neper^{#1}}

\let\imp=\Rightarrow
\def\tc{\thsp | \thsp}
\let\<=\langle
\let\>=\rangle

\def\Var{ \mathop{\rm Var}\nolimits }

\newcommand{\gap}{\text{\tt{gap}}}

\def\inte#1{\lfloor #1 \rfloor}

\newcommand{\grad}{\nabla}

\newcommand{\gtop}{\textsc{top}}
\newcommand{\gbot}{\textsc{bot}}
\newcommand{\gse}{\gamma_\textsc{se}}
\newcommand{\gsw}{\gamma_\textsc{sw}}

\newcommand{\htx}{\operatorname{ht}}
\newcommand{\gain}{\operatorname{gn}}
\newcommand{\oned}{1\textsc{d} }
\newcommand{\twod}{2\textsc{d} }
\newcommand{\Stop}{S_\textsc{top}}
\newcommand{\Sbot}{S_\textsc{bot}}

\begin{document}
\title[Quasi-polynomial mixing of 2d Ising with plus boundary]{Quasi-polynomial mixing of the 2d stochastic Ising model
with ``plus'' boundary up to criticality}
\author{Eyal Lubetzky}
\address{E.\ Lubetzky\hfill\break
Microsoft Research\\ One Microsoft Way\\ Redmond, WA 98052-6399, USA.}
\email{eyal@microsoft.com}
\author {Fabio Martinelli}
\address{F.\ Martinelli\hfill\break
Dipartimento di Matematica\\ Universit\`a Roma Tre\\ Largo S.\ Murialdo 1 \\ 00146 Roma,~Italia.}
\email{martin@mat.uniroma3.it}
\author{Allan Sly}
\address{A.\ Sly\hfill\break
 Microsoft Research\\ One Microsoft Way\\ Redmond, WA 98052-6399, USA.}
\email{allansly@microsoft.com}
\author {Fabio Lucio Toninelli}
\address{F.L.\ Toninelli\hfill\break
CNRS and ENS Lyon\\ Laboratoire de Physique\\
46 All\'ee d'Italie, 69364 Lyon, France.}
\email{fabio-lucio.toninelli@ens-lyon.fr}

\thanks{This work was
  supported by the European Research Council through the ``Advanced
  Grant'' PTRELSS 228032}

\begin{abstract}
We considerably improve upon the recent result
  of~\cite{cf:MT} on the mixing time of Glauber dynamics for the \twod Ising model in a box of
  side $L$ at low temperature and with random boundary conditions
whose distribution $\bP$ stochastically dominates the
  extremal plus phase. An important special case is when $\bP$ is
  concentrated on the homogeneous all-plus configuration, where the mixing time $\tmix$ is conjectured to be polynomial in $L$. In~\cite{cf:MT} it was shown that for a large enough inverse-temperature $\b $
  and any $\gep >0$ there exists
  $c=c(\b,\gep)$ such that $\lim_{L\to\infty}\bP\(\tmix\ge \exp({cL^\gep})\)
  =0$. In particular, for the all-plus boundary conditions and $\beta$ large enough $\tmix\le \exp({cL^\gep})$.

    Here we show that the same conclusions hold for all $\beta$ larger than the critical value
  $\beta_c$ and with $\exp({cL^\gep}) $ replaced by
  $L^{c \log L }$ (\ie quasi-polynomial mixing). The key point is a
modification of the inductive scheme of~\cite{cf:MT} together with refined equilibrium estimates that hold up to criticality, obtained via duality and random-line representation tools for the Ising model. In particular, we establish new precise bounds
on the law of Peierls contours which quantitatively
  sharpen the Brownian bridge picture established e.g.\ in~\cites{cf:GI,cf:Higuchi,cf:Hryniv}.
\end{abstract}


\maketitle

\thispagestyle{empty}

\vspace{-1cm}

\section{Introduction}\label{sec:intro}

The Ising model on lattices at and near criticality has been the focus of numerous research papers since its introduction in 1925, establishing it as one of the most studied models in mathematical physics. In two dimensions the model was exactly solved by Onsager~\cite{cf:Onsager} in 1944, determining its critical inverse-temperature $\beta_c=\frac12\log(1+\sqrt{2})$ in the absence of an external magnetic field. While the classical study of the Ising model concentrated on its static properties, over the last three decades significant efforts were dedicated to the analysis of stochastic dynamical systems that both model its evolution and provide efficient methods of sampling it. Of particular interest is the interplay between the behaviors of the static and dynamical models as they both undergo a phase transition at the critical $\beta_c$.

The Glauber dynamics for the Ising model (also known as the \emph{stochastic Ising model}), introduced by Glauber~\cite{cf:Glauber} in 1963, is considered to be the most natural sampling method for it, with notable examples including \emph{heat-bath} and \emph{Metropolis}.
It is known that on a box of side-length $L$ in $\Z^2$ with \emph{free} boundary conditions (b.c.), alongside the phase transition in the range of spin-spin correlations in the static Ising model around $\beta_c$, the corresponding Glauber dynamics exhibits a \emph{critical slowdown}: Its mixing time (formally defined in \S\ref{sec:Glauber}) transitions from being logarithmic in $L$ in the high temperature regime $\beta<\beta_c$ to being exponentially large in $L$ in the low temperature regime $\beta>\beta_c$, en route following a power law at the critical $\beta_c$.

One of the most fundamental open problems in the study of the stochastic Ising model is understanding the system's behavior in the so-called phase-coexistence region under \emph{homogenous} boundary conditions, e.g.\ \emph{all-plus} boundary. In the presence of these b.c.\ the $(-)$ phase becomes unstable and as such the reduced bottleneck between the two phases drastically accelerates the rate of convergence of the dynamics to equilibrium. Indeed, in this case the Glauber dynamics is known to mix in time that is sub-exponential in the surface area of the box, contrary to its low-temperature behavior with free boundary. The central and longstanding conjecture addressing this phenomenon states that the mixing time of Glauber dynamics for the Ising model on a box of side-length $L$ with all-plus boundary conditions is at most polynomial in $L$ at \emph{any temperature}.

So far this has been confirmed on the \twod lattice throughout the one-phase region $\beta <\beta_c$ (see~\cites{cf:MO1,cf:MO2}) and very recently at the critical $\beta=\beta_c$ (see~\cite{cf:LS}). Despite intensive efforts over the last two decades, establishing a power-law behavior for the mixing of Glauber dynamics at the phase-coexistence region $\beta > \beta_c$ under the all-plus b.c.\ remains an enticing open problem.

In~\cite{cf:FH} the precise order of mixing in this regime on a \twod square lattice of side-length $L$ was conjectured to be $L^2$ in accordance with Lifshitz's law (see~\cite{cf:Lifshitz} and also~\cites{cf:CSS,cf:Ogielski,cf:Spohn}). The heuristic behind this prediction argues that when a droplet of the $(-)$ phase is surrounded by the $(+)$ phase at low temperature it proceeds to shrink according to the mean-curvature of the interface between them.
Unfortunately, rigorous analysis is still quite far from establishing the expected Lifshitz behavior of $O(L^2)$ mixing.

Until recently the best upper bound on the mixing at the phase-coexistence region under the all-plus boundary was $\exp(L^{1/2+o(1)})$ due to~\cite{cf:Martinelli} and valid for large enough $\beta$. This bound from 1994 was substantially improved in  a recent breakthrough paper~\cite{cf:MT}, where it was shown (as a special case of a result on a wider class of b.c.) that for a sufficiently large $\beta$ and any $\epsilon>0$ the mixing time is $\exp(O(L^\epsilon))$. The approach of~\cite{cf:MT}
hinged on a novel inductive scheme on boxes with \emph{random} boundary conditions, combined with a careful use of the so-called Peres-Winkler \emph{censoring inequality}; these ideas form the foundation of the present paper. Note that the requirement of large $\beta$ in~\cites{cf:Martinelli,cf:MT} was essential in order to make use of results of~\cite{cf:DKS} on the Wulff construction, available only at low enough temperature by cluster expansion methods. For smaller values of $\beta> \beta_c$ the best known estimates on the mixing time are due to~\cite{cf:CGMS} and of the weaker form $\exp(o(L))$.

In this work we improve these estimates into an upper bound of the form $L^{O(\log L)}$ on the mixing-time (\ie quasi-polynomial in the side-length $L$) valid for any $\beta > \beta_c$.
The key to our analysis is a modification of the recursive framework introduced in~\cite{cf:MT} combined with refined equilibrium estimates that hold up to criticality. To establish these, in lieu of relying on cluster-expansions, we utilize duality and the random-line representation machinery for the high temperature Ising model developed in~\cites{cf:PV1,cf:PV2}.


A key new element of our proof concerns fine estimates on the
fluctuations of cluster boundaries. Whenever the boundary is pinned at
two vertices $u=(0,0)$ and $v=(\ell,0)$, the contour of the cluster is
known to converge to the Brownian bridge
(cf.~\cites{cf:DH,cf:Higuchi,cf:Hryniv}). This does not, however,
allow us to estimate the probability of events when these converge to
0 in the limit. In particular, we are interested in:
(i) a Gaussian
bound for the
probability that the contour would reach height $x\sqrt{\ell}$
(established in~Theorem~\ref{thm-strip-large-dev});
(ii) the probability that the contour remains in the upper half-plane,
an event that would have probability $1/\ell$ were the contour to
behave like a \oned random walk of length $\ell$ conditioned to return
to 0. In~\S\ref{sec:strip} (see Theorem~\ref{thm:positive}) we prove
that up to multiplicative constants this indeed holds for a given
contour.

These then provide important tools in estimating the probability of various other events characterizing the Ising interfaces at equilibrium.

\subsection{Glauber dynamics for the Ising model}\label{sec:Glauber}
Let $\L$ be a generic finite subset of $\bbZ^2$.
Write $x\sim y$ for the nearest-neighbor relation in $\Z^2$ (\ie $x\sim y$ if
$\max_{i=1,2}|x_i-y_i|= 1$) and define $\partial \L$, the boundary of $\L$, to be the nearest-neighbors of $\L$ in $\Z^2\setminus\L$:
\[ \partial \L := \{x\in \Z^2\setminus\L\;:\; x\sim y\mbox{ for some }y \in \L\}\,.\]
The classical Ising model on $\L$ with no external magnetic field is
a spin-system whose set of possible configurations is $\Omega_\L=\{-1,+1\}^\Lambda$.
Each configuration $\sigma\in\Omega_\Lambda$
corresponds to an assignment of plus/minus spins to the sites in $\Lambda$ and has a
statistical weight determined by the Hamiltonian
\[ H_\L^{\t} ( \s ) =
- \sum_{\substack{x,y \in \L \\ x\sim y}} \s_x\s_y
- \sum_{\substack{x \in \L\,,\, y  \in \partial\L \\ x\sim y}}\s_x \t_y \, ,
\]
where $\tau \in \Omega_{\partial \Lambda}$ forms the boundary conditions (b.c.) of the system.
The Gibbs measure associated to the spin-system with boundary conditions
$\t$ is
\begin{equation}
  \label{eq-gibbs-def}
\pi^{\t}_\L ( \s ) = \frac{1}{Z_{\b,\L}^{\t}}
  e^{ - \b H_\L^{\t} ( \s )} \qquad(\sigma\in\Omega_\L)\,,
\end{equation}
where $\b$ is the inverse of the temperature (\ie $\b = \frac{1}{T}$) and
the partition-function $Z_{\b,\L}^{\t}$ is a normalizing constant.  When the boundary
conditions are uniformly equal to $+1$ (resp.\ $-1$) we will denote the Gibbs measure by $\pi_\L^+$ (resp.\ $\pi_\L^-$). Throughout the paper we will omit the superscript $\t$ and
  the subscript $\L$ from the notation of the Gibbs measure $\pi$ when these are made clear from the context.

The Gibbs measure enjoys a useful \emph{monotonicity} property that will play a key role in our analysis. Consider the usual partial order on $\O_\L$
whereby $\s \le \h$ if $\s_x\le\h_x$ for all $x\in \L$.
A function $f:\O_\L\mapsto\mathbb R$ is {\it monotone increasing}
({\it decreasing})
if $\s\le\h$ implies $f(\s)\le f(\h)$ ($f(\s)\ge f(\h)$).
An event is increasing (decreasing) if its
characteristic function is increasing (decreasing).
Given two probability measures $\mu,\nu$ on $\O_\L$
we say that $\mu$ is \emph{stochastically dominated} by $\nu$, denoted by $\mu \preceq \nu$,
 if $\mu(f) \le\nu(f)$ for all increasing
functions $f$ (here and in what follows $\mu(f)$ stands for $\int f d\mu$).
According to these notations the well-known FKG inequalities~\cite{cf:FKG} state that
\begin{itemize}
\item If $\t\le\t'$ then $\pi_\L^\t \preceq \pi_\L^{\t'}$.
\item If $f$ and $g$ are increasing then
   $\pi_\L^\t(f g) \ge \pi_\L^\t(f) \pi_\L^\t(g)$.
\end{itemize}

The phase transition regime in the \twod Ising model occurs at low
temperature and it is characterized by spontaneous magnetization
in the thermodynamic limit.
There is a critical value $\b_c$ such that for all $\b > \b_c$,
\begin{equation}
\label{magnetization}
\lim_{\L\to \bbZ^2} \pi_\L^+ (\s_0) = -
\lim_{\L\to \bbZ^2} \pi_\L^- (\s_0) = m_\b >0 \, .
\end{equation}
Furthermore, in the thermodynamic limit the measures $\pi_\L^+$ and
$\pi_\L^-$ converge (weakly) to two distinct Gibbs measures $\pi^+_\infty$
and $\pi^-_\infty$
which are measures on the space $\Omega_{\Z^2}$, each representing a pure state. We will focus on this phase-coexistence region $\beta > \beta_c$.

The Glauber dynamics for the Ising model is a family of continuous-time Markov chains on the state space $\O_\L$, reversible with respect to the Gibbs distribution $\pi_\L^\t$. An important and natural example of this stochastic dynamics is the heat-bath dynamics, which we will now describe, postponing the formulation of the general Glauber dynamics to \S\ref{sec:prelim-glauber}. Note that our results apply to all of these chains (e.g., Metropolis etc.) by standard arguments for comparing their mixing times (see e.g.~\cite{cf:Martinelli97}).

The \emph{heat-bath dynamics} for the Ising model $\O_\L$ is defined as follows. With a rate one independent Poisson process for each vertex $x$,
the spin $\s_x$ is refreshed by sampling a new value from the set
$\{-1,+1\}$ according to the conditional Gibbs measure
\[\pi^\t_{\s,x}:=\pi_\L^\t\left(\cdot \mid \s_y,\, y\neq x\right)\,.\]
It is easy to verify that the heat-bath chain is indeed reversible with respect to $\pi_\Lambda^\tau$ and is characterized by the generator
\[(\cL_\L^{\t} f)(\s) = \sum_{x\in \L} \left[ \pi^\t_{\s,x} (f) - f(\s)
  \right]\,,
\]
where $\pi^\t_{\s,x}(f)$ is the average of $f$ with respect to the conditional Gibbs measure $\pi^\t_{\s,x}$ acting only on the variable $\s_x$.
The Dirichlet form associated to $\cL_\L^{\t}$ takes the form
\[
\cE_\L^{\t}(f,f) = \sum_{x\in \L} \pi^{\t}_\L \left(\, \Var^\t_{\s,x}(f)\,\right)
\]
where $\Var^\t_{\s,x}(f)$ denotes the variance with respect to
$ \pi^\t_{\s,x}$. It is possible to extend the above definition of the generator $\cL_\L^{\t}$ directly to the whole lattice $\bbZ^2$ and get a well defined
Markov process on $\O_{\bbZ^2}$ (see e.g.~\cite{cf:Liggett}). The
latter will be referred to as the infinite volume Glauber dynamics, with generator denoted by $\cL_{\Z^2}$.

We will denote by $\mu_t^\sigma$ the distribution of the chain at
time $t$ when the starting configuration is identically equal to $\sigma$. For instance, for any $f$ and $\sigma$ the expectation of $f$ w.r.t.\ $\mu_t^\s$ is given by $(T_\L^\tau(t))f(\sigma)$ where $T_\L^\t$ is the Markov semigroup generated by $\cL_\L^\t$.
The notation $\mu_t^\pm$ will stand for the corresponding quantity for an initial configuration of either all-plus or all-minus.

A key quantity that measures the rate of convergence of Glauber dynamics to stationarity is the gap in the spectrum of its generator, denoted by $\gap_\L^\t$. The Dirichlet form associated with $\cL_\L^\tau$ produces the following characterization for the spectral-gap:
\[
\gap_\L^\t = \inf_f \frac{\cE_\L^\tau(f,f)}
{\var^\t_\L(f)}\,,
\]
where the infimum is over all nonconstant $f\in L^2(\pi_\L^\t)$.
Another useful measure for the speed of relaxation to equilibrium is the \emph{total-variation mixing time} which is defined as follows. Recall that the total-variation distance between two measures $\phi,\psi$ on a finite probability space $\O$ is defined as
\[
\|\phi-\psi\| := \sup_{A \subset\Omega} \left|\phi(A) - \psi(A)\right| = \frac{1}{2}\sum_{x\in\Omega} |\phi(x)-\psi(x)|\,.\]
For any $\epsilon\in (0,1)$, the $\epsilon$-mixing-time of the Glauber dynamics is given by
\[
  \tmix(\epsilon):=\inf\Bigl\{t>0: \sup_{\s\in\O_\L} \|\mu_t^\s-\pi_\L^\t\|\le \epsilon\Bigr\}.
\]
When $\epsilon=1/(2e)$ we will simply write $\tmix$.
This particular definition yields the following well-known inequalities (see e.g.~\cites{cf:SaloffCoste,cf:LPW}):
\begin{align*}
  \sup_{\s\in\O_\L} \|\mu_t^\s-\pi_\L^\t\|&\leq \exp(-\inte{t/\tmix}) \quad \mbox{ for all $t\ge 0$}\,,\\
  \frac{1}{\gap} \leq \tmix &\leq \frac1{\gap} \log \frac{2e}{\pi_{\min}}\qquad\mbox{ where }\pi_{\min} = \min_{\s\in\O_\L}\pi_\L^\t(\s)\,.
\end{align*}
The last inequality shows that in our setting $\gap^{-1}$ and $\tmix$ are always within a factor of $O(|\L|)$ from one another (to see this, observe that $\pi_\L^\t(\sigma)/ \pi_\L^\t(\sigma') \leq \exp(O(|\L|))$ for any $\s,\s'\in\O_\L$ by Eq.~\eqref{eq-gibbs-def} whereas $|\O_\L|=2^{|\L|}$).
One could restate our results as well as the analogous conjecture on the polynomial mixing time under all-plus b.c.\ in terms of $\gap^{-1}$ (expected to have order $L$, the side-length of $\Lambda$, for any $\beta > \beta_c$; see~\cites{cf:BM,cf:CMST}).

\subsection{Main results}
We are now in a position to formalize the main contribution of this paper. The following theorem is the counterpart of the main result obtained by two of the authors in~\cite{cf:MT}. Here we feature an improved estimate that in addition holds not only for large enough $\beta$ but throughout the phase-coexistence region.
\begin{maintheorem}
\label{th:quadrato}
For any $\beta>\beta_c$ there exists some $c(\beta)>0$ so that the following
holds for the Glauber dynamics for the Ising model on the square $\Lambda_L$ at inverse-temperature $\beta$.
If $L$ is of the form $L=2^n-1$ for some integer $n$ then:
\begin{enumerate}
\item If the boundary conditions $\tau$ are sampled from a law $\bP$ that
either stochastically dominates the pure phase $\pi^+_\infty$ or is
stochastically dominated by $\pi^-_\infty$ then
\begin{eqnarray}
  \label{eq:maintmix00}
  \bE \|\mu^\pm_{t_L}-\pi^\tau\| \le c/L \quad\mbox{ for $t_L = L^{c\log L}$}\,.
\end{eqnarray}
In particular,
\begin{eqnarray}
\label{eq:maintmix0}
  \bP\bigl(\tmix\ge L^{c\log L}\bigr)\le c/L\,.
\end{eqnarray}
\item The estimates \eqref{eq:maintmix00},\eqref{eq:maintmix0} also hold if $\bP$ is stochastically dominated by $\pi^-_\infty$ on
one side of $\L_L$ and stochastically dominates $\pi^+_\infty$ on the union of the other three sides. A similar statement holds if the roles of $(+)$ and $(-)$ are reversed.
\end{enumerate}
\end{maintheorem}
The most natural consequence of the above result is obtained when $\bP$ concentrates on homogenous boundary conditions, where the best previous bounds were $\exp(O(L^{\gep}))$ for any $\gep>0$ and $\beta$ large enough (\cite{cf:MT}) along with $\exp(o(L))$ for all other $\beta>\beta_c$ (\cite{cf:CGMS}).
\begin{maincoro}
\label{maincor-allplus}
For any $\beta > \beta_c$ there exists some $c(\beta)>0$ so that the mixing time of Glauber dynamics for the Ising model on the square $\Lambda_L$ with b.c.\ $\tau\equiv +1$ satisfies
\begin{eqnarray}
\label{eq:maintmix}
  \tmix \le L^{c \log L}\,.
\end{eqnarray}
The same bound holds if the boundary conditions are $(+)$ on three sides
and $(-)$ on the remaining one, and similarly if $(+)$ is replaced by $(-)$.
\end{maincoro}

We believe that improving the above $L^{c \log L}$ bound into the conjectured polynomial one would require substantial new ideas. Indeed, in the present recursive framework in which the final scale of the system is reached via a doubling sequence, at each step the mixing-time estimate worsens by a power of $L$ (hence the quasi-polynomial bound).
For a polynomial upper bound one could not afford to lose more than a constant factor on average along these steps.

One may also apply Theorem~\ref{th:quadrato} to deduce the mixing behavior of the \twod Ising model under Bernoulli boundary conditions, as illustrated by the next corollary. Here and in what follows we say that an event holds \emph{with high probability} (w.h.p.) to denote that its probability tends to $1$ as the size of the system tends to $\infty$.
\begin{maincoro}
\label{maincor-bernoulli}
Let $\beta > \beta_c$ and consider Glauber dynamics for the Ising model on the square $\Lambda_L$ with b.c.\ $\tau$ comprised of i.i.d.\ Bernoulli variables, $\bP(\tau(x)=+1) = p_+$ for some $p_+ \geq \frac12(1+\tanh(4\beta))$. Then w.h.p.\ $\tmix \le L^{c \log L}$ for some $c(\beta)>0$.
\end{maincoro}
To obtain the above corollary observe that the Bernoulli boundary conditions with the above specified $p^+$ clearly stochastically dominate the marginal of $\pi_\infty^+$ on $\partial\L$.

The mixing time of Glauber dynamics for Ising on a finite box under all-plus b.c.\ is closely related to the asymptotic decay of the time auto-correlation function in the infinite-volume dynamics on $\Z^2$ started at the plus phase. Here it was conjectured in~\cite{cf:FH} that the decay should follow a stretched exponential of the form $\exp(-c\sqrt{t})$. As a by-product of Corollary~\ref{maincor-allplus} (and standard monotonicity arguments) we obtain a new bound on this quantity, improving on the previous estimate due to~\cite{cf:MT} of $\exp(-(\log t)^{c})$ with arbitrarily large $c$ which was applicable for large enough $\beta$.
\begin{maincoro}
\label{maincor-autocor}
Let $\b>\b_c$, let $f(\s) =\s_0$ and define $\rho(t)\equiv
\Var^+_\infty\(\nep{t\cL_{\Z^2}}f\)$ to be the time autocorrelation of the spin
at the origin started from the plus phase $\pi^+_\infty$ (the variance is w.r.t.\ the plus phase $\pi^+_\infty$). Then there exists some $c(\beta) > 0$ such that for any $t$,
\begin{equation}
\label{eq:decorrel}
\rho(t)\le \exp\bigl(-e^{c \sqrt{\log t}}\bigr) \,.
\end{equation}
\end{maincoro}

\subsection{Related work}

Over the last two decades considerable effort was devoted to the formidable problem of establishing polynomial mixing for the stochastic Ising model on a finite lattice with all-plus b.c. Following is a partial account of related results.

Analogous to its conjectured behavior on $\Z^2$, the mixing of Glauber dynamics for the Ising model on the lattice $\Z^d$ in any fixed dimension $d$ is believed to be polynomial in the side-length of the box at any temperature in the presence of an all-plus boundary.
Unfortunately, the state-of-the-art rigorous analysis of the problem in three dimensions and higher is far more limited. Faced with the polynomial lower bounds of~\cite{cf:BM}, the best known upper bound for dimension $d\geq 3$ is $\exp(L^{d-2+o(1)})$ for large enough $\beta$ (as usual $L$ being the side-length) due to~\cite{cf:Sugimine}. Compare this with the case of no (\ie free) boundary conditions case where it was shown in~\cite{cf:Thomas} that $\gap^{-1}$ (and thus also $\tmix$) is at least $c \exp(c' \beta L^{d-1})$ for some $c=c(\beta)>0$ and an absolute constant $c'>0$.

In two dimensions, ever since the work of Martinelli~\cite{cf:Martinelli} in 1994 (an upper bound of $\exp(L^{1/2+o(1)})$ at low enough temperatures) and until quite recently no real progress has been made on the original problem.
Nevertheless, various variants of this problem became fairly well understood. For instance, nearly homogenous boundary conditions were studied in~\cites{cf:Alexander,cf:AY}. Analogues of the problem on non-amenable geometries (in terms of a suitable parameter measuring the growth of balls to replace the side-length) were established, pioneered by the work of~\cite{cf:MSW} on trees and followed by results of~\cite{cf:Bianchi} on a class of hyperbolic graphs of large degrees. The Solid-On-Solid model (SOS), proposed as an idealization of the behavior of Ising contours at low temperatures, was studied in~\cite{cf:MS} where the authors obtained several insights into the evolution of the contours. Finally, the conjectured Lifshitz behavior of $O(L^2)$ was confirmed at zero temperature~\cites{cf:CSS,cf:FSS,cf:CMST}, with the recent work~\cite{cf:CMST} providing sharp bounds also for near-zero temperatures (namely when $\beta \geq c\log L$ for a suitably large $c>0$) in both dimensions two and three.

As mentioned above, the $\exp(L^{1/2})$ barrier was finally broken in the recent paper~\cite{cf:MT}, replacing it by $\exp(c L^\epsilon)$ for an arbitrarily small $\epsilon>0$ and sufficiently large $\beta$ (where the constant $c=c(\b,\gep)$ diverges to $+\infty$  as $\gep \to 0$). At the heart of the proof of the main result of that paper (\cite{cf:MT}*{Theorem 1.6}) was an inductive procedure which will serve as our main benchmark here. We will shortly review that argument in \S\ref{sec:ff-induction} in order to motivate and better understand the new steps gained in the present work.

Finally, there is an extensive literature on the phase-separation lines in the \twod Ising model, going back to~\cites{cf:AR,cf:Gallavotti}. In \S\ref{sec:prelim} we will review the tools we will need from the random-line representation framework of~\cites{cf:PV1,cf:PV2}. For further information
see e.g.~\cite{cf:Pfister} and the references therein.

\section{Preliminaries}\label{sec:prelim}

\subsection{General Glauber dynamics}\label{sec:prelim-glauber}
The class of Glauber dynamics for the Ising model on a finite box $\Lambda \subset \Z^2$ consists of the continuous-time Markov chains on the state space $\O_\L$ that are given by the generator
\begin{equation}
  \label{eq-Glauber-gen}
  (\mathcal{L}^\tau_\L f)(\sigma)=\sum_{x\in \Lambda} c(x,\sigma) \left(f(\sigma^x)-f(\sigma)\right)\,,
\end{equation}
where $\sigma^x$ is the configuration $\sigma$ with the spin at $x$ flipped and the transition rates $c(x,\sigma)$ should satisfy the following conditions:
\begin{enumerate}[(1)]
\item \emph{Finite range interactions}: For some fixed $R>0$ and any $x \in \Lambda$, if $\sigma,\sigma' \in \O_\L$ agree on the ball of diameter $R$ about $x$ then
$c(x,\sigma)=c(x,\sigma')$.
\item \emph{Detailed balance}: For all $\sigma\in \O_\L$ and $x \in \Lambda$,
\[ \frac{c(x,\sigma)}{c(x,\sigma^x)} = \exp\Bigl(-\beta \grad_x H_\L^\t(\s) \Bigr)\,,
\]
where $\grad_x H_\L^\t(\s) = H_\L^\t(\s^x)-H_\L^\t(\s) = 2\s_x \Bigl[\sum_{\substack{y\in\L\\ y \sim x}}\sigma_y + \sum_{\substack{y\in\partial\L \\ y\sim x}} \tau_y\Bigr]$.
\item \emph{Positivity and boundedness}: The rates $c(x,\sigma)$ are uniformly bounded from below and above by some fixed $C_1,C_2 > 0$.
\item \emph{Translation invariance}: If $\sigma \equiv \sigma'(\cdot + \ell)$, where $\ell \in \Lambda$ and addition is according to the lattice metric,
then $c(x,\sigma) = c(x+\ell,\sigma')$ for all $x \in \Lambda$.
\end{enumerate}
The Glauber dynamics generator with such rates defines a unique Markov process, reversible with respect to the Gibbs measure $\mu_\Lambda^\tau$.
The two most notable examples for the choice of transition rates are
\begin{enumerate}[(i)]
\item \emph{Metropolis}: $  c(x,\sigma) = \exp\Bigl(-\beta \grad_x H_\L^\t(\s)\Bigr)  \;\wedge\; 1\; $.
\item \emph{Heat-bath}:   $\;c(x,\sigma) = \Bigl[1+ \exp\bigl(\beta\grad_x H_\L^\t(\s)\bigr)\Bigr]^{-1}\;$.
\end{enumerate}
See e.g.~\cite{cf:Martinelli97} for standard comparisons between these chains, in particular implying that their individual mixing times are within a factor of at most $O(|\L|)$ from one another (hence our results apply to every one of these chains).

\subsection{Surface tension}\label{sec:prelim-tau}
Denote by $\tau_\beta(\theta)$ the surface tension that corresponds to the angle $\theta$, defined as follows. Associate with each angle $\theta\in[-\pi/4,\pi/4]$ the unit vector $\vec{n}_\theta=(\cos\theta,\sin\theta) \in \mathbb{S}^1$ and the following b.c.\ $\eta_\theta$ for $\Lambda_L=\{-\lfloor L/2\rfloor ,\ldots,\lfloor L/2\rfloor\}^2$:
\[\eta_\theta(v) = \left\{\begin{array}{ll}
+1 & \mbox{if $\left<v,\vec{n}_\theta\right> > 0$}\,,\\
-1 & \mbox{if $\left<v,\vec{n}_\theta\right> \leq 0$}\,.
\end{array}\right.
\]
Let $Z_{\b,\L_L}^{\eta_\theta}$ be the partition-function of the corresponding Ising model and, as usual, let $Z_{\b,\L_L}^+$ denote the partition-function under the all-plus b.c. The surface tension in the direction orthogonal to $\vec{n}_\theta$ is the limit
\[ \tau_\b(\theta) = -\lim_{L\to\infty} \frac{\cos\theta}{L}\log\frac{Z_{\b,\L_L}^{\eta_\theta}}{Z_{\b,\L_L}^{+}}\,,
\]
which gives rise to an even analytic function $\tau_\b$ with period $\pi/2$ on $\R$ (a closed formula appears e.g.\ in~\cite{cf:PV1}*{Section 5}). One can then extend the definition of $\tau_\beta$ to $\R^2$ by homogeneity, setting $\tau_\b(x)=\tau_\b(\theta_x)|x|$, where $|x|$ denotes the Euclidean norm of $x$ and $\theta_x$ is the angle it forms with $(1,0)$. For all $\beta>\beta_c$ this qualifies as a norm on $\R^2$.

The surface tension measures the effect of the interface induced by the boundary conditions $\eta_\theta$ on the free-energy and thus plays an important role in the geometry of the low temperature Ising model. For instance, it was shown in~\cite{cf:Shlosman} that the large deviations of the magnetization in a square are governed by $\tau_\b(0)$ (also see~\cites{cf:Ioffe1,cf:Ioffe2}).

One of the useful properties of the surface tension is the \emph{sharp triangle inequality} (see for instance~\cite{cf:PV1}*{Proposition~2.1}): For any $\beta>\beta_c$ there exists a strictly positive constant $\kappa_\beta$ such that for any $x,y\in\R^2$ we have
 \begin{align}
 \tau_{\beta}(x)+\tau_{\beta}(y) - \tau_{\beta}(x+y) \geq \kappa_\beta \left(|x|+|y|-|x+y|\right)\,,\label{eq-sharp-tri-ineq}
 \end{align}
 A thorough account of additional properties of the surface tension may be found e.g.\ in~\cite{cf:DKS} and~\cite{cf:Pfister}.

\subsection{Duality}\label{sec:prelim-duality}
Let $\bbZ^{2*}:=\Z^2+(\frac12,\frac12)$ denote the dual lattice to
$\bbZ^2$. The collection of edges of $\bbZ^2$ and  of $\bbZ^{2*}$ will
be denoted by $E_{\Z^2}$ and $E_{\Z^{2*}}^*$ respectively. It is useful to identify an
edge $e=(x,y)\in E_{\Z^2}$ with the closed unit segment in $\bbR^2$ whose
endpoints are $\{x,y\}$, and similarly do so for edges in $E_{\Z^{2*}}^*$. To each edge
$e=(x,y)\in E_{\Z^2}$ there corresponds a
unique dual edge $e^*\in E_{\Z^{2*}}^*$ defined by the condition
$e\cap e^*\neq \emptyset$.

Given a finite box $\L\sset \bbZ^2$ of the form
$\L=\{x=(x_1,x_2)\in \bbZ^2:\ |x_1|\le L_1, \ |x_2|\le L_2\}$, the
dual box $\L^*\sset \bbZ^{2*}$ is $\L^*=\{x=(x_1,x_2)\in \bbZ^{2*}:\
|x_1|\le L_1+\frac12, \ |x_2|\le L_2+\frac12\}$.  The set of dual edges of
$\L^*$, denoted by $E^*_{\L^*}$, is the set of dual edges for which both
endpoints lie in $\L^*$. Notice that for each edge
$e=(x,y)\in E_{\Z^2}$ such that
$\{x,y\}\cap \L\neq \emptyset$, the corresponding
dual edge $e^*$ necessarily belongs to $E^*_{\L^*}$.
These definitions readily generalize to an arbitrary finite
$\mathcal{G}\subset \bbZ^2$, in which case $\mathcal{G}^*\subset\bbZ^{2*}$ consists
of all dual sites whose $L^1$-distance from $\mathcal{G}$ equals $1$.

For any $\b > \b_c$ we associate the dual inverse-temperature $\b^*$
via the duality relation $\tanh(\b^*)=\nep{-2\b}$. Notice that for any $\b>\b_c$ the dual inverse
temperature $\b^*$ lies below $\b_c$ which is the unique fixed point of the
map $\b\mapsto \b^*$. We will often refer to the Gibbs measure on a subset of the dual lattice $\L^*$ at the inverse-temperature $\b^*$ under free boundary, denoting it by $\pi^*_{\L^*}$.
The following well-known fact addresses the exponential decay of the two-point
correlation function for the free Ising Gibbs measure above the critical temperature.
\begin{Lemma}[e.g.~\cite{cf:MW}*{p309 Eq.~(4.39)}, together with the GKS inequalities~\cites{cf:Griffiths,cf:KS}]\label{lem-spin-spin}
Let $\L \subset \Z^2$ and $\b>\b_c$. There exists some $C_\beta>0$ such that for any $x,y\in \L^*$,
\[  \pi_{\L^*}^*(\s_x\,\s_y)\leq \left(\frac{C_\beta}{\sqrt{x-y}}\wedge 1\right)\exp\big(-\tau_{\beta}(x-y)\big)\,.
\]
\end{Lemma}
A matching exponent for the spin-spin correlation was established by~\cite{cf:GI} for two opposite points in the (dual) infinite strip. Let $S = \{1,\ldots,\ell\}\times\Z$ for some integer $\ell$ and fix $\b>\b_c$. In the dual $S^*$ we let
$x=(\frac12,\frac12)$ and $y=(\ell+\frac12,\frac12)$ and consider the free Gibbs measure at inverse-temperature $\b^*$.
It was shown in \cite{cf:GI}*{formula (2.22)} that in this setting there exists some $c_\b> 0$ such that
\begin{equation}
  \label{eq-strip-two-pt-corr}
\pi_{S^*}^*(\s_x\,\s_y) = \frac{c_\b+o(1)}{\sqrt{\ell}}\exp(-\tau_{\beta}(0)\ell) \,,
\end{equation}
where the $o(1)$-term tends to $0$ as $\ell\to\infty$.

\begin{figure}
\centering
\includegraphics[width=0.6\textwidth]{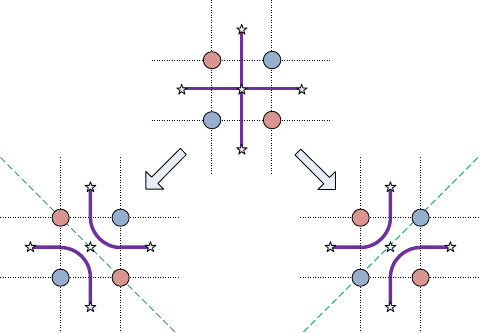}
\caption{SE and SW splitting-rules for forming the contours}
\label{fig:split}
\end{figure}

\subsection{Contours}\label{sec:prelim-contours}
Let $\cG=(V,E)$ be a finite subgraph of $(\bbZ^{2*},E_{\Z^{2*}}^*)$.
The \emph{boundary} of a subset of dual edges $B \subset E$, denoted by $\d B$, is the set of vertices of $V$ with an odd number of adjacent edges of $B$. If $\d B=\emptyset$ we say that $B$ is \emph{closed}, otherwise it is \emph{open}.

A \emph{chain} of sites of length $k$ from $x$ to $y$ in $\cG$ has the
standard definition of a sequence of sites $x=u_0,u_1,\ldots,u_k=y$ such that $u_i\in \cG$
and $|u_i-u_{i-1}|=1$ for all $i$. A $*$-\emph{chain} from $x$ to $y$ is similarly defined with the exception that the distance requirement is relaxed into $1\leq|u_i-u_{i-1}|\leq\sqrt{2}$ for all $i$.
A \emph{path} from $x$ to $y$ in $B$ is a chain of sites consisting of edges of $B$, that is $(u_{i-1} ,u_i) \in B$ for all $i$.
We say that a path is \emph{closed} if its endpoint and starting point coincide, otherwise we say that it is \emph{open}.

A set of dual edges $B\subset E$ can be uniquely partitioned into a finite number of edge-disjoint simple lines in $E \cap \bbR^2$ called \emph{contours}. This is achieved by repeating the following procedure referred to as the \emph{South-East (SE) splitting-rule}: When four bonds meet at a vertex we separate them along the SE-oriented diagonal going through the intersection. Alternatively, one may globally apply the \emph{SW splitting-rule}, analogously defined with the South-West orientation replacing the South-East one (see Figure~\ref{fig:split}).

Contours can be either open or closed (with the same distinction as in paths).
The length of a contour $\gamma$, denoted by $|\gamma|$, is the number of edges in $\gamma$,
and the length of a collection of contours $\underline\g=\{\g_1,\ldots,\g_k\}$ will simply be the sum of all the individual
lengths. Given a finite family of contours $\underline\g=\{\g_1,\ldots,\g_k\}$ we say that it is \emph{compatible}
if it is the contour decomposition of its collection of dual edges $\cup_i \g_i$. We further say that
$\underline\g$ is \emph{$E$-compatible} (or \emph{$\cG$-compatible}) to emphasize that in addition all the edges of $\cup_i\g_i$  belong to $E$, the edge-set of $\cG$.

Given boundary conditions $\t\in \{-1,1\}^{\bbZ^2}$ and a box $\L$, each spin-configuration $\s$ compatible with $\t$ outside $\L$ (\ie $\s_x=\t_x$
for any $x\notin \L$) can be uniquely specified by giving all the
edges $e=(x,y)\in E_{\Z^2}$ such that $\s_x\neq \s_y$ and
$\{x,y\}\cap \L\neq \emptyset$ (that is, all edges whose endpoint sites disagree).
Equivalently, one can specify the corresponding dual edges of $\L^*$. By applying the
above contour decomposition we see that each configuration $\s$
compatible with $\t$ is uniquely characterized by its collection of
closed and open contours (see Figure~\ref{fig:contours} for an illustration).
The open contours obtained in this manner are called the \emph{phase-separation lines}.

\begin{figure}
\centering
\includegraphics[width=5in]{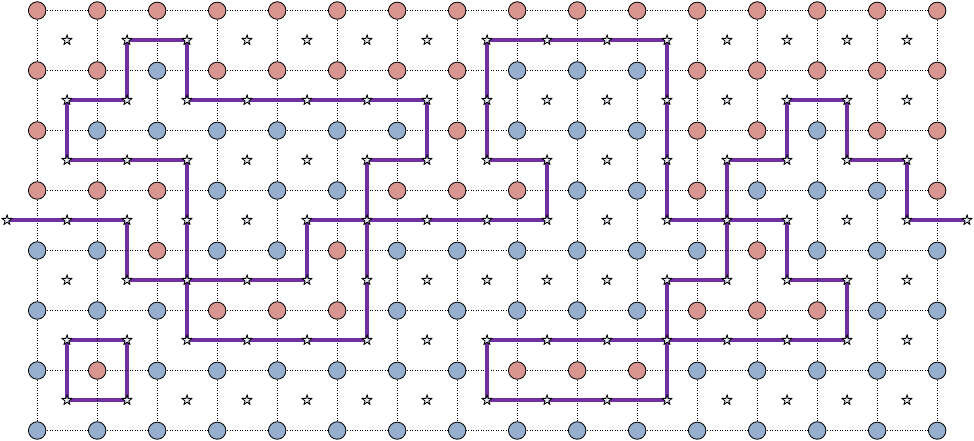}\\
\vspace{0.15in}
\includegraphics[width=2.7in]{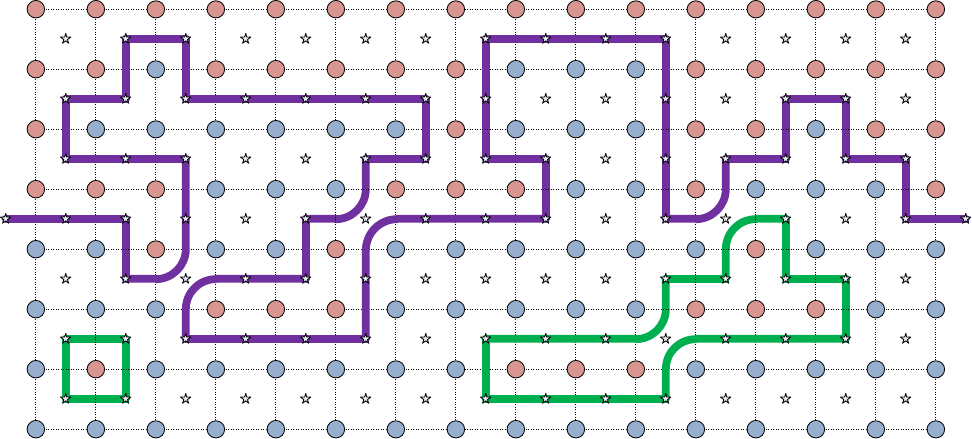}
\hspace{0.05in}
\includegraphics[width=2.7in]{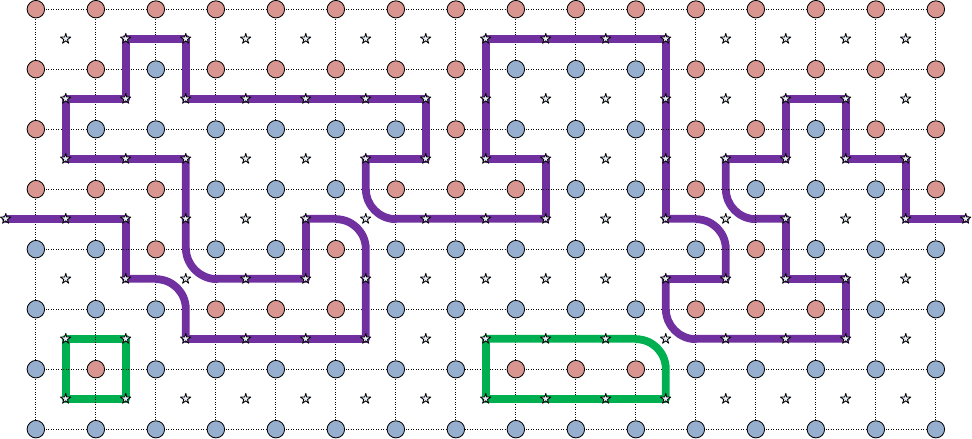}
\caption{Contour decomposition of the edge-set induced by an Ising configuration on a box with mixed boundary conditions according to both SE and SW splitting-rules.}
\vspace{-0.05in}
\label{fig:contours}
\end{figure}

It is clear that the boundary $\d
\underline \l$ of the open contours belongs to $\partial \L^*$ and
must coincide with a certain set $V(\t)$ uniquely
specified by the boundary conditions $\t$ (\ie independent of the values $\s$ gives to the spins of $\L$).
Notice that the cardinality of $V(\t)$, if different
from zero, must be even.

A family of closed and open
simple lines $\underline \g$ is called $\t$-\emph{compatible} if there
exists a configuration $\s$ compatible with $\t$ in $\L$ from which $\underline \g$ is obtained in the above procedure.
One can easily verify that when $\L$ is a box the set of
$\t$-compatible contours coincides with the set of
$E^*_{\L^*}$-compatible contours whose boundary is equal to $V(\t)$.

\subsection{Random-line representation}
For a finite subgraph $\cG=(V,E)$ of $(\bbZ^{2*},E_{\Z^{2*}}^*)$ and an $E$-compatible family of
contours $\underline{\theta}$, two different partition functions $Z(\cG)$ and $Z(\cG\tc
\underline \theta)$ will turn out to be useful for a given $\b>0$:
\begin{align}
  \label{eq:3}
  Z(\cG) &= \sum_{\substack{\underline \g:\, \d \underline \g=\emptyset \\
    \underline \g \text{ is $E$-compatible}}}\nep{-2\b|\underline\g|}\,,
  \\
  Z(\cG\tc \underline \theta) &= \sum_{\substack{\underline \g:\, \d \underline \g=\emptyset \\
    \underline \g \cup \underline \theta\text{ is $E$-compatible}}}\nep{-2\b|\underline\g|}\,.
\end{align}
Using $Z(\cG)$ and $Z(\cG\tc \underline \theta)$ we define the \emph{weight} (not necessarily a probability distribution) corresponding to the family of contours $\underline\theta$, denoted by $q_\cG(\underline \theta)$, to be
\begin{equation}
  \label{eq:2}
  q_\cG(\underline \theta)=
  \left\{\begin{array}
    {ll}
\frac{Z(\cG\tc \underline \theta)}{Z(\cG)}\,  \nep{-2\b |\underline \theta|} &
\mbox{if $\underline \theta$ is $E$-compatible,}\\
0 & \mbox{otherwise.}
  \end{array}
\right.
\end{equation}
The key reason for the above formula is the following \emph{random-line representation} for even-point correlation functions:
Consider the Ising model on $\cG$ at inverse temperature $\b^*$ and free boundary
conditions. Let $\pi^*_\cG$ be the associated Gibbs measure and let $A\sset V$
have even cardinality. Then the following holds (see~\cite{cf:PV1}*{Lemma~6.9}):
\begin{equation}
  \label{eq:7}
  \sum_{\underline \l:\ \d \underline \l =A}q_\cG(\underline \l)=
  \pi_\cG^*\Big(\prod_{x\in A}\s_x\Big)\,.
\end{equation}
\begin{remark*}
If the cardinality of $A$ is odd then the r.h.s.\ of
\eqref{eq:7} is zero by symmetry and the l.h.s.\ is zero due to the
definition of $q_{\cG}(\underline \l)$.
\end{remark*}
Back to the low temperature Ising model in a box $\L$ with boundary
condition $\t$, let $\underline \l$ be a collection of $\t$-compatible
open contours. Then, by construction,
\begin{align}
  \label{eq:8}
  \pi^\t_\L\bigl(\s:\ \underline \l(\s)=\underline \l\bigr)&= \frac{  q_{\L^*}(\underline \l)}{\sum_{\underline \l': \d\underline
    \l'=V(\t)} q_{\L^*}(\underline \l')} = \frac{q_{\L^*}(\underline \l)}{\pi_{\L^*}^*\bigr(\prod_{x\in V(\t)}\s_x\bigl)}\,,
\end{align}
where with a slight abuse of notation we have identified $\L^*$ with
the graph $\cG=(\L^*, E^*_{\L^*})$ and in the last equality we used \eqref{eq:7}.
The above formula will be the starting point of the proof of the new equilibrium estimates,
Propositions~\ref{prop:Bound1} and~\ref{prop:Bound2}.

We conclude this section with some of the
main properties of the weights $q_\cG(\underline \l)$. For further information see~\cites{cf:PV1,cf:PV2}.

\begin{Lemma}[\cite{cf:PV1}*{Lemma 6.3}]
\label{l:6.3}
Let $\cG=(V,E)$ be a finite subgraph of $(\bbZ^{2*},E_{\Z^{2*}}^*)$ and let $\underline \theta$ be a family of
$E$-compatible contours (open and closed). If $\cG'$ is a subgraph of
$\cG$ then $q_{\cG'}(\underline \theta) \geq q_{\cG}(\underline \theta)$.
\end{Lemma}
\begin{remark*}
Lemma~\ref{l:6.3} enables one to extend the definition of the weights $q_\cG(\lambda)$ for finite contours $\lambda$ in a infinite graph $\cG$ by taking a limit
for $n\to\infty$ of $q_{\cG_n}(\lambda)$, where $\cG_n$ is the intersection of $\cG$ with a box of size $n$. By Lemma~\ref{l:6.3} the sequence $q_{\cG_n}(\lambda)$ is monotone non-increasing and non-negative hence its limit indeed exists.
\end{remark*}

Let $\cG=(V,E)$ be a subgraph of $(\bbZ^{2*},E_{\Z^{2*}}^*)$. The \emph{edge-boundary} of an edge $e\in E$, denoted by $\Delta(e)$, is comprised of the edge $e$ itself together
with any edge $e'\in E$ that is incident to it and would belong to the same contour
in the contour decomposition of $E$ via the agreed splitting-rule. For instance, with the SE splitting-rule
the horizontal edge $e=[(x,y),(x+1,y)]$ in
the dual lattice $\Z^{2*}$ would have an edge-boundary of $\Delta(e) = \left\{e, [(x,y),(x,y+1)],
[(x+1,y),(x+1,y-1)]\right\}$. Given a subset of edges $B\subset E$ we define its edge-boundary as $\Delta(B) = \cup_{e\in B}\Delta(e)$.
This definition implies that two contours $\lambda$ and $\gamma$, where $\lambda$ is closed and $\gamma$ is either open or closed, are $\cG$-compatible if and only if the edge-set of $\lambda$ does not intersect $\Delta(\gamma)$ (see the related~\cite{cf:PV1}*{Lemma 6.1}).
The following lemma is a special case of~\cite{cf:PV1}*{Lemma 6.4}):
\begin{Lemma}[\cite{cf:PV1}*{Eq.~(6.17)}]\label{lem-edge-boundary}
Let $\cG=(V,E)$ be a  subgraph of $(\bbZ^{2*},E_{\Z^{2*}}^*)$ and let $\underline\theta$ and $\underline\l$ denote two $\cG$-compatible families of contours with corresponding edge-sets $E_{\underline\theta}$ and $E_{\underline\l}$ respectively. If $\underline\l\cup\underline\theta$ is $\cG$-compatible (or equivalently if
$\Delta(\underline\l)\cap E_{\underline\theta}= \emptyset$) then
\[ q_\cG(\underline\theta \cup \underline\l) = q_{\cG_{\underline\l}}(\underline\theta) q_{\cG}(\underline\l)\,,\]
where $\cG_{\underline\l}$ is the subgraph of $\cG$ given by the edge-set $E \setminus \Delta(\underline\l)$.
\end{Lemma}

We will frequently need estimates on the weight of a contour constrained to go through certain dual sites; to this end, the following definition will be useful. Let $\cG=(V,E)$ and let $\l_1,\l_2$ be two open contours such that $\d\l_1=\{x,y\}$
and $\d\l_2=\{u,v\}$. We say that $\l_1,\l_2$ are \emph{disjoint} if either they are
$\cG$-compatible or their edge-sets are disjoint and the contour
decomposition of the union of their edges is a single contour $\l$. Observe that in the latter case
necessarily $\{x,y\}\cap \{u,v\}\neq \emptyset$.
For a pair of disjoint open contours $\l_1,\l_2$ we write $\l_1 \sqcup \l_2$ to denote either the collection $(\l_1,\l_2)$ in
the former case or the single contour $\l$ in the latter.
\begin{Lemma}[\cite{cf:PV1}*{Lemma 6.5}]
\label{l:6.5}
Let $\cG=(V,E)$ be a  graph in the dual lattice $\Z^{2*}$. For any $x,y,u,v\in V$,
\[    \sum_{\substack{\underline \l= \l_1\sqcup\l_2 \\
\d\l_1=\{x,y\}\,,\, \d\l_2=\{u,v\}}}q_\cG(\underline \l)\le
\sum_{\substack{\l_1 \\
\d\l_1=\{x,y\}}}q_\cG(\l_1) \sum_{\substack{\l_2 \\
\d\l_2=\{u,v\}}}q_\cG(\l_2)\,.
\]
\end{Lemma}
In particular,
\begin{Corollary}[\cite{cf:PV2}*{Eq.~(5.29)}]\label{lem-3-pt-spins}
Let $\cG=(V,E)$ be a graph in the dual lattice $\Z^{2*}$. For any $\b>\b_c$ and any $u,v,z\in V$,
\begin{align*} \sum_{\substack{\lambda:\,\d\lambda=\{u,v\}\\ z\in\lambda}}
q_{\cG}(\lambda) &\leq \biggl(\sum_{\lambda:\,\d\lambda=\{u,z\}}
q_{\cG}(\lambda) \biggr)\biggl(\sum_{\lambda:\,\d\lambda=\{z,v\}}
q_{\cG}(\lambda)\biggr) =\pi_{\cG}^*\left(\s_u\,\s_z\right) \, \pi_{\cG}^*\left(\s_v\,\s_z\right)\,.\end{align*}
\end{Corollary}
Together with Lemma~\ref{lem-spin-spin} the above lemma immediately implies an upper bound on the weights in mention in terms of the surface tensions $\tau_\b(u-v)$ and $\tau_\b(v-z))$. The next lemma provides an analogous bound for the weights of \emph{closed} contours going through a set of prescribed sites.
\begin{Lemma}[\cite{cf:PV2}*{Lemma 5.5 part (ii)}]\label{lem-long-loop}
Let $\cG=(V,E)$ be a graph in $\Z^{2*}$. Let $x_1,\ldots,x_k\in V$ and identify $x_0 \equiv x_k$. Then
\[
\sum_{\substack{\lambda: \d\lambda=\emptyset\\ x_1,\ldots,x_k\in\lambda}} q_{\cG}(\lambda) \leq \exp\bigg( -\sum_{i=1}^k\tau_{\beta}(x_i - x_{i-1})\bigg)\,.\]
\end{Lemma}

\section{Inductive framework for rectangles with ``plus'' boundaries}\label{sec:ff-induction}
In this section we outline the recursive scheme developed in~\cite{cf:MT} which, as mentioned in \S\ref{sec:intro}, established a significantly improved upper bound of $\exp(c L^\epsilon)$ for the mixing time on a box of side-length $L$ with ``plus'' b.c.\ at sufficiently low temperatures.

Given $\gep>0$ (to be thought of as very small) and $L\in \N$ let
\begin{equation*}
  R_L=\{x=(i,j)\in \bbZ^2:\ 1\le i\le L, \, 1\le j\le \lceil L^{\frac 12+\gep}\rceil\}.
\end{equation*}
Similarly one defines the rectangle $Q_L$, the only difference being
that the vertical sides contain now $\lceil (2L+1)^{\frac 12+\gep}\rceil$
sites.
\begin{definition}
  A distribution $\bP$ of b.c.\ for a rectangle $R$ (which
  will be $R_L$, $Q_L$ or some translation of them)
is said to belong to $\mathcal D(R)$ if its marginal on the union of
  North, East and West borders of $R$ is stochastically dominated by
  (the marginal of) the minus phase $\pi^-_\infty$ of the infinite system,
  while the marginal on the South border of $R$ dominates the
  (marginal of the) infinite plus phase $\pi^+_\infty$.
\end{definition}
The most natural example is to take $\bP$ concentrated on the boundary
condition $\tau\equiv -1$ on the North, East and West
borders, and $\tau\equiv+1$ on the South border.
\begin{definition}
\label{def-stat}
For any given $L\in\N, \delta>0, t>0$ consider the Ising model in $R_L$, with boundary condition $\tau$ chosen from
some distribution $\bf P$. We say that
$\cA(L,t,\delta)$ holds if
\begin{eqnarray}
  \bE\|\mu^{\pm}_t-\pi^\t\|\le \delta
\label{eq:A(L)}
\end{eqnarray}
for every $\bP\in \mathcal D(R_L)$.
The statement $\cB(L,t,\delta)$ is defined similarly, the only difference
being that the rectangle $R_L$ is replaced by $Q_L$ (and $\bP$ is required to belong to $\mathcal D(Q_L)$).
\end{definition}
With these definitions the iterative scheme developed in \cite{cf:MT} can be summarized as follows.
\begin{Proposition}[The starting point]
\label{th:raf}
  For every $\beta$ (thus not necessarily large) there exists $c=c(\beta)$ such that for every
$L\in\N$ 
the statements
$
\cA(L,t,e^{-t\,e^{-c L^{1/2 +\gep}}})
$
and
$
\cB(L,t,e^{-t\,e^{-c L^{1/2 +\gep}}})
$
hold.
\end{Proposition}
\begin{remark*}
Notice that the factor $e^{-c L^{1/2 +\gep}}$ in front of the time $t$
is nothing but the negative exponential of  the shortest side of
the rectangle.
\end{remark*}
\begin{Theorem}[The inductive step]
\label{th:rec}
For every $\beta$ large enough there exist constants $c_1,c_2,c_3$ such that:
\begin{equation}
  \label{eq:rec1}
  \cA(L,t,\delta) \imp \cB(L,t_1,\delta_1)\imp \cA(2L+1,t_2,\delta_2)
\end{equation}
where
\begin{align}
  \label{eq:rec2}
\delta_1&=c_1\left(\delta+\nep{-c_2 L^{2\gep}}+L^2\nep{-c_2 \log t}\right) & &; &  t_1&= 2t \\
\d_2&= c_1(\delta_1+\nep{-c_2L^{3\gep}})=c_3(\delta + \nep{-c_2 L^{2\gep}}
+ L^2\nep{-c_2 \log t})& &; &  t_2&= \nep{c_3 L^{3\gep}}t_1=2\nep{c_3 L^{3\gep}}t
\end{align}
\end{Theorem}
\begin{remark*}
In the original statement in \cite{cf:MT} the obvious requirement of $\beta$ large was missing due to a typo.
\end{remark*}

\begin{Corollary}[Solving the recursion]
\label{th:corolla}
In the same setting of Theorem \ref{th:rec},
for every $L\in\{2^{n}-1\}_{n\in\N}$
there exists
\begin{eqnarray}
\Delta(L)\le \exp\left(-c'L^{\gep^2}\right)
\end{eqnarray}
such that
$\cA\(L,t,\Delta(L)\)$ holds for every
$t\ge T(L):=  e^{cL^{3\gep}}$.
\end{Corollary}
In turn, at the basis of the proof of Theorem \ref{th:rec}, besides
the so called \emph{Peres-Winkler censoring inequality} (see
\cite{cf:noteperes} and \cite{cf:MT}*{Section~2.4}), there were two
key equilibrium estimates on the behavior of (very) low temperature
Ising interfaces which we now recall and which were the responsible for
both the various $\nep{-L^\gep}$ error terms in $\d_1,\d_2$ and the constraint $\beta \gg 1$ on the inverse-temperature. The latter was necessary since the techniques of \cite{cf:MT} were based on several results of \cite{cf:DKS} on the Wulff construction which in turn use in an essential way low temperature cluster expansion.

\subsection{Equilibrium bounds on low temperature Ising interfaces used in~\cite{cf:MT}}
The first estimate is the key for the proof of the first part of the inductive statement namely $ \cA(L,t,\delta) \imp \cB(L,t_1,\delta_1)$. Given the rectangle $Q_L$ write it as the union of two overlapping rectangles, each of which is a suitable vertical  translate of the rectangle $R_L$ (see Figure~\ref{fig:Q}).
\begin{figure}
\centering
\includegraphics[width=0.8\textwidth]{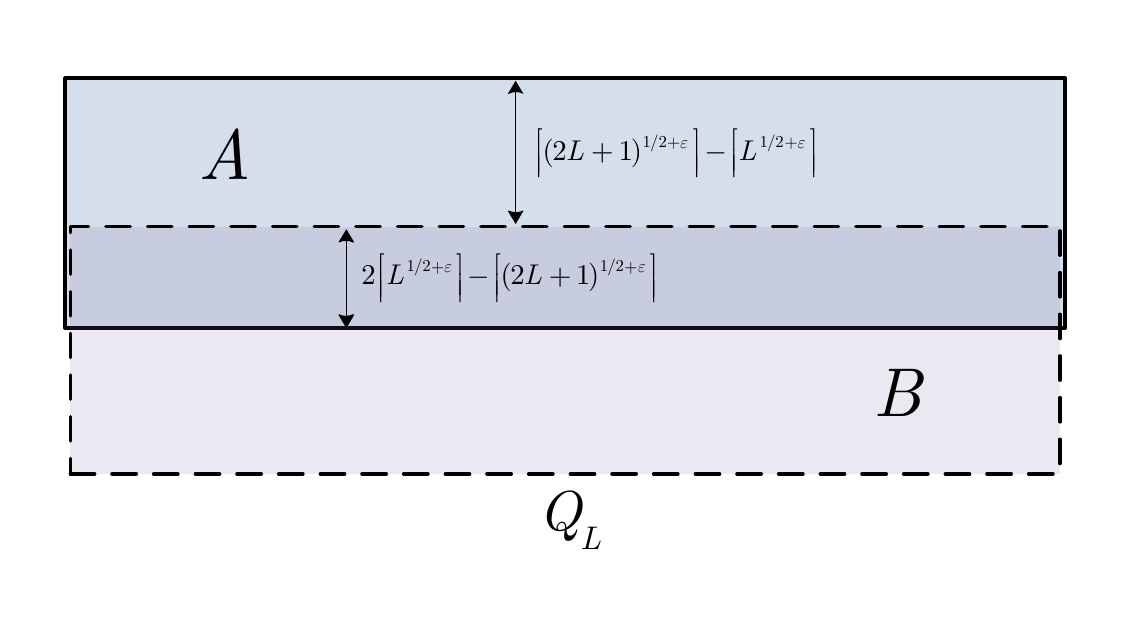}
\vspace{-0.5in}
\caption{The box $Q_L$ and its
covering with the rectangles $A,B$.}
\label{fig:Q}
\end{figure}
Call $B$ the lowest rectangle and $A$ the highest one. Then
\begin{Lemma}[see Claim 3.6 in \cite{cf:MT}]
\label{claim1}
There exists $c=c(\beta,\gep)>0$ such that
\begin{eqnarray}
\label{eq:gamma3}
\sum_{x\in B^c}  \bE\(\pi^\t(\s_x=+)-\pi^{\t,-}(\s_x=+)\)\le \nep{-cL^{2\gep}}
\end{eqnarray}
where $\pi^{\t,-}$ denotes the Gibbs measure in $Q_L$ with minus boundary conditions on its lowest side and $\t$ on the other three sides.
\end{Lemma}
In turn, by suitably playing with monotonicity properties of the measure $\pi$ as a function of the boundary conditions (see the short discussion in the proof of Claim 3.6 in \cite{cf:MT}),  the proof of the Lemma can be reduced to establishing the following bound.

Consider the enlarged rectangle $E_L$ with sides $3L$ and $2 \lceil (2L+1)^{\frac 12+\gep}\rceil$ respectively, which can be viewed as consisting of six rectangles $Q_L$ stacked together. Let $\pi^{(-,-,+,-)}$ be the associated Gibbs measure with $(-)$ boundary conditions on the North, East and West sides and $(+)$ on the South side. For any spin configuration $\s\in \{-1,+1\}^{E_L}$ let $\g=\g(\s)$ denote the unique open contour corresponding to these boundary conditions. Then
  \begin{Lemma} For any $\beta$ large enough there exists $c=c(\beta,\gep)$ such that for any $L$
    \label{lem-geo1}
    \begin{equation}
      \label{eq:rec3}
      \pi^{(-,-,+,-)}\Bigl(\g \text{ reaches height }  L^{\frac 12+\gep}\Bigr)\le \nep{-cL^{2\gep}}
    \end{equation}
  \end{Lemma}
Notice that the height $L^{\frac 12+\gep}$ is well beyond the typical
$O(\sqrt{L})$ fluctuations of the interface.

The second equilibrium bound is required for the proof of the statement $ \cB(L,t_1,\delta_1) \imp \cA(2L+1,t_2,\delta_2)$ (see Section~3.2 and in particular Claim~3.10 in~\cite{cf:MT}). Here the bottom line is the following bound.

Let $\bar R_L$ consists of two copies of $R_L$ stacked one on top of
the other. Let $\D\sset \partial \bar R_L$ consists of those boundary
sites $x=(i,j)$ in the South border such that $|i-L/2|\le \frac 12
L^{3\gep}$ and $j=0$. Consider the Gibbs
measure $\pi_{\bar R_L}^{(-,+,\D)}$ on $\bar R_L$ with $(-)$ boundary
conditions on the union of the North boundary and $\D$ and $(+)$ on the rest of $\partial \bar R_L$.  Let $\G_1$ be the event that the open contour $\g_1$ starting on the upper left corner of $\bar R_L$  ends at the left end of the interval $\D$ without ever crossing the vertical  line at $i=L/2$. Then
\begin{Lemma}
  \label{lem-geo2}
For any $\beta$ large enough there exists $c=c(\beta,\gep)$ such that for any $L$
\begin{equation}
  \label{eq:5}
  \pi_{\bar R_L}^{(-,+,\D)}\bigl(\G^c_1\bigr) \le \nep{-cL^{3\gep}}
\end{equation}
\end{Lemma}
In the scheme envisaged in \cite{cf:MT} the role played by the tiny extra piece of $(+)$ boundary conditions at
the vertices of $\D$, being the main source of the $\nep{c_3L^{3\gep}}$
factor relating the time scales $t_2,t_1$ in \eqref{eq:rec2}, is quite crucial and therefore it needs a bit of
explanation.

\begin{figure}
\centering
\includegraphics[width=0.49\textwidth]{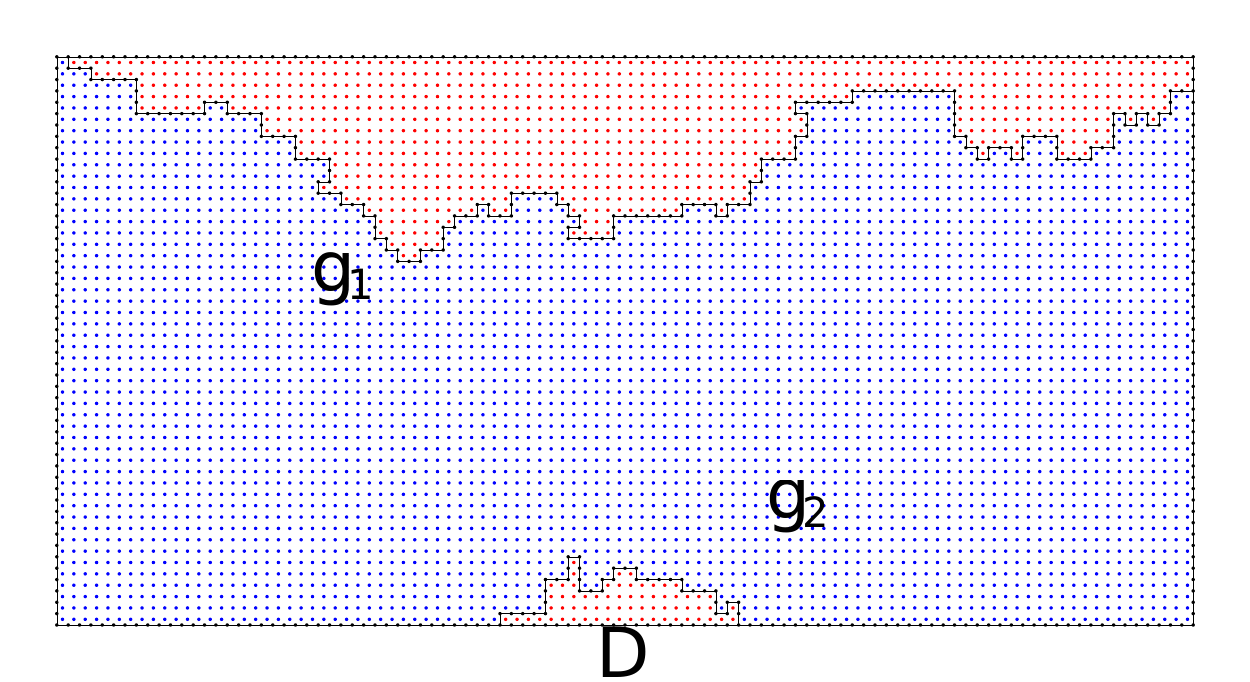}
\includegraphics[width=0.49\textwidth]{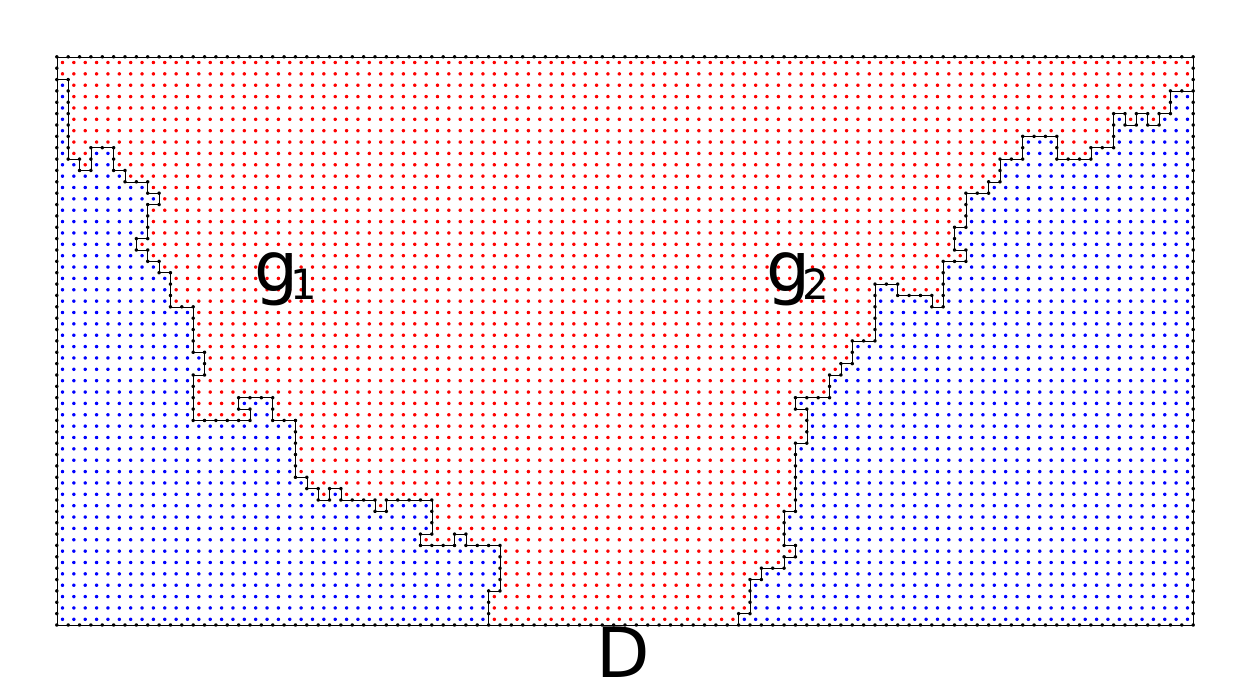}
\caption{Two possible arrangements for the open Peierls contours in the Gibbs measure under a $(-,+,\D)$ boundary condition.}
\label{fig:eq-cont}
\vspace{-0.1in}
\end{figure}

Let us first explain why the length of $\D$ was chosen equal to
$L^{3\gep}$. Under the boundary conditions $(-,+,\D)$,
for any configuration there exist exactly two open
Peierls contours $\g_1,\g_2$ with two possible
scenarios for their endpoints (illustrated in Figure~\ref{fig:eq-cont}):
\begin{enumerate}[(a)]
\item $\g_1$ joins the two upper corners of $\bar R_L$ and $\g_2$ the two ends of the interval $\D$;
\item $\g_1$ joins the left upper corner of $\bar R_L$ with the left
  boundary of $\D$ whereas $\g_2$ joins the right upper corner of $\bar R_L$ with the right boundary of $\D$.
\end{enumerate}

In~\cite{cf:MT} it was shown, using a significant part of the main
machinery of~\cite{cf:DKS}, that the ratio between the probabilities of the two
cases is roughly of the form $\nep{-\b(L + |\D|-2 D)\t_\b(0)}
$ where $D$ is the Euclidean distance between the left upper corner of $\bar R_L$
and the left boundary of $\D$.
Clearly $D\approx L/2-|\D| + O(L^{2\gep})$ and therefore case (b) is much more likely
than case (a) iff $|\D|\gg L^{2\gep}$. The choice $L^{3\gep}$ was
clearly not optimal and just a very safe one. Once the
first scenario can be neglected
then the fact that $\g_1$ does not intersect the vertical line at
$i=L/2$ is quite natural (but painful to prove).

Next we sketchily explain why the need of attracting the contour $\g_1$
deep down inside the rectangle $\bar R_L$.

When proving the
implication $\cB(L,t_1,\delta_1)\imp \cA(2L+1,t_2,\delta_2)$ we can
imagine that the rectangle $R_{2L+1}$ is written as the union of three
copies of the rectangle $Q_L$ denoted by $Q_L^{\text{centr}}, Q_L^{\text {left}},
Q_L^{\text{right}}$ (see Figure~\ref{fig:R}).

\begin{figure}
\centering
\includegraphics[width=\textwidth]{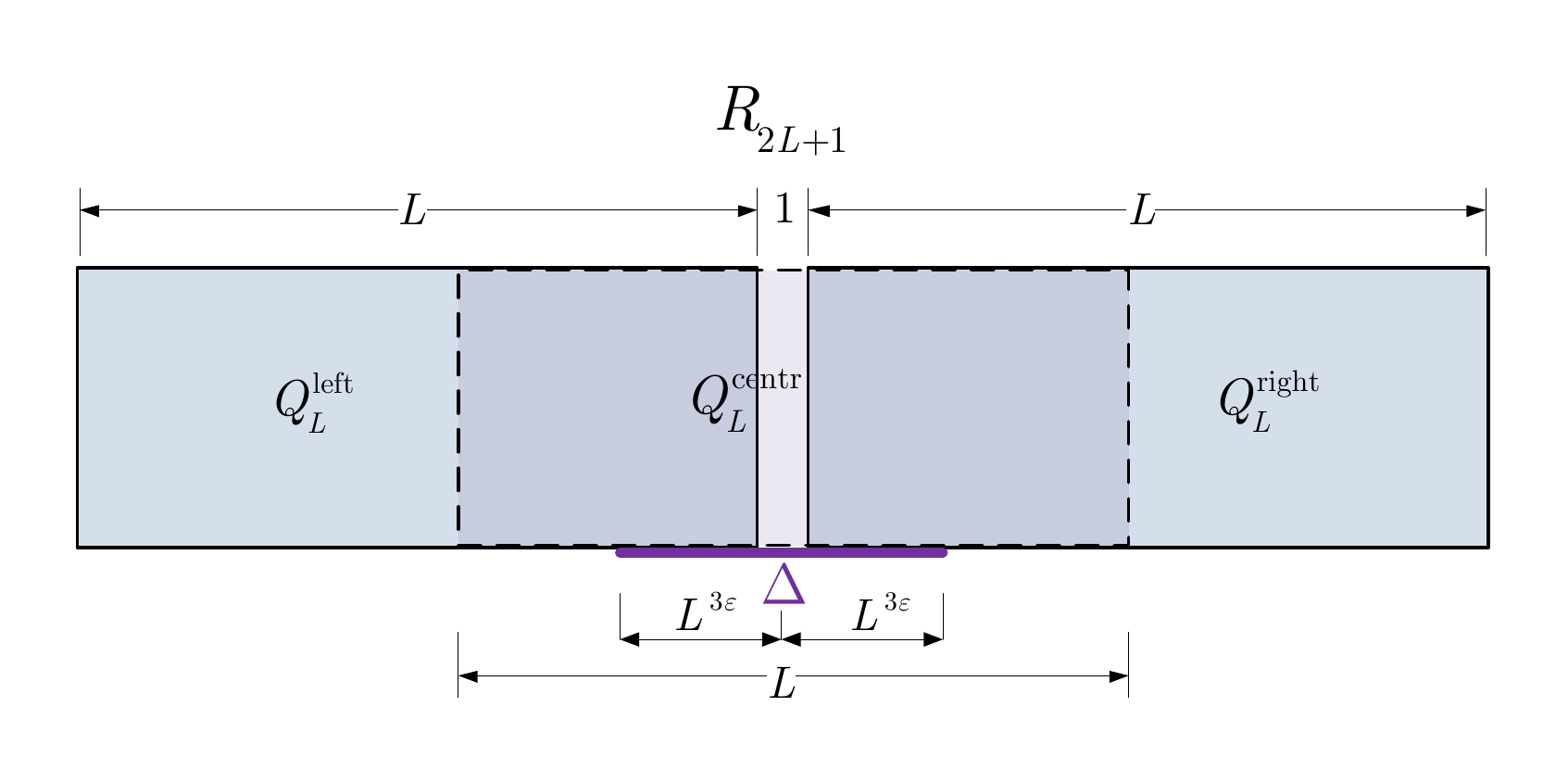}
\vspace{-0.5in}
\caption{The box $R_{2L+1}$ and its
covering with $Q_L^{\text{centr}}, Q_L^{\text {left}}, Q_L^{\text{right}}$. In bold the exceptional set $\Delta$.}
\label{fig:R}
\vspace{-0.1in}
\end{figure}

For simplicity suppose that the boundary conditions around $R_{2L+1}$
are the ``extreme ones'' namely $(-,-,+,-)$  ordered clockwise starting
from the North one and imagine starting the
dynamics from all pluses.

The Peres-Winkler results allow us to e.g.\ first
run the dynamics in the central rectangle $Q_L^{\text{centr}}$ for a
time $t_1$ and then in the left and right ones for some other time lag. Thus the dynamics in
$Q_L^{\text{centr}}$ runs with b.c.\
$(-,+,+,+)$ and after a time lag $t_1$ it will be close to the Gibbs measure
$\pi^{(-,+,+,+)}_{Q_L^{\text{centr}}}$ by less than $\d_1$ because of $\cB(L,t_1,\delta_1)$. The trouble
is that the marginal of this measure on e.g.\ the East boundary of
$Q_L^{\text{left}}$ is \emph{not} dominated by $\pi_\infty^-$ because
the unique open contour joining the left upper corner of
$Q_L^{\text{centr}} $ to the right one will stay  close to the upper side of
$Q_L^{\text{centr}}$. Therefore we cannot use statement
$\cB(L,t_1,\delta_1)$ for the dynamics in $Q_L^{\text{left}}$ to force
equilibrium there in another time lag $t_1$.

An appealing and very intuitive possible way out of this serious problem
would be to run many times the dynamics in $Q_L^{\text{centr}}$
until a large deviation forces the open contour to go below and to the
left of the East side of $Q_L^{\text{left}}$. Since the probability of
this fluctuation is $O(\exp(-c L^{2\gep}))$ it would be enough to wait
$O(\exp(c L^{2\gep})t_1)$ runs. However a rigorous implementation of
this idea is far from trivial and in \cite{cf:MT} the solution was
another one, less natural but much easier to carry out.

If one, by brute force, flips the boundary conditions inside the interval $\D$
on the South side of $Q_L^{\text{centr}}$ to $(-)$ the mixing time of the dynamics cannot change by more that
$\exp(c(\b)|\D|)$ (see \cite{cf:MT}*{Section 2.5} for more
details). Once the boundary conditions have been flipped then, thanks
to  \eqref{eq:5}, the contours in $Q_L^{\text{centr}}$ will follow
scenario (b) above and the resulting distribution over the East
boundary of $Q_L^{\text{left}}$ will now be dominated by the minus
phase $\pi_\infty^-$ allowing another application of the inductive statement
$\cB(L,t_1,\delta_1)$ to $Q_L^{\text{left}}$ and $Q_L^{\text{right}}$.

\section{A new recursive scheme}\label{sec:new-induction}
In this section we modify the recursion scheme of \cite{cf:MT} and,
modulo two equilibrium estimates very similar to Lemma~\ref{lem-geo1} and~\ref{lem-geo2}, we prove Theorem~\ref{th:quadrato}. We begin by fixing
some notation.

Let $N\in \mathbb N$ be a large integer, let $L=L_N= 2^N-1$
and choose $N_0$ to be the smallest integer such that
$L_{N_0}:=2^{N_0}-1\ge \inte{\log(L)^3}$. In our recursion $N_0$ and $N$ will
represent the
initial and final scales respectively.  To any intermediate scale
$n\in [N_0,N]$ we associate a length scale $L_n=2^n-1$.
We also define the rectangles $R_n,Q_n$ to have sides
(parallel to the coordinate axes) of length
$(L_n, \k_N \sqrt{L_n})$ and $(L_n, \k_N \sqrt{L_{n+1}})$
respectively where $\k_N=\sqrt{\k N}= O(\log(L)^{1/2})$ and $\k$ is a positive constant that later will
be chosen large enough depending on $\b$. Thus
the very definition of the rectangles depends on the final
scale. It is worth noticing
that $L_n\gg \k_N \sqrt{L_{n+1}}$ for any $n\in [N_0,N]$. Finally, for any $n\in [N_0,N]$, we define the statements $\cA(L_n, t, \d)$
and $\cB(L_n,t,\d)$ as in Definition~\ref{def-stat}.

Having fixed the basic notation our inductive scheme can be formulated
as follows. We repeat the result on the starting point for
completeness, despite it being completely obvious after Proposition~\ref{th:raf} and the remark after it.

\begin{Proposition}[The starting point]
\label{th:rafnew}
  For every $\beta$ there exists $c=c(\beta)$ such that for every
$n\in [N_o,N]$
the statements
$
\cA(L_n,t,e^{-t\,e^{-c \k_N\sqrt{L_n}}})
$
and
$
\cB(L_n,t,e^{-t\,e^{-c \k_N\sqrt{L_n}}})
$
hold.
\end{Proposition}

\begin{Theorem}[The inductive step]
\label{th:recnew}
There exist constants $c_1,c_2,c_3$  and for every
$\beta>\beta_c$ there exists $\k_0$ such that for any $\k\ge \k_0$,
for any $N$ large enough and for any $n\in [N_0,N]$,
\begin{equation}
  \label{eq:rec1b}
  \cA(L_n,t_n,\delta_n) \imp \cB(L_n,t'_n,\delta'_n)\imp \cA(L_{n+1},t_{n+1},\delta_{n+1})
\end{equation}
where
\begin{align}
  \label{eq:rec2b}
\delta'_n&=c_1\left(\delta_n+\nep{-c_2 \k_N^2 }+L_n^2\nep{-c_2 \log t_n}\right) & &; &  t_n'&= 2t_n \\
\d_{n+1}&= c_3(\delta_n + \nep{-c_2 \k_N^2 }
+ L_n^2\nep{-c_2 \log t_n} )& &; &  t_{n+1}&= \nep{c_3 \k_N^2 }t_n
\end{align}
\end{Theorem}

\begin{Corollary}[Solving for the final scale]
\label{th:solving}
In the same setting of Theorem \ref{th:recnew} there exists $c>0$ such
that, if $t_N:=\nep{c\, \k N^2}$ and $\d_N := c\,\nep{- c^{-1}\k N}$, then
for any $N\in \bbN$ large enough
statement $\cA(L_N,t_N,\d_N)$ holds.
\end{Corollary}
\begin{proof}[Proof of the Corollary]
Choose $t_{N_0}= \nep{c'\, \k N^2}$ for some $c'>0$.  Thanks to Proposition~\ref{th:rafnew}, for any $\b\ge 0$ it is possible to choose $c'=c'(\b)$ in
such a way that
$\mathcal A(L_{N_0},t_{N_0},\delta_{N_0})$ holds with
$\d_{N_0}=\nep{-c' \k N^2/2}$. Theorem~\ref{th:recnew}
immediately implies (use $t_n\ge t_{N_0}$) that $t_N\le \nep{c\, \k N^2}$ and $\d_N\le
c\nep{-c^{-1} \k N}$ for some other constant $c$.
\end{proof}
Once Corollary~\ref{th:solving} is proved, Theorem~\ref{th:quadrato} and its
corollaries (Corollaries~\ref{maincor-allplus} and \ref{maincor-autocor}) follow by exactly
the same arguments envisaged in \cite{cf:MT} for the analogous
results.

In turn the proof of Theorem~\ref{th:recnew} follows step by step the
proof of Theorem~\ref{th:rec} in \cite{cf:MT} once we assume two key
bounds on Ising interfaces that we state below.

\subsection{Two key equilibrium estimates for the new recursion}

Consider a rectangle with boundary conditions that are identically $(+)$ on the South boundary
and $(-)$ elsewhere. The following proposition
addresses a large deviation estimate for the vertical fluctuations of the unique open contour
in this setting, as illustrated in Figure~\ref{fig:rect-dev}.
\begin{figure}
\centering
\hspace{-0.1in}\includegraphics[width=0.8\textwidth]{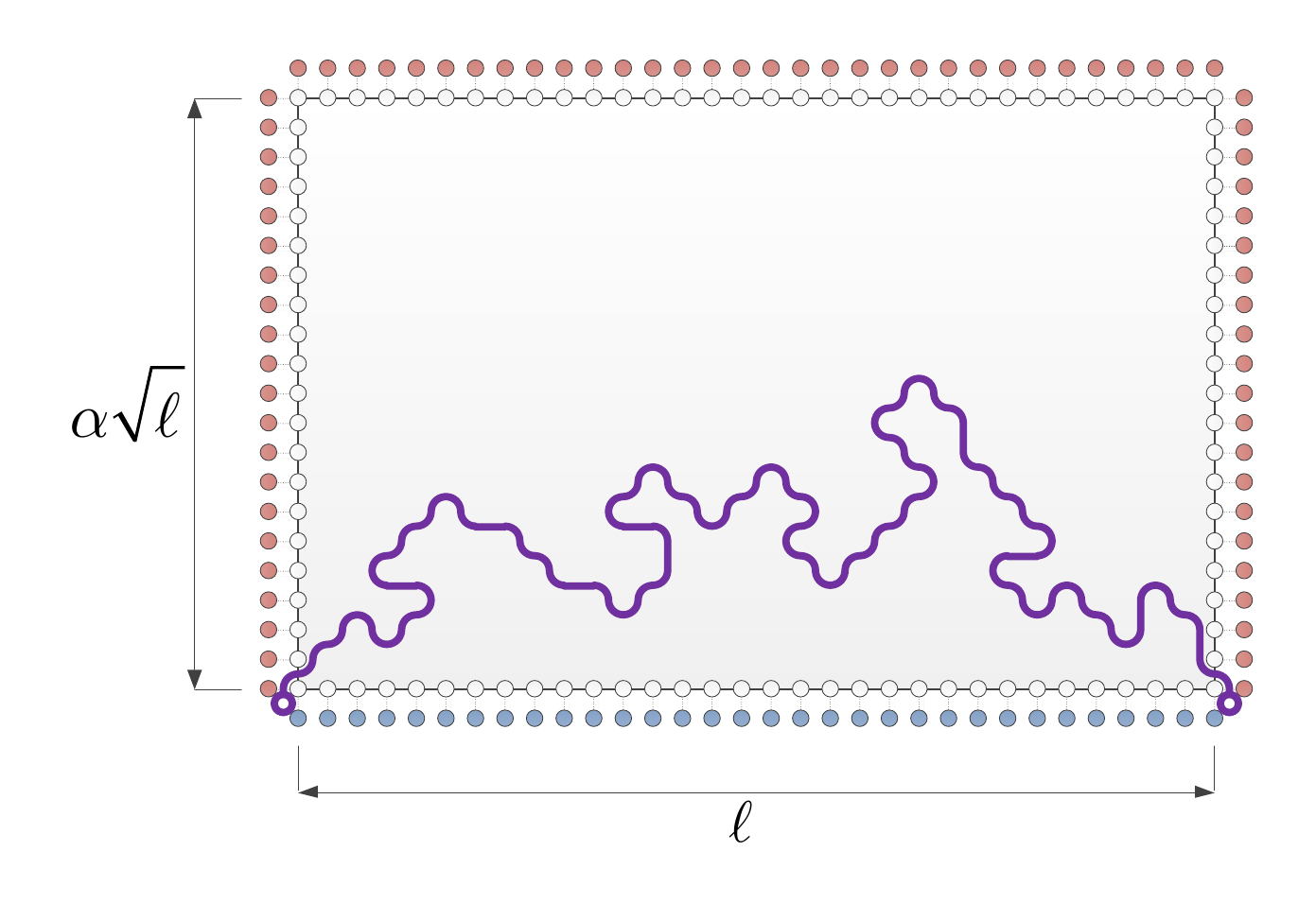}
\vspace{-0.3in}
\caption{Large deviation estimate for vertical fluctuations of the open Peierls contour under $(-,-,+,-)$ b.c., established in Proposition~\ref{prop:Bound1}.}
\label{fig:rect-dev}
\end{figure}

\begin{Proposition}
\label{prop:Bound1}
Let $R$ be a rectangle of dimensions $\ell\times \a\sqrt{\ell}$ with $1<\a\leq\sqrt{\ell}$ and let
$\pi_R^{(-,-,+,-)}$ be the corresponding Ising Gibbs measure with
$(-,-,+,-)$ ordered clockwise starting from the North side. Let $\l=\l(\s)$ denote
the unique open Peierls contour of the spin configuration $\s\in\O_R$.
Then for any $\beta>\beta_c$ there exist constants
$c_1,c_2>0$ depending only on $\b$ such that for any $0<\delta<1$ and $\a,\ell$ as above
    \begin{equation}
      \pi_R^{(-,-,+,-)}\Bigl(\s:\ \l(\s) \text{ reaches height }  \d \a
      \sqrt{\ell}\Bigr)\le \ell^{\,c_1}\nep{-c_2 (\d\a)^2}\,.
    \label{eq:B1}
\end{equation}
  \end{Proposition}
  \begin{remark*}
In the proof of the statement $\cA(L_n,t_n,\d_n)\imp \cB(L_n,t'_n,\d'_n)$ for $n\in
  [N_0,N]$, the above
  proposition is used with $\ell = 3L_n$, $\a= 3\sqrt{2}\k_N$ and $\d
  = (3\sqrt{2})^{-1}$. Thus for $\k$ large enough depending on
  $\b$ and for every $n\in [N_0,N]$ the r.h.s.\ of~\eqref{eq:B1} is quite small.
  \end{remark*}

The second equilibrium bound that is needed can be formulated as
follows. Mark the rectangle $R$ as given above by the corners $(x,y,y',x')$ clockwise starting from the Northwest corner.
Consider the Ising Gibbs measure on $R$ with the following b.c.:
\begin{enumerate}[(i)]
\item $-1$ on the North boundary and on an interval $\D$
of length $s\a^2$ belonging to the South boundary and centered around
its midpoint;
\item $+1$ elsewhere.
\end{enumerate}
We refer to these boundary conditions as the b.c.\ $(-,+,\D)$ and let $u$ and $v$ denote
the West and East endpoints of the interval $\Delta$ centered on the South border.

In this new setting we aim to show that w.h.p.\ in the random-line representation there are two open contours $\lambda_1$ and $\lambda_2$ with $\delta\lambda_1=\{ x,  u\}$ and $\delta \lambda_2=\{y, v\}$ and such that $\lambda_1$ (resp.\ $\lambda_2$) lies entirely in the left (resp.\ right) half of $R$, as shown in Figure~\ref{fig:rect-equil2} (recall the discussion following Lemma~\ref{lem-geo2} for the role of this event in the inductive scheme). This is established by the next proposition.

\begin{figure}
\centering
\hspace{-0.1in}\includegraphics[width=0.9\textwidth]{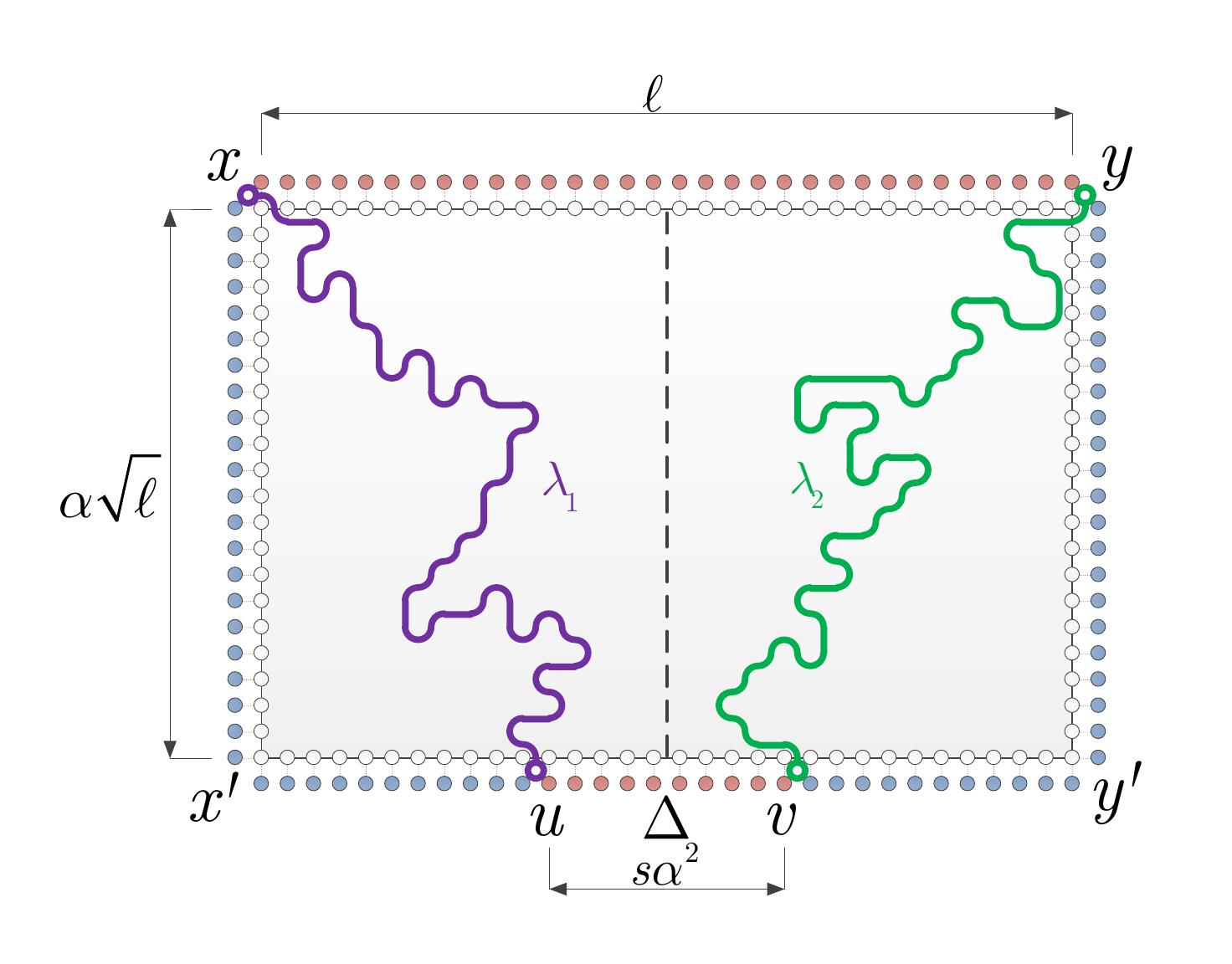}
\vspace{-0.3in}
\caption{Open Peierls contours confined to the left and right halves of the rectangle $R$ under $(-,+,\D)$ b.c., addressed by the equilibrium estimate of Proposition~\ref{prop:Bound2}.}
\label{fig:rect-equil2}
\end{figure}

\begin{Proposition}
\label{prop:Bound2}
For any $\b>\b_c$ there exist $c_1,c_2,s_0>0$ depending only on $\beta$ so that the following holds.
Let $R$ be a rectangle of size $\ell\times \a\sqrt{\ell}$ with $1<\a < (1/s_0)\sqrt{\ell}$ and let
$\D=[u,v]$ be an interval of length $s \a^2$ centered on the South border for some $s\geq s_0$.
Let $\mathcal{V}$ denote the event that there are two open Peierls contours confined to the left and right halves of $R$ and connecting the top corners
with $u,v$. Then
\begin{equation}
  \label{eq:B2}
  \pi^{(-,+,\D)}(\mathcal{V}^c) \le \ell^{c_1} \nep{-c_2 \a^2}\,.
\end{equation}
\end{Proposition}
\begin{remark*}
In the proof of the statement $\cB(L_n,t'_n,\d'_n)\imp \cA(L_{n+1},t_{n+1},\d_{n+1})$ the above
proposition is invoked with a choice of $\ell= L_n$, $n\in [N_0,N]$, and $\a =
2\sqrt{2}\k_N$, so that $\alpha = o(\sqrt\ell)$; also,
the r.h.s.\ of \eqref{eq:B2} is always very
small provided that the constant $\k$ is chosen to be large enough.
\end{remark*}

\section{Equilibrium crossing probabilities for the infinite strip}\label{sec:strip}

In this section we study the behavior of the unique open contour in the infinite strip with boundary conditions $(+)$ in the upper half-plane and $(-)$ in the lower half-plane.
Deriving sharp estimates for the probability that this contour is confined to the upper half-plane, as well as a large deviation estimate for the its vertical fluctuations, will later serve as a key element in the proofs of Propositions~\ref{prop:Bound1} and~\ref{prop:Bound2}.
The analysis in this section hinges on the duality tools developed in~\cites{cf:PV1,cf:PV2}, which enable us to characterize the Ising interfaces for any $\b>\b_c$. By using this machinery together with some additional ideas we establish various properties of the contours, roughly analogous to Brownian bridges with logarithmic ``decorations''.

\begin{figure}
\centering
\includegraphics[width=0.6\textwidth]{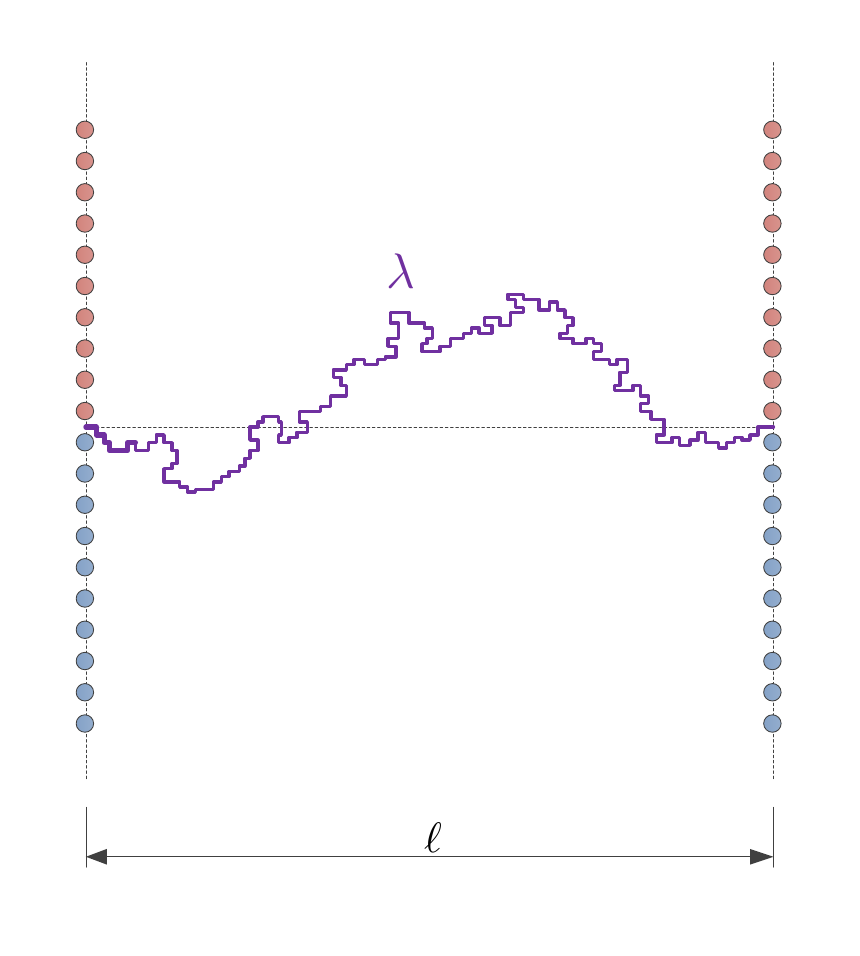}
\vspace{-0.3in}
\caption{Infinite strip with mixed b.c.\ and its unique open contour.}
\label{fig:infstrip}
\end{figure}

For $S \subset \Z^2$ define the boundary condition $\eta\in\{\pm1\}^{\Z^2\setminus S}$ to be
\begin{equation}
  \label{eq-plus-minus-bc-def}
\eta(x,y) = \left\{\begin{array}{ll}
-1 & y > 0\,,\\
+1 & y \leq 0\,.
\end{array}\right.
\end{equation}
We focus on the case where $S$ is the infinite-strip of width $\ell$,
\begin{equation}
  \label{eq-infinite-strip-def}
  S = \{1,\ldots,\ell\} \times \Z\,,
\end{equation}
whereby the above b.c.\ $\eta$ gives rise to a unique open contour $\lambda$ connecting the dual vertices $\{\frac12,\frac12\},\{\ell+\frac12,\frac12\}$ in $S^*$ (see Figure~\ref{fig:infstrip}).
Such contours have been intensively studied and the scaling limit of $\lambda$ is known to be the \oned Brownian bridge between these two points~\cite{cf:Hryniv} while our proof requires more quantitative estimates.  Tight large deviation estimates for vertical fluctuations of $\lambda$ are necessary in several places in our proof. This will be established by Theorem~\ref{thm-strip-large-dev} below (in a slightly more general setting) via an argument akin to those used for controlling the deviations of the Brownian bridge yet carried out within the duality framework of~\cites{cf:PV1,cf:PV2}.

Significantly more delicate is the crucial estimate of obtaining a lower bound on the probability that $\lambda$ is contained in the upper half-plane. The Brownian bridge heuristic suggests that this event holds with probability proportional to $1/\ell$, and as the following theorem confirms this is indeed the case.

\begin{Theorem}\label{thm:positive} Let $S$ be the infinite strip of width $\ell$ with b.c.\ $\eta$ as given in~\eqref{eq-plus-minus-bc-def},\eqref{eq-infinite-strip-def}. For an Ising configuration $\sigma$ on $S$ let $\lambda=\lambda(\sigma)$ be its unique open contour in the dual lattice $S^*$ (\ie $\d \l = \{(\frac12,\frac12),(\ell+\frac12,\frac12)\}$). For $i\in\Z$ let $H^*_i = \{\frac12,\ldots,\ell+\frac12\}\times\{i+\frac12\}$ be the vertices comprising the $i$-th horizontal level of $S^*$.  Then for every $\ell$, \[ \frac{c}\ell \leq \pi^\eta_S\left(\sigma :\, \lambda(\sigma) \mbox{ stays above } H^*_{-1}\right) \leq \frac{C}\ell\,,\] where $c,C>0$ are constants that depend only on $\beta$.  \end{Theorem}

As an immediate consequence we obtain the following lower bound on the spin-spin correlation at high temperature for two points on the horizontal boundary of the half-strip $ S^+=\{\frac12,\ldots,\ell+\frac12\}\times \bbZ^{+*}$, which to our knowledge was previously unknown.
\begin{Corollary}\label{cor-spin-spin-strip}
Let $u=(\frac12,\frac12)$ and $v=(\ell+\frac12,\frac12)$.
For every $\beta>\beta_c$
there exist constants $c,c'$ such that
\[
\frac{c}{\ell^{3/2}}e^{-\tau_\beta(u-v)}\le  \pi^*_{S^+}(\sigma_u \sigma_v)\le \frac{c'}{\ell^{3/2}}e^{-\tau_\beta(u-v)}\,.
\]
\end{Corollary}
 Note that the above corollary also extends to other geometries, for instance rectangles with a wide range of aspect ratios where $u,v$ correspond to the upper corners. We postpone the proof of Theorem~\ref{thm:positive} and Corollary~\ref{cor-spin-spin-strip} in order to first obtain several of the ingredients that it would require, the first of which being the aforementioned large deviation inequality for the open contour in the infinite strip $S$.
\begin{Theorem}\label{thm-strip-large-dev}
Let $\bar{S}=\bar{S}(a,b)$ be the infinite strip $\{1,\ldots,\ell\}\times \Z$ excluding the horizontal slits $\{1,\ldots, a\}\times\{0,1\}$ and $\{b,\ldots,\ell\}\times\{0,1\}$ for $0\leq a < b \leq \ell+1$ with b.c.\ $\eta$ as defined in~\eqref{eq-plus-minus-bc-def}. For an Ising configuration $\sigma$ on $\bar{S}$ let
$\lambda=\lambda(\sigma)$ be its unique open contour in the dual $\bar{S}^*$, and for $i\in\Z$ let $H^*_i = \{\frac12,\ldots,\ell+\frac12\}\times\{i+\frac12\}$. 
Then there exist some constant $C(\beta)>0$ such that for any $\ell$ the following holds:
\begin{align*}
\pi^\eta_{\bar{S}}\left(\sigma :\, \lambda(\sigma) \mbox{ reaches } H^*_{x\sqrt{\ell}}\right) &\leq C\exp\left(-\kappa_\beta\; x^2\right)&\mbox{for all $x \leq \sqrt{\ell}$}\,,\\
\pi^\eta_{\bar{S}}\Big(\sigma :\, \lambda(\sigma) \mbox{ reaches } H^*_{h}\Big) &\leq C \exp\left(-\kappa_\beta\; h\right)&\mbox{for all $h \geq \ell$}\,,
 \end{align*}
 where $\kappa_\beta>0$ is the constant in the sharp triangle inequality of the surface tension $\tau_\beta$.
\end{Theorem}
\begin{remark*}
It is fairly straightforward to establish upper bounds as above with an extra prefactor of order $|b-a|$
(see e.g.\ the first inequality in~\eqref{eq-middle-large-dev-rough}). Eliminating this spurious prefactor requires a
delicate multi-scale analysis.
\end{remark*}
\begin{proof}[Proof of theorem]
In what follows we will prove the following inequality, which is a stronger form of the required large deviation estimates: For some $C=C(\beta)>0$,
\begin{align}
  \label{eq-level-large-dev}
  \pi^\eta_{\bar{S}}\left(\sigma :\, \lambda(\sigma) \mbox{ reaches } H^*_{h}\right) &\leq C e^{-\kappa_\beta\big(\tfrac{h^2}{b-a-1} \;\wedge\; h\big)}\quad\mbox{ for any $h>0$}
\end{align}
(we may clearly assume that $b > a+1$ otherwise the unique open contour is trivial).
Indeed, the above probability estimate is clearly increasing in the value of $b-a$, which in turn is guaranteed to be at most $\ell+1$ (reflecting the bounds in the proposition). Notice that by choosing $C$ to be appropriately large we need only consider $h \geq \sqrt{b-a}$.

Fix some large cutoff height $n \geq (h\;\vee\;\ell)^2$ and let
\[\bar{S}_n =\bar{S} \cap (\Z \times \{-n,\ldots,n\})\]
be the strip $\bar{S}$ truncated at $\pm n$ with boundary conditions analogous to $\eta$, \ie negative on the upper half-plane and positive elsewhere. Due to the uniqueness of the Gibbs measure on $\bar{S}$, the probabilities we seek to bound are obtained as a limit of the corresponding ones for $\bar{S}_n$ as $n\to\infty$. Further let $u=(a+\frac12,\frac12)$ and $v=(b-\frac12,\frac12)$ denote the endpoints of the unique open contour in $\bar{S}_n^*$.
Define the \emph{height} of this open contour $\lambda\subset \bar{S}_n^*$ at the horizontal coordinate $x \in \{\frac12,\ldots,\ell+\frac12\}$ to be
\[ \htx(\lambda,x) = \max\{ y : (x,y)\in\lambda\}\,.\]
The main effort in the proof will be devoted to the analysis of the vertical fluctuations of the contour $\lambda$ within the inner strip with $x$-coordinates $\{a+\frac12,\ldots,b-\frac12\}$. It is the case that large vertical fluctuations in the margins (\ie large values of $\htx(\l,x)$ for $x< a$ or $x > b$) are far more unlikely and can be estimated via standard properties of the surface tension. To control the delicate fluctuations of $\htx(\l,x)$ for $a < x < b$ we will apply a multiscale approach, repeatedly bounding the deviations at the horizontal midpoints in a nested dyadic partition of the interval between $u$ and $v$.

The first step in the proof is to bound the event that the contour includes a given point $w=(x,h) \in \bar{S}^*$ in terms of its coordinates $h$ and $x$. First notice that by~\eqref{eq:8}, \begin{equation}\label{eq-large-dev-frac}
\pi^\eta_{\bar{S}_n}\left(\sigma :\, w\in\lambda(\sigma)\right) = \biggl[\sum_{\substack{\l:\, \d\l=\{u,v\} \\ w\in\l}} q_{\bar{S}_n^*}(\l)\biggr]\,\big/\, \biggl[\sum_{\substack{\l:\, \d\l=\{u,v\}}} q_{\bar{S}_n^*}(\l)\biggr]\,.
\end{equation}
Consider the numerator in the last expression: Corollary~\ref{lem-3-pt-spins} implies that
\[ \sum_{\substack{\lambda:\,\d\lambda=\{u,v\}\\ w\in\lambda}}
q_{\bar{S}_n^*}(\lambda)\leq \pi_{\bar{S}_n^*}^*\left(\s_u\,\s_w\right) \, \pi_{\bar{S }^*}^*\left(\s_v\,\s_w\right)\,,\]
and together with Lemma~\ref{lem-spin-spin} we deduce that for some $c_0=c_0(\beta)>0$
\begin{align}
  \sum_{\substack{\lambda:\,\d\lambda=\{u,v\}\\ w\in\lambda}}
q_{\bar{S}_n^*}(\lambda)&\leq
 \frac{c_0}{\sqrt{|u-w||v-w|}}\exp\big(-\tau_{\beta}(u-w)-\tau_{\beta}(v-w)\big)\,. 
  \label{eq-large-dev-num}
\end{align}
To estimate the denominator in~\eqref{eq-large-dev-frac} recall Eq.~\eqref{eq:7} according to which
\[
\sum_{\substack{\l:\, \d\l=\{u,v\}}} q_{\bar{S}_n^*}(\l) = \pi_{\bar{S}_n^*}^*(\s_u\,\s_v)\,.\]
As it follows from GKS that decreasing our domain can only decrease the spin-spin correlations,
letting $S_n = \{a,\ldots,b\}\times \{-n,\ldots,n\}$ (\ie $S_n$ is the result of ``pushing'' the West and East boundaries of $\bar{S}_n$ to $a$ and $b$ resp.) we have
\[ \pi_{\bar{S}_n^*}^*(\s_u\,\s_v) \geq \pi_{S_n^*}^*(\s_u\,\s_v)\,.\]
By Eq.~\eqref{eq-strip-two-pt-corr} there exists some $c_1=c_1(\beta)> 0$ such that
the spin-spin correlation between $u,v$ in the dual to the infinite strip $S=\{a,\ldots,b\}\times \Z$ is
\[
\pi_{S^*}^*(\s_u\,\s_v) = \frac{c_1+o(1)}{\sqrt{|u-v|}}\exp(-\tau_{\beta}(|u-v|)) \,,
\]
where the $o(1)$-term tends to $0$ as $|u-v|\to\infty$.
Due to the strong spatial mixing properties of the high temperature region $\beta^* < \beta_c$, the value of $\pi_{S_n^*}^*(\s_u\,\s_v)$ converges to the above r.h.s.\ exponentially fast in $n$. Already for $n\geq \ell^2$ we could absorb the error in the constant $c_1$ and obtain that for some $c'_1(\beta)>0$,
\begin{equation}\label{eq-large-dev-denom}\sum_{\substack{\l:\, \d\l=\{u,v\}}} q_{\bar{S}_n^*}(\l) \geq \frac{c'_1}{\sqrt{|u-v|}}\exp(-\tau_{\beta}(|u-v|)) \,.\end{equation}
By combining~\eqref{eq-large-dev-frac} with \eqref{eq-large-dev-num} and \eqref{eq-large-dev-denom} we conclude that for some $c_2 =c_2(\beta)>0$,
\begin{equation}
  \label{eq-w-in-contour-gen}
 \pi^\eta_{\bar{S}_n}\left(\sigma :\, w\in\lambda(\sigma)\right) \leq
\frac{c_2\sqrt{|u-v|}}{\sqrt{|u-w||v-w|}}
\exp\left(-\tau_{\beta}(u-w)-\tau_{\beta}(v-w)+\tau_{\beta}(u-v)\right)\,.
\end{equation}
At the same time, by the sharp triangle inequality property~\eqref{eq-sharp-tri-ineq} of the surface tension,
 \begin{align}
 \tau_{\beta}(u-w)&+\tau_{\beta}(v-w) \geq \tau_{\beta}(u-v) + \kappa_\beta \left(|w-v|+|u-w|-|u-v|\right)\,.\label{eq-contours-sharp-tri}
 \end{align}
Recalling that $w$ is at height $h$ it is easy to verify that
 \[ |w-v|+|u-w|-|u-v| \geq \frac{4h^2}{\sqrt{|u-v|^2+4h^2}+|u-v|}\,.\]
Set $\xi = \frac4{1+\sqrt{5}} > \frac65 $ and now observe that whenever $h^2 \leq |u-v|^2$ the last expression is at least $\xi\tfrac{h^2}{|u-v|}$ and otherwise it is at least
$ \xi h$. Using this bound for the r.h.s.\ of Eq.~\eqref{eq-contours-sharp-tri} now allows us to produce the following bound out of Eq.~\eqref{eq-w-in-contour-gen}:
\begin{equation}
  \label{eq-w-in-contour}
\pi^\eta_{\bar{S}_n}\left(\sigma :\, w\in\lambda(\sigma)\right) \leq
\frac{c_2\sqrt{|u-v|}}{\sqrt{|u-w||v-w|}} \exp\left(-\tfrac65 \kappa_\beta \left(\tfrac{h^2}{|u-v|}\;\wedge\; h\right)\right)\,.
\end{equation}

 Straightforward applications of the above bounds will now yield the required bounds on the height of $\l$ along the margins $x \leq a+\frac12$ and $x \geq b-\frac12$ as well as whenever $b-a$ is a uniformly bounded. Indeed, by symmetry we may assume without loss of generality that $x \leq a+\frac12$ and note that in this case $w=(x,h)$ satisfies $|w-v| \geq |u-v|$. Applying~\eqref{eq-w-in-contour-gen} combined with the sharp triangle inequality as in Eq.~\eqref{eq-contours-sharp-tri} we get
\[ \pi^\eta_{\bar{S}_n}\left(\sigma :\, w\in\lambda(\sigma)\right) \leq
\frac{c_2}{\sqrt{|u-w|}}\exp\left(-\kappa_\beta|u-w|\right)\,.\]
Summing the last expression over all $w=(x,y)$ with $x\leq a+\frac12$ and $y\geq h$ gives that
\begin{equation}
   \label{eq-margin-large-dev1}
   \pi^\eta_{\bar{S}_n}\big(\sigma :\, \htx(\lambda(\sigma),x)\geq h\mbox{ for some $x\leq a+\tfrac12$}\big) \leq C_1\, e^{-\kappa_\b h}
 \end{equation}
for some $C_1=C_1(\beta)>0$, and analogously
\begin{equation}
   \label{eq-margin-large-dev2}
   \pi^\eta_{\bar{S}_n}\big(\sigma :\, \htx(\lambda(\sigma),x)\geq h\mbox{ for some $x\geq b-\tfrac12$}\big) \leq C_1\, e^{-\kappa_\b h}\,.
 \end{equation}
We now turn our attention to the main task of bounding the vertical fluctuations of $\lambda$ along the interval $(a+\frac12,b-\frac12)$.
First observe that \eqref{eq-w-in-contour} immediately provides the bound we seek (Eq.~\eqref{eq-level-large-dev}) in the special case where $|u-v|= O(1)$ (with an implicit constant that may depend on $\beta$): In that case a simple union bound over $w=(x,h)$ for $x\in(a+\frac12,b-\frac12)$ yields
\begin{align}
   \pi^\eta_{\bar{S}_n}\big(\sigma :\, &\htx(\lambda(\sigma),x)\geq h\mbox{ for some $a+\tfrac12< x < b-\tfrac12$}\big) \nonumber\\ & \leq c_2 |u-v|  \exp\left(-\tfrac65 \kappa_\beta \left(\tfrac{h^2}{|u-v|}\;\wedge\; h\right)\right) \leq
   C_2\, \exp\left(-\kappa_\beta \left(\tfrac{h^2}{|u-v|}\;\wedge\; h\right)\right)\,,
      \label{eq-middle-large-dev-rough}
 \end{align}
where $C_2(\beta)>0$ incorporates the uniform bound on $|u-v|$. Combined with~\eqref{eq-margin-large-dev1} and~\eqref{eq-margin-large-dev2} this concludes the bound in Eq.~\eqref{eq-level-large-dev} when $|u-v|=O(1)$.

Let $M\geq 2$ be some fixed integer whose value will depend only on $\beta$ and will be specified later. Justified by the above argument, assume without loss of generality that
\begin{equation}
  \label{eq-u-v-large-enough}
 |u-v| \geq M^2\quad\mbox{ and }\quad \exp\Big(\tfrac1{10}\kappa_\beta|u-v|^{1/4}\Big)\geq |u-v|\,.
\end{equation}
We claim that this in turn narrows our attention to proving Eq.~\eqref{eq-level-large-dev} for $h$ satisfying
\begin{equation}
  \label{eq-h-region}
 M |u-v| \leq h^2 \leq \tfrac12 |u-v|^{5/4}\,.
\end{equation}
To see this recall first that the lower bound on $h$ is justified by selecting a suitably large constant $C(\beta)$ in Eq.~\eqref{eq-level-large-dev}. For the upper bound, note that if $h^2 > \frac12|u-v|^{5/4}$ (in which case $\tfrac{h^2}{|u-v|} > \frac12 |u-v|^{1/4}$ whereas $h > \frac1{\sqrt{2}}|u-v|^{5/8}$) then~\eqref{eq-u-v-large-enough} implies that
$|u-v|$ is at most $\exp\left(\frac{\kappa_\b}{5} \big(\tfrac{h^2}{|u-v|}\;\wedge\; h\big)\right)$ and hence Eq.~\eqref{eq-level-large-dev} follows from a union bound over $x\in(a+\frac12,b-\frac12)$ as in~\eqref{eq-middle-large-dev-rough}.

Consider the event whereby the contour $\lambda$ visits a point $w\in \bar{S}_n$ given by
\[w=(x,y)\mbox{ for some }x\in\left(a+\tfrac1M|u-v|, b-\tfrac1M|u-v|\right)\mbox{ and }y\geq h\,.\]
Clearly $\sqrt{|u-w||v-w|} \geq \tfrac1M|u-v|$ thus we can rewrite~\eqref{eq-w-in-contour} as
\[ \pi^\eta_{\bar{S}_n}\left(\sigma :\, w\in\lambda(\sigma)\right) \leq
\frac{c_3}{\sqrt{|u-v|}} \exp\left(-\tfrac65 \kappa_\beta \left(\tfrac{y^2}{|u-v|}\;\wedge\; y\right)\right)\,,
\]
where $c_3>0$ depends only on $\beta$. Summing over all possible values of $y\geq h$ we now obtain that
\begin{align}
\pi^\eta_{\bar{S}_n}&\Big(\sigma :\, \htx\left(\l(\sigma),x\right)\geq h\Big) \leq
\frac{c_3}{\sqrt{|u-v|}}\bigg(\sum_{y =h}^{|u-v|} e^{-\tfrac65\kappa_\beta\, \tfrac{y^2}{|u-v|}}+
\sum_{y \geq (h\;\vee\;|u-v|)}\!\! e^{-\tfrac65\kappa_\beta\, y}\bigg)
\nonumber\\
&\leq
c_3 \sum_{z =h/\sqrt{|u-v|}}^{\sqrt{|u-v|}} e^{-\tfrac65\kappa_\beta z^2}
+ \frac{c_3}{\sqrt{|u-v|}} \sum_{y \geq h} e^{-\tfrac65\kappa_\beta\, y}
\leq c_3' e^{-\tfrac65\kappa_\beta \frac{h^2}{|u-v|}}
+ \frac{c'_3}{\sqrt{|u-v|}} e^{-\tfrac65\kappa_\beta h} \nonumber\\
&\leq C_3\, \exp\left(-\tfrac65 \kappa_\beta \left(\tfrac{h^2}{|u-v|}\;\wedge\; h\right)\right)\,,\label{eq-level-large-dev-mid}
 \end{align}
where the constant $C_3>0$ depends only on $\beta$.

We next wish to extend the above bound on $\htx(\l,x)$ to hold simultaneously for all $x\in(a+\frac12, b-\frac12 )$ by means of a dyadic partition of the interval between $u$ and $v$. Set
\[ K = \Big\lfloor \frac12\log_{M}|u-v|\Big\rfloor \]
and notice that~\eqref{eq-u-v-large-enough} ensures that $K \geq 1$. Define the following sequence of refinements of the interval between $u$ and $v$, indexed by $k = 0,\ldots,K$.
We begin with the trivial partition at level $0$,
\begin{align*}
 z_0^{(0)}=a+\tfrac12\,, &\quad z_1^{(0)}=b-\tfrac12\,,
\end{align*}
and refine level $k$ into level $k+1$ by subdividing each subinterval $(z_{j-1}^{(k)},z_{j}^{(k)})$ into $M$ equal parts (up to integer rounding):
\begin{align*}
  z_{Mj}^{(k+1)}&=z_{j}^{(k)}  &\mbox{ for }j=0,\ldots,M^k\,,\\ z_{M(j-1)+i}^{(k+1)}&=\left\lfloor z_{j-1}^{(k)} + \tfrac{i}M \big(z_{j-1}^{(k)}+z_{j}^{(k)} \big)\right\rfloor+\tfrac12 &\mbox{ for }i=1,\ldots,M-1\mbox{ and }j=1,\ldots,M^k\,.
\end{align*}
Observe that for all admissible $j,k$ we have
\[ M^{-k}|u-v| - 2 < z_{j}^{(k)}-z_{j-1}^{(k)} < M^{-k}|u-v| + 2\,,  \]
where the additive terms account for the rounding corrections along the refinements.
In particular, the expression in the lower bound on the sub-interval lengths satisfies
\[ M^{-k}|u-v| \geq M^{-K}|u-v| \geq \sqrt{|u-v|} > 10\]
(as $|u-v|$ is large enough). Next, define
\begin{align*}
 h_k = M^{-k/4} h &\quad\mbox{ for $k=0,\ldots,K$}\,,
\end{align*}
and let $\Upsilon_j^{(k)}$ be the event that the height of the contour at $z_j^{(k)}$ does not exceed $\sum_{i<k} h_i$:
\[ \Upsilon_{j}^{(k)} = \bigg\{ \sigma:\, \htx\big(\lambda(\sigma),z_j^{(k)}\big) < \sum_{i=0}^{k-1} h_i \bigg\}\mbox{ for $k \geq 1$ and $1 \leq j < M^k$}\,.\]
Recalling~\eqref{eq-level-large-dev-mid} and rewriting it in terms of
$\Upsilon_j^{(k)}$ and its complement $\overline{\Upsilon}_{j}^{(k)}$ we have that
\[ \pi^\eta_{\bar{S}_n}\left(\overline{\Upsilon}_j^{(1)}\right) \leq C_3\, \exp\bigg(-\tfrac65\kappa_\b \bigg(\frac{h_0^2}{z^{(0)}_1- z^{(0)}_0}\;\wedge\; h_0\bigg)\bigg)\mbox{ for $j=1,\ldots,M-1$}\,.
\]
Exactly the same argument yields that for general $k$, $1\leq j\leq M^{k}$ and $1\leq i \leq M-1$,
\begin{equation}
\pi^\eta_{\bar{S}_n}\left( \overline{\Upsilon}^{(k+1)}_{M(j-1)+i}\,,\,\Upsilon^{(k)}_{j-1}\,,\,
\Upsilon^{(k)}_{j} \right) \leq C_3\, \exp\bigg(-\tfrac65\kappa_\b \bigg(\frac{h_{k}^2}{z^{(k)}_{j}- z^{(k)}_{j-1}}\;\wedge\; h_{k}\bigg)\bigg)\,.
\label{eq-Upsilon-general}
\end{equation}
To estimate the last expression, observe that $h_{k}/(z_j^{(k)} - z_{j-1}^{(k)})$ increases with $k$ roughly as $M^{3k/4}$. More accurately,
\begin{align}
\frac{h_{k}}{z_j^{(k)} - z_{j-1}^{(k)}} \leq
  \frac{M^{\frac{k}4} h}{M^{-k}|u-v|-2} \leq
  \frac{M^{\frac{3k}4} h}{|u-v|} \Big(1+\frac2{M^{-k}|u-v|-2}\Big)<
  \frac54 M^{\frac{3k}4}
  \frac{h}{|u-v|}
  \label{eq-hk-z-upper-bound}
\end{align}
(where we used the fact that $M^{-k}|u-v|> 10$) and similarly
\begin{align}
\frac{h_{k}}{z_j^{(k)} - z_{j-1}^{(k)}} \geq
  \frac{M^{\frac{k}4} h}{M^{-k}|u-v|+2} >
\frac45 M^{\frac{3k}4}
  \frac{h}{|u-v|}\,.
  \label{eq-hk-z-lower-bound}
\end{align}
Our choice of $K$ and the upper bound~\eqref{eq-h-region} on $h$ enable us to derive from~\eqref{eq-hk-z-upper-bound} that for all $k \leq K$,
 \[\frac{h_{k}}{z_j^{(k)} - z_{j-1}^{(k)}} \leq
\frac54 M^{3 K /4}\frac1{\sqrt{2}} |u-v|^{-3/8} \leq \frac5{4\sqrt{2}} < 1\,.\]
In particular, this identifies the minimizer of the exponent in the r.h.s.\ of~\eqref{eq-Upsilon-general} and implies that
\[ \pi^\eta_{\bar{S}_n}\left(\overline{\Upsilon}^{(k+1)}_{M(j-1)+i}\,,\,\Upsilon^{(k)}_{j-1}\,,\,
\Upsilon^{(k)}_{j} \right) \leq C_3\, \exp\bigg(-\tfrac65\kappa_\b \frac{h_{k}^2}{z^{(k)}_{j}- z^{(k)}_{j-1}} \bigg)\,.
\]
Crucially however, the lower bound~\eqref{eq-hk-z-lower-bound} also gives that
\[
\frac{h^2_{k}}{z_j^{(k)} - z_{j-1}^{(k)}} \geq
  \frac45 M^{3k/4} \frac{h}{|u-v|} M^{-k/4}h = \frac45 M^{k/2} \frac{h^2}{|u-v|} \,.
\]

To simplify the notation put $ \rho = \frac{6}5\kappa_\b \frac{h^2}{|u-v|}$ and recall that $\rho \geq \frac65\kappa_\b M$ by~\eqref{eq-h-region}, hence we may take
$M$ sufficiently large so $\rho$ would also be large. The combination of the above inequalities together with a union bound gives
\begin{align*}
\pi^\eta_{\bar{S}_n}\bigg(\bigcup_{k=1}^{K-1} \bigcup_j \overline{\Upsilon}^{(k)}_{j} \bigg) &=
\pi^\eta_{\bar{S}_n}\bigg(\bigcup_{k=1}^{K-1} \bigcup_{i,j} \left\{\overline{\Upsilon}^{(k)}_{M(j-1)+i}\,,\,\Upsilon^{(k-1)}_{j-1}\,,\,
\Upsilon^{(k-1)}_{j}\right\} \bigg) \\
&\leq C_3 M e^{-\rho } + C_3 \sum_{k=2}^{K-1} M^{k} e^{-\rho\,\frac45 M^{(k-1)/2}}
\leq C'_3 e^{-\rho } = C'_3 e^{\frac65\kappa_\b\frac{h^2}{|u-v|}}\,,
 \end{align*}
where we used that $\rho \geq 2$ and $\sqrt{M} \geq \log M$ for any sufficiently large $M$ and it is understood that $\Upsilon^{(0)}_j$ is the full
 probability space.

We have reached level $K$ at which point we wish to examine the remaining points altogether.
Fix some $x \in  (z_{j-1}^{(K-1)}, z_{j}^{(K-1)})$ and let
\[ \Upsilon'_x = \bigg\{ \sigma:\, \htx\big(\lambda(\sigma),x\big) < \sum_{i=1}^{K} h_i \bigg\}\,.\]
As established before $h_{K}/(z_j^{(K)} - z_{j-1}^{(K)}) < 1$ and so
\begin{equation}
  \label{eq-Upsilon'-bound}
  \pi^\eta_{\bar{S}_n}\left(\overline{\Upsilon}'_x\,,\,\Upsilon^{(K)}_{j-1}\,,\,
\Upsilon^{(K)}_{j} \right) \leq C_3\, \exp\bigg(-\tfrac65\kappa_\b \frac{h_{K}^2}{z^{(K)}_{j}- z^{(K)}_{j-1}} \bigg)\,.
\end{equation}
On the other hand, by the definition of $K$ we have that $M^{K} \geq |u-v|^{1/2}/M$ (with the factor of $M$ due to possible integer rounding in $K$) and hence
\begin{align*}
\frac{h^2_{K}}{z_j^{(K)} - z_{j-1}^{(K)}} &\geq
\frac45 M^{K/2} \frac{h^2}{|u-v|} \geq \frac45 \frac{|u-v|^{1/4}}{\sqrt{M}} \frac{h^2}{|u-v|} \geq |u-v|^{1/4}\,,
\end{align*}
where the last inequality is due to the lower bound on $h^2$ in~\eqref{eq-h-region}. It now follows from~\eqref{eq-u-v-large-enough} that
\[ \exp\bigg(-\tfrac1{10}\kappa_\b \frac{h_{K}^2}{z^{(K)}_{j}- z^{(K)}_{j-1}} \bigg) \leq
\exp\left(-\tfrac1{10}\kappa_\beta|u-v|^{1/4} \right) \leq |u-v|^{-1}\,.\]
Together with~\eqref{eq-Upsilon'-bound} this implies that
\[ \pi^\eta_{\bar{S}_n}\left( \overline{\Upsilon}'_x\,,\,\Upsilon^{(K)}_{j-1}\,,\,
\Upsilon^{(K)}_{j} \right) \leq \frac{C_3}{|u-v|}\,\exp\bigg(-\tfrac{11}{10}\kappa_\b \frac{h^2}{|u-v|} \bigg)\,.\]
Summing over at most $|u-v|$ possible choices for $x$ we may now conclude that
\begin{align}\label{eq-large-dev-alpha-h}
\pi^\eta_{\bar{S}_n}\left(\htx(\lambda,x) \geq \alpha_M h \mbox{ for some $x\in(a+\tfrac12,b-\tfrac12)$}\right) \leq C_4 \exp\bigg(- \tfrac{11}{10}\kappa_\b \frac{h^2}{|u-v|}\bigg)\,,
 \end{align}
 where $\alpha_M = \sum_{0}^K M^{-i/4} < \big(1-M^{-1/4}\big)^{-1}$.

 Finally, by choosing $M$ to be sufficiently large we can obtain that $\alpha_M^2 < \frac{11}{10}$ and plugging this in~\eqref{eq-large-dev-alpha-h} (while recalling that we are in
 the regime where $\tfrac{h^2}{|u-v|} \leq h$ due to Eq.~\eqref{eq-h-region}) concludes the proof of~\eqref{eq-level-large-dev}, as required.
\end{proof}
\begin{remark*}
The truncation argument that was used in the proof of Theorem~\ref{thm-strip-large-dev} to reduce the problem to a finite domain is applicable in our upcoming arguments as well. Henceforth, when needed, we will thus work directly in the infinite volume setting to simplify the exposition.
\end{remark*}

We now introduce the main conceptual element in the proof of Theorem~\ref{thm:positive}. Recall our aim is to show that the open contour in the infinite strip $S=\{1,\ldots,\ell\}\times\Z$  has a reasonable probability --- namely of order $c/\ell$ --- of remaining in the upper half-plane (\ie above the dual line $\{(x,y) : y=-\frac12\}$).

%

Our approach, based on the Brownian bridge heuristics, is iterative and very much based on the intuitive picture in which the open contour really consists of two simple lines $\gamma_1, \ \gamma_2$, traveled at constant speed, one starting from the left boundary and moving towards the right boundary and viceversa for the second one, meeting in some intermediate point.  Such a picture, which can be made more precise by progressively revealing the contour from the left to right and from right to left (see Figure~\ref{fig:contparts}), allows \emph{hitting times} kind of arguments that we now explain.  Let $\tau_j^{(i)}$ ($i=1,2$) be the hitting time of either level $0$ or level $2^j$ for the curve $\gamma_i$. Then, conditioned to the event that both curves at their respective times $\tau_j^{(i)}$ have not yet joined and are both at level $2^j$, by monotonicity and symmetry, with probability at least $1/4$ both curves will either hit the next level $2^{j+1}$  or join together before hitting level $0$ (see Claim~\ref{cl:pathInduction} below for a precise formulation). Thus, with probability at least $4^{-n}$ we can force both curves to either hit level $2^n$ or join together before hitting level $0$. However, and that explains the heuristic bound $1/\ell$, once the curves are at level $2^n\approx \sqrt{\ell}$, then with probability bounded away from $0$ they will join together without hitting level $0$. In other words it is enough to force the curves to climb only $n=\frac 12 \log_2 \ell$ levels in order not to hit level~$0$.

The above sketch, however, suppresses a number of technical difficulties such as the boundary conditions and the dependence between the two contours. Moreover, and contrary to the behavior of the Brownian bridge, the law of the contour $\lambda$ is in fact \emph{asymmetric} w.r.t.\ the horizontal axis. This follows from our splitting-rule, which introduces a vertical bias for the contour: For instance, as illustrated in Figure~\ref{fig:asymm}, applying the SE splitting-rule clearly has the open contour move up with probability uniformly bounded away from $\frac12$.

\begin{figure}
\centering
\includegraphics[width=0.6\textwidth]{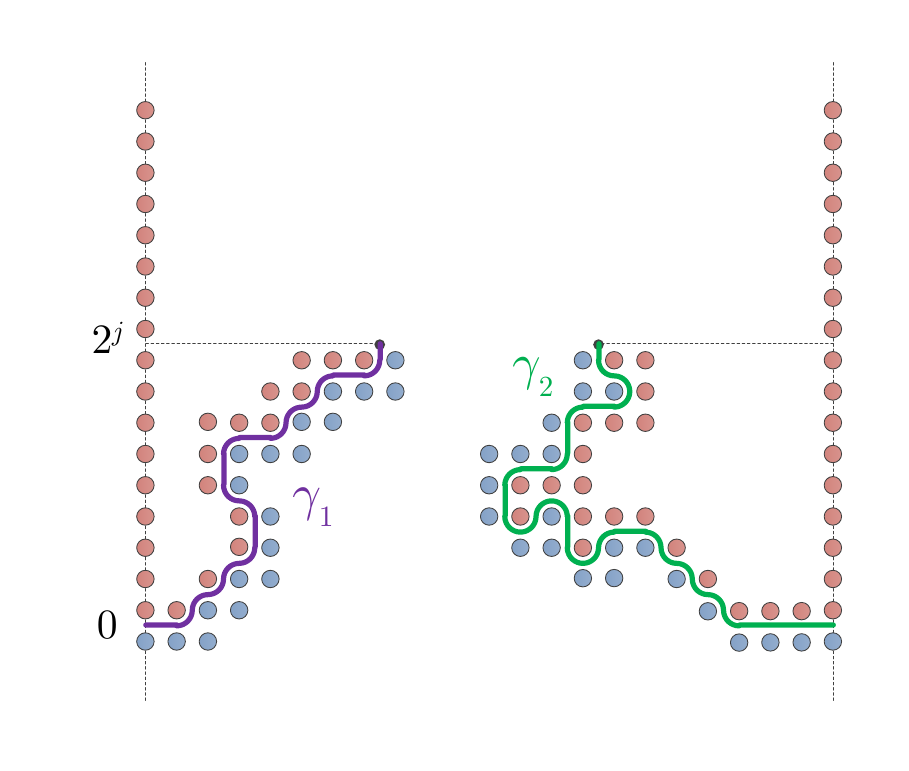}
\vspace{-0.3in}
\caption{The open contour in the infinite strip with mixed b.c., progressively exposed as two curves $\gamma_1,\gamma_2$ originating at its endpoints.}
\label{fig:contparts}
\end{figure}

\begin{figure}
\centering
\includegraphics[width=0.7\textwidth]{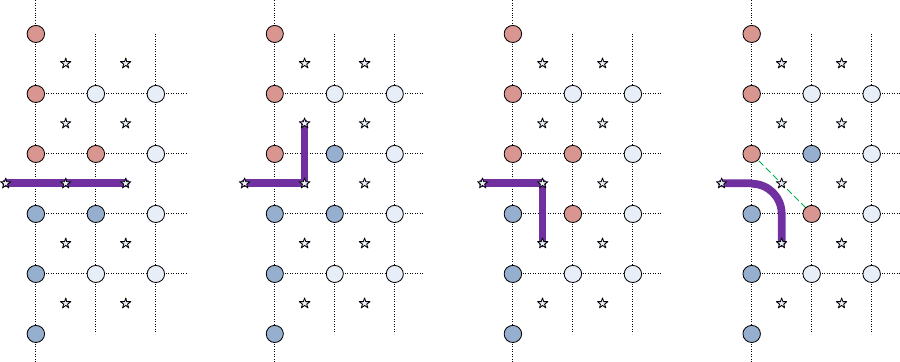}
\caption{Asymmetric behavior of the contour due to the global SE splitting-rule.}
\label{fig:asymm}
\end{figure}

To overcome this difficulty we consider the open contours formed by both the SE and the SW splitting-rules, $\gse$ and $\gsw$ resp., and examine their union $\Gamma = \gse \cup \gsw$. Most importantly the law of their union is symmetric w.r.t.\ the horizontal axis.  We will show that $\Gamma$ is essentially a ``tube'' of logarithmic width surrounding $\gse$, with added ``decorations'' from $\gsw$ which are components of at most logarithmic diameter (and similarly if we reverse the roles of $\gse,\gsw$). Up to these logarithmic corrections we may implement the heuristics of our above sketch, as stated in the following lemmas.
Here and in what follows we associate with an open contour $\gamma$ going from $u$ to $v$ a unit speed parametrization $\gamma(t)$, justifying hitting-time type of events (e.g.\ ``$\gamma$ hits the vertex $y$ prior to hitting $z$'' etc.).

\begin{Lemma}\label{lem-strip-hitting-h}
Let $\bar{S}=\bar{S}(a,b)$ and $H^*_i$ ($i\in\Z)$ be as in Theorem~\ref{thm-strip-large-dev}. For an Ising configuration $\sigma$ on $\bar{S}$ let $\gse(\sigma)$ and $\gsw(\sigma)$ be the two unique open contours in $\bar{S}^*$ formed by the SE and SW splitting-rules resp., \ie going from $(a+\frac12,\frac12)$ to $(b-\frac12,\frac12)$.
There exists some $C^\star(\beta)>0$ so that for any $h\ge 1$ the contour $\gse$ (resp.\ $\gsw$) hits $H^*_{-h-C^\star\log h}$ before hitting either $H^*_{h-C^\star\log h}$ or $(b-\frac12,\frac12)$ with probability at most $\frac12+C^\star/h$.
\end{Lemma}
\begin{proof}

We define the \emph{gain} of a connected subset of dual edges $A$ in the infinite strip $\bar{S}$ in the interval $I\subset\Z$ over distance $m$, denoted by $\gain(A,I,m)$, to be the maximal difference in $y$-coordinates between any two points in $A$ whose $x$-coordinates are contained in $I$ and differ by at most $m$:
\begin{equation}
  \label{eq-gain-def}
  \gain(A,I,m) := \max_{x,y,x',y'} \left\{ |y-y'| : |x-x'|\leq m, (x,y)\in A,\,(x',y')\in A, x,x'\in I \right\}\,.
\end{equation}
We define the \emph{gradient} of $A$ as its gain over distance 0.
The following claim bounds the gain of $\gamma$ in a neighborhood of $a$ and $b$:
\begin{claim}\label{clm-height-gamma-se}
Let $\gamma$ be the open contour with either SE or SW splitting-rule in the infinite strip $\bar{S}=\bar{S}(a,b)$ of side-length $\ell$ defined in Theorem~\ref{thm-strip-large-dev}.  Then for any $c>0$ there exists a constant $C^\star_1=C^\star_1(\beta,c)>0$ such that for all $1\leq m \leq \ell$,
\[
\pi^\eta_{\bar{S}}\Big(\gain(\gamma,[a-m,a+m],c\log m) > C^\star_1 \log m \Big) \leq C^\star_1/m\,.
\]
The analogous statement holds replacing $a$ with $b$.
\end{claim}
\begin{proof}[Proof of Claim~\ref{clm-height-gamma-se}]
Define $I=[a-m,a+m]\cap\bbZ$ and in what follows take $C_1^\star \geq c$.
Further let $u=(a+\frac12,\frac12)$ and $v=(b-\frac12,\frac12)$ denote the endpoints of the open contour $\gamma$, and let $C^\star_1=C^\star_1(\beta)$ be some constant to be determined later.
Define the set
\[
\Xi=\{((x,y),(x',y'))\in \bar{S}^2: x,x'\in I, |x-x'|\leq c\log m, |y-y'| \geq C^\star_1 \log m\}\,.
\]
If $\gain(\gamma,I,c\log m) > C^\star_1 \log m$ then there exist $z=(x,y), z'=(x',y')\in\gamma$ such that $(z,z')\in\Xi$.
Taking a union bound over ordered pairs of intermediate points $z,z'\in \Xi$ such that $z=(x,y)$ and $z'=(x',y')$ we get that
\[
\pi^\eta_{\bar{S}}\Big( \gain(\gamma,I,c\log m) > C^\star_1 \log m\Big) \leq \biggl[
\sum_{\substack{(z,z')\in\Xi}}
\sum_{\substack{\l:\, \d\l=\{u,v\} \\ z,z'\in\l}} q_{\bar{S}^*}(\l)\biggr]\,\big/\, \biggl[\sum_{\substack{\l:\, \d\l=\{u,v\}}} q_{\bar{S}^*}(\l)\biggr]\,.
\]
As we have already seen, Eq.~\eqref{eq-strip-two-pt-corr} provides a sharp estimate for the above denominator and it remains to consider the numerator. Recall Corollary~\ref{lem-3-pt-spins} that treated the measure of all open
contours $\lambda$ in a domain $\Lambda^*$ that go between two
endpoints $u,v$ as well as an intermediate point $z$, bounding it from
above by the product of the spin-spin correlations
$\pi_{\L^*}^*\left(\s_u\,\s_{z}\right)$ and $\pi_{\L^*}^*\left(\s_v\,\s_{z}\right)$.
Following essentially the same proof, \cite{cf:PV2}*{Lemma~5.4} gives an analogous version of this statement for all such contours $\lambda$ going through two ordered intermediate points $z,z'$ (that is, $\gamma$ connects $u$ to $z$, thereafter proceeds to $z'$ and ends at $v$) whereby
\begin{align}
\label{eq:whereby}
\sum_{\substack{\lambda:\,\d\lambda=\{u,v\}\\ z,z'\in\lambda}}
q_{\Lambda^*}(\lambda) &\leq \pi_{\L^*}^*\left(\s_u\,\s_{z}\right) \,
\pi_{\L^*}^*\left(\s_{z}\,\s_{z'}\right) \,
\pi_{\L^*}^*\left(\s_{z'}\,\s_{v}\right) \,.\end{align}
Therefore, Lemma~\ref{lem-spin-spin} implies that
\[
  \sum_{\substack{\lambda:\,\d\lambda=\{u,v\}\\ z,z'\in\lambda}}
q_{\bar{S}^*}(\lambda)\leq \frac{C_\beta^3 \exp\big(-[\tau_{\beta}(u-z)+\tau_{\beta}(z-z')+\tau_{\beta}(v-z')]\big)}{\sqrt{|u-z|\cdot|z-z'|\cdot|z'-v|}}
 \,.
\]
Since $\tau_{\beta}(\theta)\ge \tau_{\beta}(0)$ (see e.g.\ \cite{cf:BEF}) we can bound the
last exponent from above by
\[ e^{-\tau_{\beta}(0)[|u-z|+|z-z'|+|v-z'|]} \leq
e^{-\tau_{\beta}(0)[|u-v|+c'(|y|/m+|y-y'|)]}\,, \]
where $c'>0$ is an absolute constant; indeed, the last inequality is justified by the fact that $|y-y'|$ is at least a constant times $|x-x'|$
(recall that $C_1^*\geq c$) and a similar statement holds w.r.t. $y/m$ compared to $|x-a|$. Since summing over $x,x'$ amounts to a factor of $O(m^2)$, absorbing an additional $O(m)$ term from the sum over $y/m$ while recalling that $|y-y'|\geq C_1^\star\log m$
now implies that
\[
  \sum_{\substack{\lambda:\,\d\lambda=\{u,v\}\\ z,z'\in\lambda}}
q_{\bar{S}^*}(\lambda)\leq Cm^{-p} |u-v|^{-1/2} \exp(-\tau_\beta(u-v))\,,
\]
where $p$ can be made arbitrarily large by taking $C_1^\star$ large enough.
In conclusion,
\[
\pi^\eta_{\bar{S}}\Big( \gain(\gamma,I,c\log m) > C^\star_1 \log
m\Big)\le C_1^\star/m\,,
\]
completing the proof.
\end{proof}

\begin{claim}\label{clm-log-tube}
Let $\gse$ and $\gsw$ be the open contours with the SE and SW splitting-rule resp.\ in the infinite strip $\bar{S}=\bar{S}(a,b)$ of side-length $\ell$ defined in Theorem~\ref{thm-strip-large-dev}.
Then there exists some $C^\star_2=C^\star_2(\beta)>0$ so that for all $1\leq m \leq \ell$ with probability at least $1-C^\star_2/m$ every connected component of $\gsw \setminus \gse$ with zero distance from  $([a-m,a+m]\times\Z)\cap\gse$ has diameter at most $C^\star_2 \log m$.
\end{claim}
\begin{proof}[Proof of Claim~\ref{clm-log-tube}]
Let $\mathcal{B}$ denote the event that there exists a connected component of $\gsw \setminus \gse$ with zero distance from $([a-m,a+m]\times\Z)\cap\gse$ and has diameter at least $C^\star_2 \log m$.
We begin by conditioning on $\gse$.  The contour partitions $\bar{S}$ into two sets $\Stop$ and $\Sbot$.  For a set of dual edges $A$ let $V(A)\subset \Z^2$ denote the set of vertices at distance $\frac12$ from $A$.
The effect of conditioning on $\gse$ is equivalent to conditioning that $\sigma_U=\eta_U$
where $U=V(\Delta(\gse))$ (recall the definition of the edge-boundary $\D(\cdot)$ in \S\ref{sec:prelim} after Lemma~\ref{l:6.3}) and $\eta_{U}\in\{-1,+1\}^U$ is the configuration given by
\[
\eta_{u}=\begin{cases}
-1 & u \in \Stop,\\
+1 & u \in \Sbot.
\end{cases}
\]
Conditional on $\gse$ the configuration $\sigma$ on $\bar{S}\setminus U$ is given by the Ising model on $\Stop\setminus U$ and $\Sbot\setminus U$ with minus and plus boundary conditions respectively.  Let $\theta$ denote the ensemble of contours of this configuration given by the SW (not SE!) splitting rule.  Since the boundary conditions are all minus and all plus there are no open contours.  Every maximal connected segment of $\gsw \setminus \gse$ must be a subset of one of the closed contours of $\theta$ and must share a common vertex with $\gse$.

By Theorem~\ref{thm-strip-large-dev} we see that $\gse\cup \gsw \subset \Lambda^*$ where $\Lambda=\{1,\ldots,\ell\}\times\{-\ell,\ldots,\ell\}$ except with probability $O(\exp(-c\ell))$ for some $c(\beta)>0$.
By Claim~\ref{clm-height-gamma-se} there are at most $2 C^\star_1 m\log m$ vertices in $\mathcal{I}=([a-m,a+m]\times\Z)\cap\gse$ except with probability $C_1^\star/m$.
For  $z,z'\in \Lambda^*$ the probability that both lie in the same closed contour of $\theta$ is at most $\exp(-\tau_{\beta}(z-z'))$ by Lemma~\ref{lem-long-loop}.  Since the surface tension achieves its minimum on the sphere at $\tau_{\beta}(0)>0$ \cite{cf:BEF} combining the above estimates we have that
\begin{align*}
\pi^\eta_{\bar{S}}\left(\mathcal{B}\right) &\leq C\exp(-c\ell) + C_1^\star/m + \sum_{z\in\mathcal{I}} ~ \sum_{z' : |z-z'| > C_2^\star \log m} \exp(-|z-z'|\tau_{\beta}(0))\\
  &\leq C\exp(-c\ell) + C_1^\star/m + C m^{1-C_2^\star \tau_{\beta}(0)}\log^3 m \, ,
\end{align*}
The desired result follows from a sufficiently large choice of $C^\star_2$.
\end{proof}

Finally we show that the contours $\gse$ and $\gsw$ are unlikely to travel much farther than $h^2$ in the $x$-coordinate before attaining height $h$ or $-h$.
\begin{Lemma}\label{cl:tube}
Let $\bar{S}=\bar{S}(a,b)$.
For any $w>0$ and $0<h\leq \ell$ define the rectangle
\[
\mathcal{R}=\{a-\tfrac12-w,\ldots,a-\tfrac12+w\}\times\{\tfrac12-h,\ldots,h+\tfrac12\}\,.
\]
Let $\mathcal{B}$ denote the event that the contour $\gse$ (resp.\ $\gsw$) beginning at $(a-\frac12,\frac12)$  exits $\mathcal{R}$ and the first point it hits in $\mathcal{R}^c$ is in $\{a-\frac12-w-1,a-\frac12+w+1\}\times\{\frac12-h,\ldots,h+\frac12\}$.
There exists a constant $C^\star_3(\beta)>0$ independent of $w$ and $h$ such that
\[
\pi^\eta_{\bar{S}}\left(\mathcal{B}\right) \leq C^\star_3 \exp\big(- w/(C_3^\star h^2) \big) \, .
\]
\end{Lemma}
\begin{proof}
As before 
let $u=(a+\frac12,\frac12)$ and $v=(b-\frac12,\frac12)$ denote the endpoints of the open contour $\gamma$.
Let $\mathcal{B}_r$ and $\mathcal{B}_l$ denote the events that the contour exits to the right and left of $\cR$ respectively so that $\mathcal{B}=\mathcal{B}_l\cup\mathcal{B}_r$.  We will examine the case that the contour exits to the right and the left case will follow similarly.  Fix $c(\beta)=(3e C_\beta)^2$ where $C_\beta$ is the constant in Lemma~\ref{lem-spin-spin}.  For large enough $C_3^\star$ the bound holds trivially when $w\leq 2c h^2$ so assume that $w>2ch^2$.  We may define $M(w,h)\in\Z$ such that $0<w/(4ch^2)\leq M(w,h) \leq w/(ch^2)$ and $|cM h^2 - (b-a)|\geq \frac12 |b-a|$.  Define the set of sequences of points
\[
\Xi = \{(x_1,y_1),\ldots,(x_M,y_M)\in \bar{S}^M: \forall i, \ x_i=a+\tfrac12 + ich^2, |y_i| \leq h\}
\]
and note that if the contour exits $\mathcal{R}$ to the right then it must pass from $u$ through a sequence of points in $\Xi$ in order and then to $b$. For a sequence $\xi\in\Xi$ we say that $\lambda$ is $\xi$-admissible if it passes through the points in $\xi$ in order and then returns to $b$. Note also that by the construction of $M$ we have that $|x_M - (b-\frac12)| \geq \frac12|b-a|$.  Taking a union bound over sequences in $\Xi$ we get that
\[
\pi^\eta_{\bar{S}}\Big(\mathcal{B}_R \Big) \leq \biggl[
\sum_{\substack{\xi\in\Xi}}
\sum_{\substack{\l:\, \d\l=\{u,v\} \\ \xi-\text{admissible}}} q_{\bar{S}^*}(\l)\biggr]\,\big/\, \biggl[\sum_{\substack{\l:\, \d\l=\{u,v\}}} q_{\bar{S}^*}(\l)\biggr]\,.
\]

In analogy to  Corollary~\ref{lem-3-pt-spins} and \eqref{eq:whereby} one has
\begin{align*} \sum_{\substack{\lambda:\,\d\lambda=\{u,v\}\\ \xi-\text{admissible}}}
q_{\Lambda^*}(\lambda) &\leq \prod_{i=1}^{M+1} \pi_{\L^*}^*\left(\s_{(x_i,y_i)}\,\s_{(x_{i-1},y_{i-1})}\right)  \, ,
\end{align*}
where we denote $u=(x_0,y_0)$ and $v=(x_{M+1},y_{M+1})$.
Therefore, Lemma~\ref{lem-spin-spin} implies that
\begin{align*}
\sum_{\substack{\lambda:\,\d\lambda=\{u,v\}\\ \xi-\text{admissible}}}
q_{\bar{S}^*}(\lambda)&\leq \frac{C_\beta^{M+1} \exp\big(-\sum_{i=1}^{M+1}\tau_{\beta}((x_i,y_i)-(x_{i-1},y_{i-1}))\big)}{\sqrt{\prod_{i=1}^{M+1}|(x_i,y_i)-(x_{i-1},y_{i-1})|}}\\
&\leq C_\beta\frac{ \exp\big(-M-\tau_{\beta}(u-v)\big)}{(3h)^M\sqrt{\frac12|b-a|}}
 \,.
\end{align*}
Summing over the $(2h+1)^M$ elements of $\Xi$ we have that
\[
\sum_{\substack{\xi\in\Xi}}
\sum_{\substack{\l:\, \d\l=\{u,v\} \\ \xi-\text{admissible}}} q_{\bar{S}^*}(\l) \leq C_\beta\frac{ \exp\big(-M-\tau_{\beta}(u-v)\big)}{\sqrt{\frac12|b-a|}}\,.
\]
and it follows (recall Eq.~\ref{eq-strip-two-pt-corr}) that there exists a constant $C(\beta)>0$ such that
\[
\pi^\eta_{\bar{S}}\left(\mathcal{B}_R \right) \leq C e^{-M} \,
\]
and a similar estimate holds for $\pi^\eta_{\bar{S}}\left(\mathcal{B}_L \right)$.
Recalling that $M\geq w/(4ch^2)$ and taking a suitably large $C^\star_3$ now completes the proof.
\end{proof}

We now complete the proof of Lemma~\ref{lem-strip-hitting-h}.  Recalling that $\Gamma=\gse\cup\gsw$ we define $\Gamma_\gtop$ to be the highest path in $\Gamma$ connecting $(a+\frac12,\frac12)$ to $(b-\frac12,\frac12)$.
First observe that $\Gamma_\gtop$ is indeed well defined.  The collection of dual edges $\Gamma$ partitions $\bar{S}$ into two infinite components and possibly a number of finite components.
To construct $\Gamma_\gtop$, view the upper infinite component as a subset of $\R^2$ by
drawing a unit square centered at each of its points. Then $\Gamma_\gtop$ is its ``horizontal'' boundary connecting $(a+\frac12,\frac12)$ to $(b-\frac12,\frac12)$.
Define $\Gamma_\gbot$ similarly as the lowest path in $\Gamma$.  Then
\begin{align}\label{e:symmetryInequality}
&\pi_{\bar{S}}\left(\Gamma_\gtop\mbox{ hits $H^*_{-h}$ before hitting $H^*_{h}$ or $(b-\frac12,\frac12)$}\right)\nonumber\\
&\qquad \leq \pi_{\bar{S}}\left(\Gamma_\gbot\mbox{ hits $H^*_{-h}$ before hitting $H^*_{h}$ or $(b-\frac12,\frac12)$}\right)\nonumber\\
&\qquad = \pi_{\bar{S}}\left(\Gamma_\gtop\mbox{ hits $H^*_h$ before hitting $H^*_{-h}$ or $(b-\frac12,\frac12)$}\right) \, ,
\end{align}
where the inequality follows by the fact that $\Gamma_\gtop$ lies above $\Gamma_\gbot$ while the equality is by the symmetry of $\Gamma$.  It follows that
\[
\pi_{\bar{S}}\left(\Gamma_\gtop\mbox{ hits $H^*_{-h}$ before hitting $H^*_{h}$ or $(b-\frac12,\frac12)$}\right)\leq \frac12 \, .
\]
Lemma~\ref{cl:tube} guarantees that except with probability $O(\exp(-h^2))$ both contours $\gse$ and $\gsw$ hit either $H^*_{-h-C^\star\log h}$ or $H^*_{h-C^\star\log h}$ before traveling distance order $h^4$ in the horizontal direction, and as such we only need to consider the interval $[a-h^4,a+h^4]$.
Now set
\[C^\star := 8C^\star_1(\beta,C^\star_2) + 4C^\star_2(\beta)\]
where $C^\star_1,C^\star_2$ are the constants from Claim~\ref{clm-height-gamma-se} and Claim~\ref{clm-log-tube}. This guarantees that
\begin{align*}
\gain\left(\gse,[a-h^4,a+h^4],4C_2^\star \log h\right) &\leq 4C^\star_1(\beta,4C^\star_2) \log h\,,\\
\gain\left(\gsw,[a-h^4,a+h^4],4C_2^\star \log h\right) &\leq 4C^\star_1(\beta,4C^\star_2) \log h\,,\\
\max_{\substack{(x,y)\in\gsw\\x\in[a-h^4,a+h^4]}} d\left((x,y),\gse\right) &\leq 4C^\star_2 \log h
\end{align*}
and similarly around $b$ with probability at least $1 - C^\star/h$.  In particular, given the above event we have that the vertical distance between $\Gamma_\gtop$ and $\gse$ does not exceed $C^\star\log h$ in the intervals $[a-h^4,a+h^4]$ and $[b-h^4,b+h^4]$ which together with Eq.~\eqref{e:symmetryInequality} completes the proof.
\end{proof}

\begin{proof}[\textbf{\emph{Proof of Theorem~\ref{thm:positive}, lower bound}}]
The proof proceeds by progressively revealing the contour $\gse$. Let $w_0$ denote some large constant and let $w_i=2w_{i-1}-2C^\star \log w_{i-1}$ where $C^\star$ is the constant from Lemma~\ref{lem-strip-hitting-h}. Taking $w_0$ sufficiently large it is easily confirmed that $w_j \geq c 2^j$ for some constant $c>0$.

Starting from the left at $(\frac12,\frac12)$ for $j\geq 0$ let $A_{l,j}$ be the event that $\gse$ hits $H^*_{w_j}$ before hitting $H^*_{-1}$ or reaching $(\ell+\frac12,\frac12)$.  Similarly starting from the right  at $(\ell+\frac12,\frac12)$ let $A_{r,j}$ be the event that the contour hits $H^*_{w_j}$ before hitting $H^*_{-1}$ or reaching $(\frac12,\frac12)$.  Let $A_j=A_{l,j}\cap A_{r,j}$ and let $B_j$ be the event that the contour $\gse$ hits neither $H^*_{-1}$ nor $H^*_{w_j}$.

We begin by giving a crude lower  bound on the probability of $A_0$. Let $\mathcal{U}$ be the event that the spin configuration takes the value $(+)$ for all the vertices $(1,j)$ and $(\ell,j)$ for $0\leq j \leq w_0+1$. This occurs with probability at least
\[
\pi_S(\mathcal{U})\geq(\tfrac12 e^{-8\beta})^{2w_0+2}\, .
\]
On the event $\mathcal{U}$ the contour $\gse(\sigma)$ directly passes from $(\frac12,\frac12)$ to $(\frac12,w_0+\frac12)$ and from $(\ell+\frac12,\frac12)$ to $(\ell+\frac12,w_0+\frac12)$. It follows that $\mathcal{U}\subset A_0$ and hence,
\begin{equation}\label{e:A0}
\pi_S(A_0)\geq(\tfrac12 e^{-8\beta})^{2w_0+2}\, ,
\end{equation}
so in particular $A_0$ occurs with constant probability.

We will establish the following claim.
\begin{claim}\label{cl:pathInduction}
There exists a constant $c>0$ such that for all $j\geq 0$ we have that
\[\pi_S(A_{j+1}\cup B_{j+1} \mid A_j)\geq \frac14-c2^{-j}.
\]
\end{claim}
If we assume the claim then
\begin{align*}
\pi_S(A_{j+1}) + \pi_S(B_{j+1} ) & = \pi_S(A_{j+1}\cap A_j) + \pi_S(B_{j+1}\cap A_j) + \pi_S(B_j)\\
&\geq\Big(\frac14-c2^{-j}\Big)\left[ \pi_S(A_{j}) + \pi_S(B_{j} ) \right]
\end{align*}
where the first inequality follows from the fact that $A_j$ and $B_j$ are disjoint.  Hence by induction and equation \eqref{e:A0} we have that for any fixed positive integer $K$
\begin{equation}\label{e:inductiveHeight}
P(A_{K+\frac12\log_2 \ell}) + P(B_{K+\frac12\log_2 \ell})\geq P(A_0) \prod_{j=1}^{K+\frac12\log_2 \ell}\Big(\frac14-c2^{-j}\Big)\geq c'4^{-K}\ell^{-1}\,.
\end{equation}

We now prove the Claim~\ref{cl:pathInduction}.  Note that if $\sigma'\ge\sigma$ then the curve $\gse(\sigma')$ must lie on or above $\gse(\sigma)$ and hence the event $A_{j+1}$ is  increasing in $\sigma$ and so is $A_{j+1}\cup B_{j+1}$.  Through a series of monotonicity arguments we will relate this event to that in Lemma~\ref{lem-strip-hitting-h}.
Suppose that $A_j$ holds and that the left part of $\gse$ first hits $H_{w_j}^*$ at dual vertex $(z^L_{j}+\frac12,w_j+\frac12)$ and denote this part of the contour by $\gamma^L_{j}$.  Similarly denote the right part of the contour as $\gamma^R_{j}$ from $(\ell+\frac12,\frac12)$ to $(z^R_{j}+\frac12,w_j+\frac12)$ with $\frac12 \leq z^L_j < z^R_j \leq \ell+\frac12$.
Finally, let $D_{j+1}$ denote the event that the contour $\gse$ running between $(z^L_{j}+\frac12,w_j+\frac12)$ and $(z^R_{j}+\frac12,w_j+\frac12)$ either hits $H^*_{w_{j+1}}$ at both ends before hitting $H^*_{-1}$ or hits neither $H^*_{w_{j+1}}$ nor $H^*_{-1}$. With these definitions we claim that
\begin{eqnarray}
  \label{eq:1}
 \pi_S\left(A_{j+1} \cup B_{j+1} \mid A_j\right) \geq \pi^j_S\left(D_{j+1} \mid \sigma_{{U}_j}={\eta}_{{U}_j}\right)\,,
\end{eqnarray}
where $\pi_S^j$ denote the measure on the strip $S$ with boundary condition given by $(+)$ up to $w_j$ and $(-)$ above $w_j$,
\[
{U_j}=\left([1,z^L_j +1]\cup [z^R_j ,\ell] \right) \times\{w_j,w_j+1\}
\]
and
\[
{\eta}_{(u_1,u_2)}=\begin{cases}
+1 & u_2=w_j,\\
-1 & u_2= w_j +1\,.
\end{cases}
\]
The sequence of monotonicity arguments establishing~\eqref{eq:1} is best explained schematically, see Fig.~\ref{fig:transform} and its caption.

\begin{figure}
\centering
\includegraphics[width=6in]{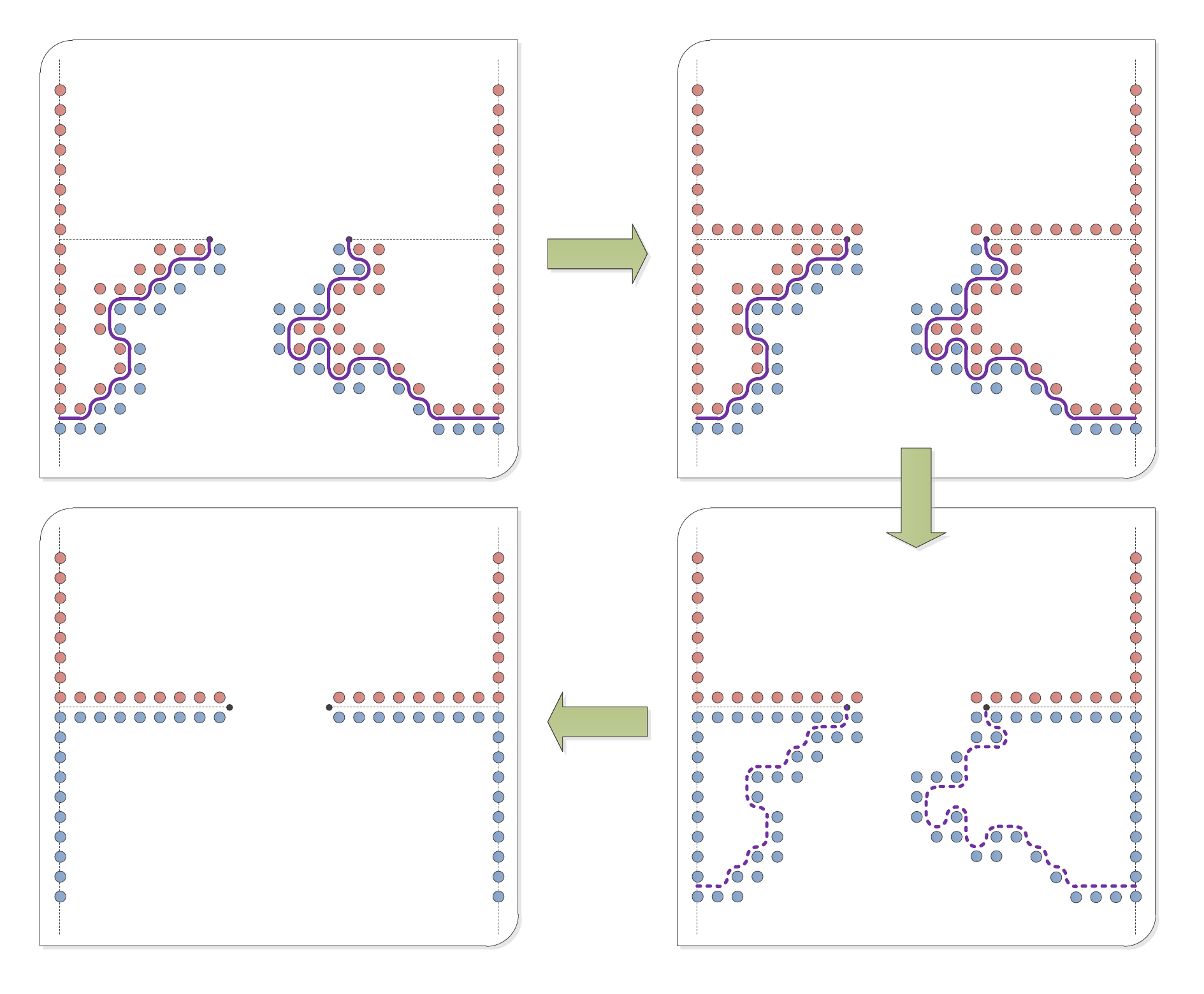}
\caption{Monotonicity
transformations to reduce the two segments of the open contour in the infinite strip to the setting of Lemma~\ref{lem-strip-hitting-h}.
First (top-right) impose extra $(-)$ spins at height $w_j+1$, making the increasing event $A_{j+1}\cup B_{j+1}$ more unlikely.
Next (bottom-right) set the spins at height $w_j$ and above the contours $\gamma_j^L,\gamma_j^R$ to $(+)$; this does not
change the probability of $A_{j+1}\cup B_{j+1}$ due to the Markov property of the Gibbs measure. Finally (bottom-left) remove the constraint that the spins just below $\gamma_j^L,\gamma_j^R$
are $(+)$, again making the event $A_{j+1}\cup B_{j+1}$ more unlikely. The dotted line is at height $w_j+1/2$.}
\label{fig:transform}
\end{figure}

Let $D^L_{j+1}$  denote the event that the contour $\gse$ from $(z^L_{j}+\frac12,w_j+\frac12)$ (resp.\ $(z^R_{j}+\frac12,w_j+\frac12)$) hits $H^*_{w_{j+1}}$ or $(z^R_{j}+\frac12,w_j+\frac12)$ (resp.\ $(z^L_{j}+\frac12,w_j+\frac12)$) before hitting $H^*_{-1}$.  Clearly $D_{j+1}=D^L_{j+1} \cap D^R_{j+1}$ and by symmetry we have that
\[
\pi^j_S\left(D^L_{j+1} \mid \sigma_{{U}_j}={\eta}_{{U}_j}\right)=\pi^j_S\left(D^R_{j+1} \mid \sigma_{{U}_j}={\eta}_{{U}_j}\right)\,.
\]
Since these events are both monotone, by the FKG inequality we have that
\[
\pi^j_S\left(D_{j+1} \mid \sigma_{{U}_j}={\eta}_{{U}_j}\right)\geq\pi^j_S\left(D^L_{j+1} \mid \sigma_{{U}_j}={\eta}_{{U}_j}\right)^2\,.
\]
It finally suffices to note that the conditional event $D^L_{j+1}$ is exactly the event considered in Lemma~\ref{lem-strip-hitting-h} once we shift the strip down by $w_j$ and set $h=w_j-C^\star \log w_j$, $a=z^L_{j}$ and $b=z^R_{j}-1$.  Then we have that
\[
\pi^j_S\left(D^L_{j+1} \mid \sigma_{{U}_j}={\eta}_{{U}_j}\right)\geq \frac12-O(1/w_j)\ge \frac12-c 2^{-j}\,,
\]
which completes the proof of Claim~\ref{cl:pathInduction}.

To complete the proof of the lower bound in Theorem~\ref{thm:positive} we will show that
\[
P(B_{K+1+\frac12\log_2 \ell} \mid A_{K+\frac12\log_2 \ell}) > \frac12\,,
\]
for a large enough constant $K=K(\beta)$.  Applying the same monotonicity transformations as in Claim~\ref{cl:pathInduction} this reduces to the probability of a contour reaching height $w_{j+1}-w_j > c 2^K \sqrt{\ell}$, which is less that $\frac12$ for large enough $K$ by Theorem~\ref{thm-strip-large-dev}.  Combining with equation~\eqref{e:inductiveHeight} it follows that
\[
P(B_{K+1+\frac12\log_2 \ell} ) > c \ell^{-1}\,,
\]
as required.
\end{proof}

\begin{proof}[\textbf{\emph{Proof of Theorem~\ref{thm:positive}, upper bound}}]
Let $\bbH = \{w=(x,y): y \geq 0\}$ denote the upper half-plane and consider the correlation between the spins at $u=(\frac12,\frac12)$ and $v = (\ell+\frac12,\frac12)$ in $\bbH^*$.
It is known (see e.g.~\cite{cf:MW}*{p161 Eq.~(5.29)}) that for some $C_1(\beta)>0$,
\[ \pi_{\bbH^*}^*(\s_u\, \s_v) \leq \frac{C_1}{\ell^{3/2}}e^{-\tau_\b(u-v)}\,.\]
By GKS we can reduce the domain to the dual half-strip of width $|u-v|$
\[ S^+ = \bbH^* \cap \{w=(x,y): \frac12\leq x \leq \ell+\frac12\}\,,\]
and obtain that
\begin{align*}
\pi_{\bbH^*}^*(\s_u\, \s_v) &\geq \pi_{S^+}^*(\s_u\, \s_v) = \sum_{\substack{\gamma\in S^+\\ \delta\gamma = \{u,v\}}} q_{S^+}(\gamma) \geq \sum_{\substack{\gamma\in S^+\\ \delta\gamma = \{u,v\}}} q_{S}(\gamma)\,,
\end{align*}
where $S=\{w=(x,y)\in\Z_2^*: \frac12\leq x \leq \ell+\frac12\}$ is the dual strip and the last inequality is justified by Lemma~\ref{l:6.3}.
Combining with Eq.~\eqref{eq-strip-two-pt-corr}
we can conclude that
\[ \pi_S^\eta\left(\sigma:\lambda(\sigma)\mbox{ stays above }H_{-1}^*\right) = \frac{\sum_{\substack{\gamma\in S^+\\ \delta\gamma = \{u,v\}}} q_{S}(\gamma)}{\sum_{\substack{\gamma\in S\\ \delta\gamma = \{u,v\}}} q_{S}(\gamma)} \leq \frac{{C_1 \ell^{-3/2}}e^{-\tau_\b(u-v)}}{{c_2 \ell^{-1/2}}e^{-\tau_\b(u-v)}} = C / \ell\,,\]
thus completing the proof.
\end{proof}

\begin{proof}[\textbf{\emph{Proof of Corollary~\ref{cor-spin-spin-strip}}}]
The upper bound follows directly from the GKS inequalities and the exact solution~\cite{cf:MW} in the infinite half-plane. As for the lower bound, one has
\begin{align*}
 \pi^*_{S^+}(\sigma_u \sigma_v)&=\sum_{\lambda\in S^+:\delta\lambda=\{u,v\}} q _{S^+} (\lambda)\ge
\sum_{\lambda\in S^+:\delta\lambda=\{u,v\}} q _{S^*} (\lambda)\\
& = \frac{\sum_{\lambda\in S^+:\delta\lambda=\{u,v\}} q _{S^*} (\lambda)}{\sum_{\lambda\in  S^*:\delta\lambda=\{u,v\}} q _{S^*} (\lambda)}
\sum_{\lambda\in S^*:\delta\lambda=\{u,v\}} q _{S^*} (\lambda) \ge \frac c\ell\pi^*_{S^*}(\sigma_u \sigma_v)
  \end{align*}
where in the first inequality we applied Lemma \ref{l:6.3} and in the second one we used~\eqref{eq:7} and the lower bound of Theorem~\ref{thm:positive}.
Recalling~\eqref{eq-strip-two-pt-corr}, the desired result follows.
\end{proof}

\section{Proofs of Propositions~\ref{prop:Bound1} and~\ref{prop:Bound2}}\label{sec:new-equlibrium-proofs}
The remaining part of this work is devoted to the proof of the two equilibrium estimates needed for the new recursive scheme (detailed in Section~\ref{sec:new-induction})
using the estimates obtained thus far.

\subsection{Proof of Proposition~\ref{prop:Bound1}}
Let $u,v$ be the initial and final point of $\l$ and as before let $H_h^*$ denote level $h$ of the rectangle $R^*$. By Eq.~\eqref{eq:8} we have the following:
\begin{gather}
  \pi^{(-,-,+,-)}\Bigl(\s:\, \l(\s) \text{ reaches height }  \d \a
      \sqrt{\ell}\Bigr)\nonumber \\= \sum_{\substack{\l \text{ reaches }H_{\d\alpha\sqrt{\ell}}^* \\ \d\l=\{u,v\}}}q_{R^*}(\l)\,\big/\sum_{\substack{\l\subset
      R^* \\ \d\l=\{u,v\}}} q_{R^*}(\l)\,.
  \label{eq:10}
\end{gather}
Observe that by monotonicity that \eqref{eq:10} is increasing in the height of $R$ and so without loss of generality we take
\[
R=S^+=\{(x_1,x_2):\ 1\le x_1\le \ell \text{ and } x_2\ge 1\}.
\]
We now bound separately the numerator and the denominator in \eqref{eq:10}.  By the same argument used to prove equation~\eqref{eq-w-in-contour} (note that here we bound the probability of the contour exceeding height $\delta\alpha\sqrt{\ell} < \ell$ by the assumptions on $\delta$ and $\alpha$), there exists some $c_1(\beta)>0$ such that for any $z \in H_{\delta\alpha\sqrt{\ell}}^*$,
\[
\sum_{\substack{\lambda:\,\d\lambda=\{u,v\}\\ z\in\lambda}}
q_{R^*}(\lambda)\leq
 \frac{c_1}{\ell} \exp\big(-\tau_{\beta}(0)\ell - \kappa_\beta (\delta\alpha)^2\big)\,,
\]
and therefore
\begin{equation}\label{e:BoxContourNumerator}
\sum_{z\in H_{\delta\alpha\sqrt{\ell}}^*} \sum_{\substack{\lambda:\,\d\lambda=\{u,v\}\\ z\in\lambda}}
q_{R^*}(\lambda)\leq
 c_1 \exp\big(-\tau_{\beta}(0)\ell - \kappa_\beta (\delta\alpha)^2\big)\,.
\end{equation}
Next we bound from below the denominator in \eqref{eq:10}. Recall that $S$ is the infinite
strip $S=\{(x_1,x_2):\ 1\le x_1\le
\ell\}$. Then by Lemma~\ref{l:6.3},
\begin{equation}
  \label{eq:14}
\sum_{\substack{ \l\subset R_n^* \\ \d\l=\{u,v\}}} q_{R^*}(\l)\ge \Biggl[\,\frac{\sum_{\substack{\l\subset
   R^* \\ \d\l=\{u,v\}}} q_{S^*}(\l)}{\sum_{\substack{\l\subset S^* \\ \d\l=\{u,v\}}}
 q_{S^*}(\l) }\,\Biggr]\,\sum_{\substack{\l\subset S^*\\ \d\l=\{u,v\}}} q_{S^*}(\l)\,.
\end{equation}
The last factor is estimated in Eq.~\ref{eq-strip-two-pt-corr}
while the first factor can be interpreted as the probability in the
\emph{canonical ensemble}
given by the weights $q_{S^*}(\cdot)$ and conditioned to start at $u$
and to end at $v$ that the contour $\l$ stays above the line at
height $-1/2$. By Theorem~\ref{thm:positive}, this is of order $c/\ell$. The desired claim then immediately follows.
\qed

\subsection{Proof of Proposition~\ref{prop:Bound2}}

Let $R$ be the $\ell \times \alpha\sqrt{\ell}$ rectangle given by the endpoints $(x,y,y',x')$ clockwise starting from the Northwest corner, and let $u$ and $v$ denote the West and East endpoints of the interval $\Delta$ centered on the South border, as was shown in Figure~\ref{fig:rect-equil2} in \S\ref{sec:new-induction}.
Recall that in our setting we have a b.c.\ $\eta$ which is $(-)$ in the North and on $\Delta$ and $(+)$ otherwise, and that the event $\mathcal{V}$ states that $\Delta$ is connected to the North via two contours confined to the left and right halves of $R$ respectively.

Our first step in establishing that $\mathcal{V}$ occurs except with probability $c_1 \ell^{c_1}\exp(-c_2 \alpha^2)$ is eliminating (except with the aforementioned error probability) the scenario where open contours connect $x$ to $y$ and $u$ to $v$.

\begin{Lemma}\label{lem:ab-cd}
Let $R = (x,y,y',x')$ and $\Delta$ be an interval of length $s \alpha^2$ centered on the South border $x'y'$.
For any $\beta > \beta_c$ there exist $c_3,c_4>0$ and $s_0 >0$ depending only on $\beta$ such that if $s \geq s_0$ then for any $\ell\in \N$
we have
\[ \pi_R^\eta(\delta\lambda_1=\{ x,  y\}\,,\, \delta \lambda_2=\{u, v\}) \leq \ell^{c_3}\exp(-c_4 \alpha^2)\,.\]
\end{Lemma}
\begin{proof}
Our starting point is the equality~\eqref{eq:8}, which allows us to rewrite the probability of certain contours in terms of their weights:
\[ \pi_R^\eta(\delta\lambda_1=\{ x,  y\}\,,\, \delta \lambda_2=\{u, v\}) = \frac{\Psi_1}{\Psi_2}\]
where
\begin{align}\label{eq-Psi-def}
  \Psi_1 :=\sum_{\substack{\underline{\lambda}= \lambda_1\sqcup\lambda_2 \\ \d\lambda_1=\{x,y\},\, \d\lambda_2=\{u,v\}}} q_{R^*}(\underline{\lambda})\,,&\quad&
  \Psi_2 := \sum_{\substack{\underline\l = \l_1 \sqcup \l_2 \\ \d\underline\l=\{x,y,u,v\}}} q_{R^*}(\underline\l)\,.
\end{align}
By Lemma~\ref{l:6.5} we have
\begin{align*}
\Psi_1 &\leq \sum_{\substack{\l_1 \\
\d\l_1=\{x,y\}}}q_{R^*}(\l_1)
\sum_{\substack{\l_2 \\
\d\l_2=\{u,v\}}}q_{R^*}(\l_2)
= \pi_{R^*}^*(\sigma_x\sigma_y) \pi_{R^*}^*(\sigma_u \sigma_v)\,,
\end{align*}
where the last equality is by~\eqref{eq:7}.
Plugging in Lemma~\ref{lem-spin-spin} it now follows that
\begin{align}
  \Psi_1 &\leq \exp\left(-\left[\tau_{\beta}(x-y) +\tau_{\beta}(u-v)\right]\right)= \exp\left(-\left[\ell +s\alpha^2\right]\tau_{\beta}(0)\right)\,.
\label{eq-Psi1}
\end{align}
Next we consider $\Psi_2$. As before,~\eqref{eq:7}, the GKS-inequalities~(\cites{cf:Griffiths,cf:KS}) and symmetry imply that
\begin{align*}
\Psi_2 &= \pi_{R^*}^*(\sigma_x\sigma_y\sigma_u\sigma_v) \geq
\pi_{R^*}^*(\sigma_x\sigma_u) \pi_{R^*}^*(\sigma_y\sigma_v)=\pi_{R^*}^*(\sigma_x\sigma_u)^2\,.
\end{align*}
\begin{figure}
\centering
\includegraphics[width=0.95\textwidth]{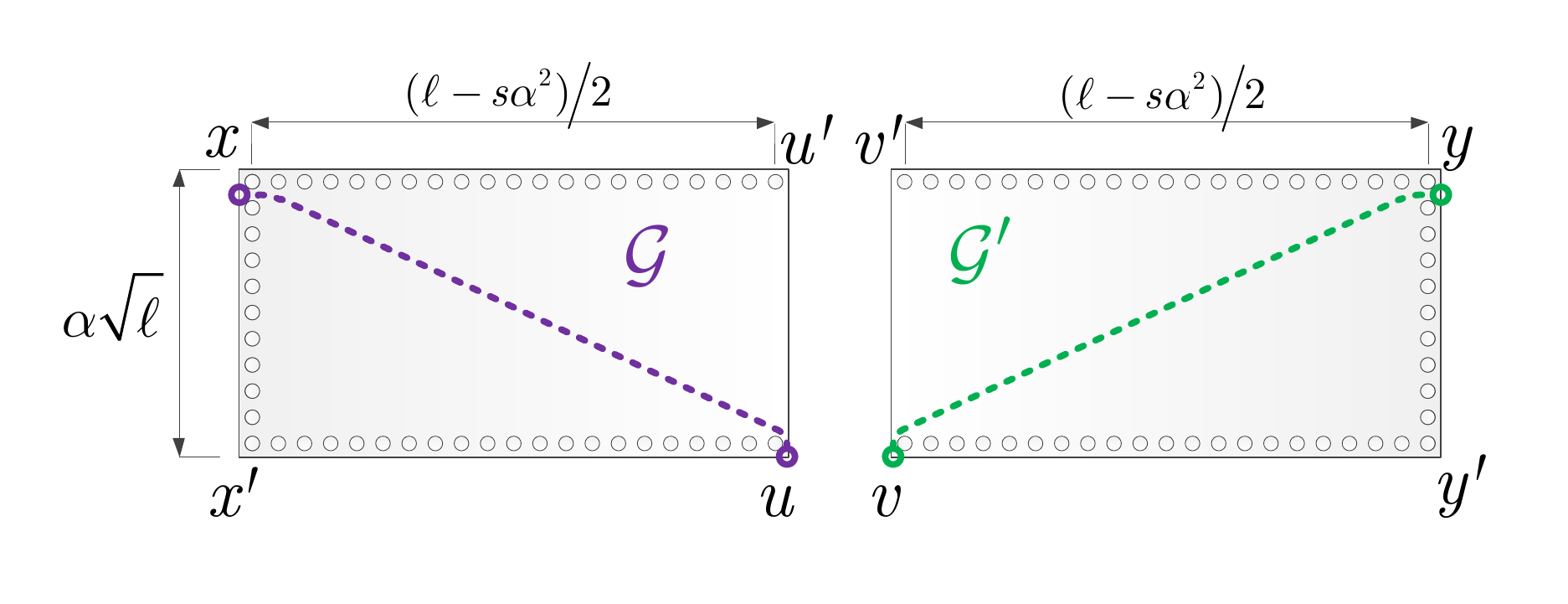}
\vspace{-0.3in}
\caption{Domain decrease from the rectangle $R$ to a disjoint union of rectangles separating the two open contours.}
\label{fig:rect-split}
\end{figure}
Furthermore, by GKS, spin-spin correlations are non-decreasing in the domain so
we can clearly confine our domain to the disjoint union of the two rectangles $\cG = (x,u',u,x')$ as shown in Figure~\ref{fig:rect-split}, and obtain by symmetry that
\[ \Psi_2 \geq X^2 \quad\mbox{ where } X:= \pi^*_{\cG^*}(\sigma_x \sigma_u) \,.\]
To control the value of $X$, let $z$ be the center of the rectangle $\cG$ and further define $\cG_1 = (x,a,b,x')$ and $\cG_2 = (a,u',u,b)$ to be the left and right halves of $\cG$, each of dimensions $\tfrac14(\ell-s\alpha^2) \times \alpha\sqrt{\ell}$. See Figure~\ref{fig:rect-split2} for an illustration.
The GKS-inequalities (together with a reduction of the domain) yield
\[ \pi^*_{\cG^*}(\sigma_x \sigma_u) \geq \pi^*_{\cG_1^*}(\sigma_x \sigma_z) \pi^*_{\cG_2^*}(\sigma_u \sigma_z)\,.\]
Another application of ~\eqref{eq:7} gives
\begin{align*}
 X &\geq \sum_{\substack{\l\subset\cG_1 \\ \d\l=\{x,z\}}} q_{\cG_1^*}(\l)
 \sum_{\substack{\l\subset\cG_2 \\ \d\l=\{u,z\}}} q_{\cG_2^*}(\l)
= \bigg(\sum_{\substack{\l\subset\cG_1 \\ \d\l=\{x,z\}}} q_{\cG_1^*}(\l) \bigg)^2\,,
\end{align*}
with the equality due to symmetry. Define $\bar{S}$ to be the infinite half-strip of width $\frac14(\ell-s\alpha^2)$ obtained by extending the South border of $\cG_1$ (\ie the edge $b x'$) to $-\infty$. Since $\cG_1$ is  a subgraph of $\bar{S}$ it  follows from Lemma~\ref{l:6.3} that
\[ X \geq \bar{X}^2 \quad\mbox{ where }\quad \bar{X} := \sum_{\substack{\l\subset\cG_1 \\ \d\l=\{x,z\}}} q_{\bar{S}^*}(\l) \,.\]
We now claim that
\begin{equation}
  \label{eq-barX-over-barY}
  \frac{\bar{X}}{\bar{Y}} \geq 1 - (\ell/4)^{c_1}\exp(-c_2\alpha^2)  \,,
\end{equation}
where
\[
\bar{Y}:=\sum_{\substack{\l\subset\bar{S} \\ \d\l=\{x,z\}}} q_{\bar{S}^*}(\l)\,.
\]
Indeed, $\bar{X} / \bar{Y}$ is precisely the probability that the contour $\lambda \subset \bar{S}_n$ whose endpoints are $\d\l = \{x,z\}$ stays above the horizontal line $x' b$ (the South border of $\cG_1$).
If $z'$ is the midpoint of $x$ and $x'$ then by monotonicity if we condition all $(+)$ in the rectangle $(x,a,z,z')$ then this only increases the probability that the contour hits the line $x' b$.
Proposition~\ref{prop:Bound1} then establishes equation~\eqref{eq-barX-over-barY}.  Taking $c_3$ large enough
in the statement of Lemma \ref{lem:ab-cd} we can assume that $\alpha$ is at least a large constant times $\sqrt{\log \ell}$, in which case \eqref{eq-barX-over-barY}
gives
$\bar{X}=(1-o(1))\bar{Y}$.

\begin{figure}
\centering
\includegraphics[width=0.7\textwidth]{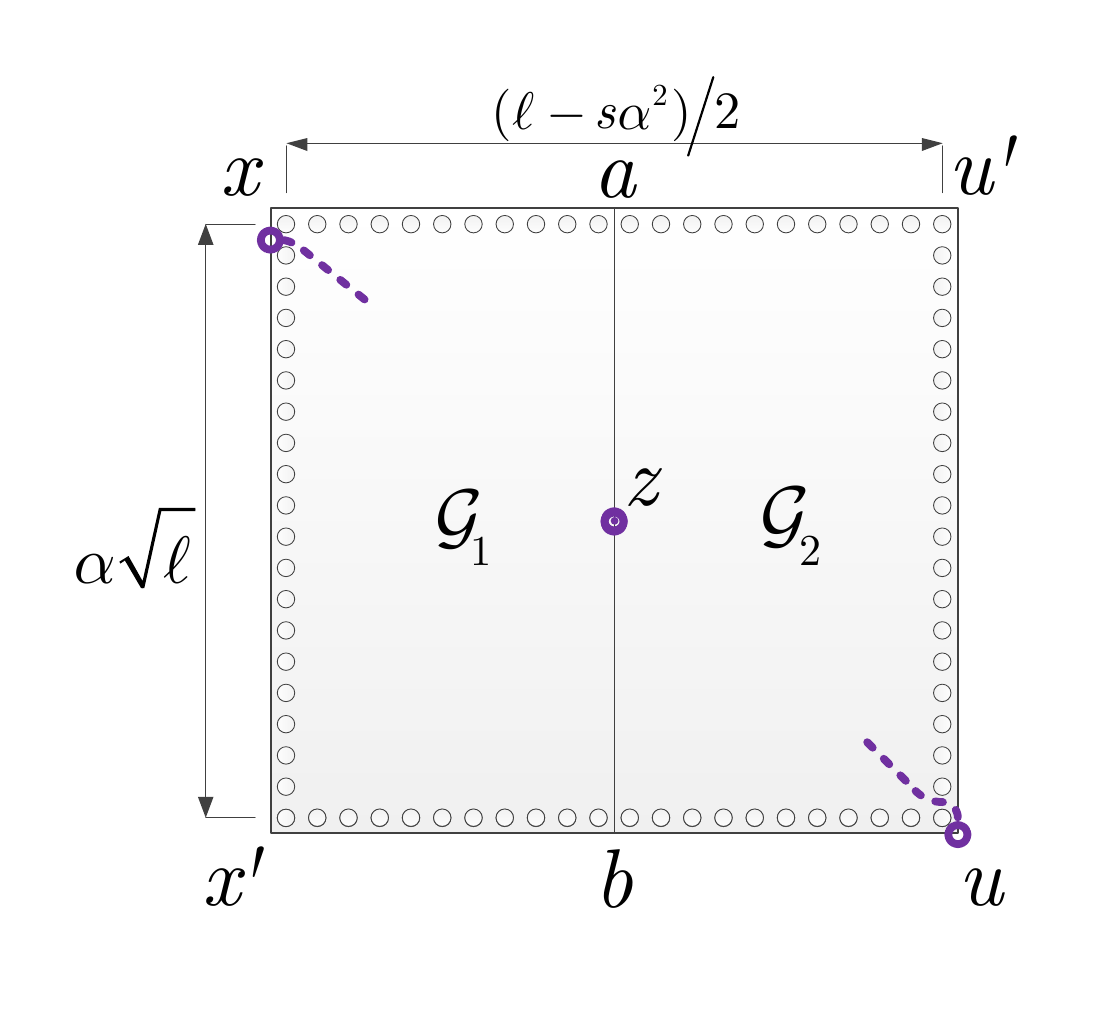}
\vspace{-0.3in}
\caption{Comparing the weight of contours from $x$ to $u$ to that of contours connecting both points to the center of the rectangle $z$.}
\label{fig:rect-split2}
\end{figure}

On the other hand, taking $S$ to be the  strip obtained by extending the North and South boundaries of $\cG_1$ to $\pm\infty$, we can now express $\bar{Y}$ in terms of
\[ Y:=\sum_{\substack{\l\subset S \\ \d\l=\{x,z\}}} q_{S^*}(\l)\,.\]
To this end, observe that $\bar{Y}/Y$ is the probability that the open contour $\lambda$ connecting $x,z$ in $S$ stays below the horizontal line $x a$ (the North border of $\cG_1$).
Therefore, by monotonicity and Theorem~\ref{thm:positive} there exists some $c=c(\beta)$ such that
\begin{equation}
  \label{eq-barY-over-Y}
  \frac{\bar{Y}}{Y} \geq c/\ell\,,
\end{equation}
and on the other hand \cite{cf:GI}*{Formula (2.22)} gives for some $c_\beta>0$
\begin{equation}
  \label{eq-Y}
  Y = \frac{c_\beta+o(1)}{\sqrt{\ell}}\exp(-\tau_{\beta}(x-z)) \,.
\end{equation}
Combining~\eqref{eq-barX-over-barY},\eqref{eq-barY-over-Y} and \eqref{eq-Y} it now follows that
\[ \bar{X} \geq c_1'\ell^{-c_2'}\exp(-\tau_{\beta}(x-z))\,,\]
and recalling that $\Psi_2 \geq X^2 \geq \bar{X}^4$ we deduce that there exists some $c''(\beta)>0$ such that
\begin{equation}
  \label{eq-Psi2}
  \Psi_2 \geq c'' \ell^{-c''}\exp(-2\tau_{\beta}(x-u))\,.
\end{equation}
To conclude the proof, we combine \eqref{eq-Psi1}, \eqref{eq-Psi2} and get
that
\[ \frac{\Psi_1}{\Psi_2} \leq (1/c'') \ell^{c''} \exp\left(-\left[\ell + s\alpha^2\right]\tau_{\beta}(0)+2\tau_{\beta}(x-u) \right)  \,.
 \]
Recall that $\tau_{\beta}(\theta)$ is an analytic and even function of $\theta$ for any $\beta>\beta_c$, hence in particular there exists some $c=c(\beta)>0$ such that
\[ \tau_{\beta}(\theta) - \tau_{\beta}(0) \leq  c \theta^2 \mbox{ for any $\theta\in\R$.}\]
Since in our case $\theta \leq \arctan\big(\frac{\alpha\sqrt{\ell}}{\ell/2}\big) \leq 2\alpha/\sqrt{\ell}$ it follows that for some $c'=c'(\beta)>0$
\[ \tau_{\beta}(x-u) \leq \tau_{\beta}(0)|x-u| + c' \alpha^2\,.\]
On the other hand,
\[\ell + s\alpha^2 - |x-u| = O(s \alpha^2)\,.\]
Combining these inequalities completes the proof.
\end{proof}

We next wish to show that whenever $\delta\lambda_1=\{ x,  u\}\,,\, \delta \lambda_2=\{y, v\}$ we also have that the corresponding open contours are confined to the $R_l$ and $R_r$, the left and right halves of $R$ respectively, except with an appropriate exponentially small probability. The complement event we wish to analyze is illustrated in Figure~\ref{fig:rect-central}.
\begin{figure}
\centering
\includegraphics[width=0.9\textwidth]{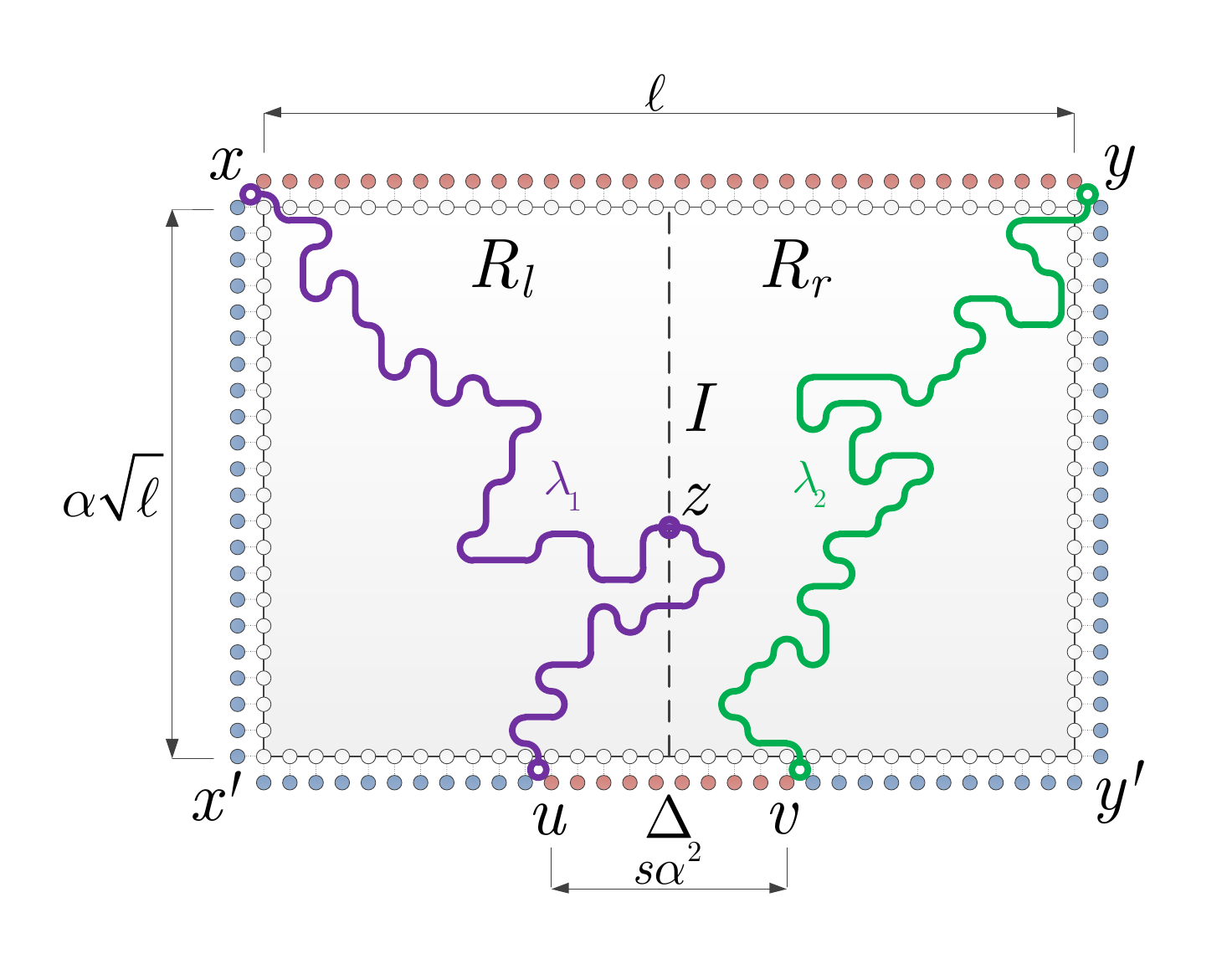}
\vspace{-0.3in}
\caption{Open contours under b.c.\ $(-,+,\Delta)$ crossing the central column of $R$, addressed by the estimate in Lemma~\ref{lem:ac-left-bd-right}.}
\label{fig:rect-central}
\end{figure}

\begin{Lemma}\label{lem:ac-left-bd-right}
Let $R = (x,y,y',x')$ and $\Delta$ be an interval of length $s \alpha^2$ centered on the South border $x'y'$. Denote by $R_l$ and $R_r$ the left and right halves of $R$ respectively.
For any $\beta > \beta_c$ there exist $c_5,c_6>0$ and $s_0 >0$ depending only on $\beta$ such that if $s \geq s_0$ then for any $\ell\in \N$
we have
\[ \pi_R^\eta\left(\lambda_1 \subset R_l \,,\, \lambda_2 \subset R_r  \given \delta\lambda_1=\{ x,  u\}\,,\, \delta \lambda_2=\{y, v\} \right) \geq 1 - \ell^{c_5}\exp(-c_6 \alpha^2)\,.\]
\end{Lemma}
\begin{proof}
Let $I$ denote the central column of $R$, \ie $I$ is the vertical line connecting the centers of the North and South boundaries of $R$.
We aim to bound the probability that the contour connecting $x$ to $u$ crosses $I$, and similarly for the contour connecting $y$ to $v$.
By equality~\eqref{eq:8} we can write the former probability as $\Phi_1/\Phi_2$
where
\begin{align*}
  \Phi_1 :=\sum_{\substack{\underline{\lambda}= (\lambda_1,\lambda_2) \\ \d\lambda_1=\{x,u\} \,,\, \lambda_1 \cap I \neq \emptyset \\ \d\lambda_2=\{y,v\}}} q_{R^*}(\underline{\lambda})\,,&\quad&
  \Phi_2 := \sum_{\substack{\underline\l = (\lambda_1,\lambda_2) \\ \d\underline\l_1=\{x,u\} \\ \d\underline\l_2=\{y,v\}}} q_{R^*}(\underline\l)\,,
\end{align*}
and a union bound (together with symmetry) gives
\[ \pi_R^\eta\left(\lambda_1 \subset R_l \,,\, \lambda_2 \subset R_r  \given \delta\lambda_1=\{ x,  y\}\,,\, \delta \lambda_2=\{u, v\} \right) \geq 1 - 2 \frac{\Phi_1}{\Phi_2} \,.\]
To bound $\Phi_2$ from below we compare it to $\Psi_1,\Psi_2$ defined in~\eqref{eq-Psi-def}. Indeed,
\[ \frac{\Phi_2}{\Psi_2} = 1 - \frac{\Psi_1}{\Psi_2} = 1 -  \pi_R^\eta(\delta\lambda_1=\{ x,  y\}\,,\, \delta \lambda_2=\{u, v\}) \geq 1- \ell^{c_3}\exp(-c_4 \alpha^2)\,,\]
where the last inequality is precisely the statement of Lemma~\ref{lem:ab-cd}. Combining this with the estimate on $\Psi_2$ given in ~\eqref{eq-Psi2} we conclude that for some absolute $c''>0$,
\begin{equation}
  \label{eq-Phi2}
  \Phi_2 \geq \left(1- \ell^{c_3}\exp(-c_4 \alpha^2)\right)c'' \ell^{-c''} e^{-\tau_{\beta}(x-u)-\tau_{\beta}(y-v)}
\end{equation}
(we can assume that $\alpha$ is at least a large constant times $\sqrt{\log \ell}$, otherwise the statement of Lemma \ref{lem:ac-left-bd-right} trivially holds).
\begin{remark*}
  The estimate given in~\eqref{eq-Psi2} for $\Psi_2$ in the proof of Lemma~\ref{lem:ab-cd} in fact had an exponent of $-2\tau_{\beta}(x-u)$. The above bound featuring $-(\tau_{\beta}(x-u)+\tau_{\beta}(y-v))$ in the exponent readily follows from \eqref{eq-Psi2} by the symmetry between $\{x,u\}$ and $\{y,v\}$ in the definition of $\Psi_2$.
\end{remark*}
It remains to bound $\Phi_1$. To this end, for a given contour $\gamma$ define $\cG_{\gamma}$ to be the graph with the edge-set $R^* \setminus \Delta(\gamma)$, where $\Delta(\gamma)$ is the edge-boundary of the contour $\gamma$. Crucially, in our case the edge-boundary of $\gamma$ is disjoint to the edges of $\l_1$ since these contours are compatible. Therefore, by Lemma~\ref{lem-edge-boundary},
\[ q_{R^*}(\l,\gamma) = q_{R^*}(\gamma) q_{\cG_{\gamma}}(\l)\,,\]
and plugging this into the definition of $\Phi_1$ we deduce that
\[ \Phi_1 = \sum_{\gamma \,:\, \d\gamma=\{y,v\}} \bigg( q_{R^*}(\gamma)\sum_{\substack{\l \,:\, \d\l=\{x,u\} \\ \l\cap I \neq \emptyset}} q_{\cG_{\gamma}}(\l)\bigg)\,.\]
The sum over the weights $q_{\cG_\gamma}$ can be estimated via Corollary~\ref{lem-3-pt-spins} using a simple union bound (see e.g.~\cite{cf:CGMS}*{Eq.(3.3)} for a similar argument).
Indeed,
\[
\sum_{\substack{\l \,:\, \d\l=\{x,u\} \\ \l\cap I \neq \emptyset}} q_{\cG_{\gamma}}(\l)
\leq \sum_{z\in I} \sum_{\substack{\lambda:\,\d\lambda=\{x,u\}\\ z\in\lambda}}
q_{\cG_\gamma}(\lambda) \leq \sum_{z\in I} \pi_{\cG_\gamma}^*\left(\s_x\,\s_z\right) \, \pi_{\cG_\gamma}^*\left(\s_u\,\s_z\right)\,.
\]
At the same time, we can use \eqref{eq:7} and then increase the appropriate domains to the infinite-volume lattice to obtain that
\begin{align*}
\sum_{\gamma \,:\, \d\gamma=\{y,v\}} q_{R^*}(\gamma) &\leq \pi^*_\infty(\sigma_y \,\sigma_v)\,,\\
\sum_{\substack{\l \,:\, \d\l=\{x,u\} \\ \l\cap I \neq \emptyset}} q_{\cG_{\gamma}}(\l)
 &\leq \sum_{z\in I} \pi^*_\infty(\sigma_x\, \sigma_z)\, \pi_\infty^*(\s_u\,\s_z)\,,
\end{align*}
where here $\pi_\infty^*$ denotes the (unique) Gibbs measure on $\Z^2$ at $\beta^*$. Altogether,
\begin{align*}
\Phi_1 &\leq \sum_{z \in I}  \pi_\infty^*(\s_x\,\s_z)
\pi_\infty^*(\s_z\,\s_u) \pi_\infty^*(\s_y\,\s_v) \\
&\leq \sum_{z\in I} \exp\left(-\tau_{\beta}(z-x)-\tau_{\beta}(z-u)-\tau_{\beta}(y-v)\right)\,.
 \end{align*}
At the same time, due to the sharp triangle inequality property of the surface tension (Eq.~\eqref{eq-sharp-tri-ineq}), for any $\beta>\beta_c$ there exists some $\kappa_\beta>0$ such that for any $u,x,z$
 \begin{align*}
 \tau_{\beta}(z-x)&+\tau_{\beta}(u-z) \geq \tau_{\beta}(u-x) + \kappa_\beta \left(|z-x|+|u-z|-|u-x|\right)\,.
 \end{align*}
Combining the last two displays implies that
\begin{align}
\label{eq-Phi1}
\Phi_1 &\leq e^{-\tau_{\beta}(x-u)-\tau_{\beta}(y-v)} \sum_{z\in I} \exp\left[
-\kappa_\beta \left(|z-x|+|u-z|-|u-x|\right) \right]\,.
 \end{align}
The first term above cancels when combining~\eqref{eq-Phi1},\eqref{eq-Phi2} and we obtain that for some $c''>0$,
\begin{equation}\label{eq-Phi1/Phi2}
\frac{\Phi_1}{\Phi_2} \leq c'' \ell^{c''} \sum_{z\in I} \exp\left[
-\kappa_\beta \left(|z-x|+|u-z|-|u-x|\right) \right]\,.
\end{equation}
A straightforward manipulation of the above exponent will now complete the proof: One can easily infer from the triangle-inequality and symmetry
that
\[\min_{z \in I} \left(|z-x|+|u-z|-|u-x|\right) = |v-x|-|u-x|\,,\]
and recalling that the dimensions of $R$ are $\alpha\sqrt{\ell}\times\ell$ and $|u-v|=s\alpha^2$, we have
\begin{align*}
|v-x|^2 &= \tfrac14(\ell + s\alpha^2)^2 + \alpha^2 \ell \,,&\quad& |u-x|^2 = \tfrac14(\ell - s\alpha^2)^2 + \alpha^2 \ell\,.
\end{align*}
Hence, as long as $\alpha < (1/s)\sqrt{\ell}$ (guaranteed by the assumptions of Proposition~\ref{prop:Bound2}) we have $|u-x|+|v-x| \leq 2\ell$ and
\[ |v-x| - |u-x| = \frac{s\alpha^2 \ell}{|u-x|+|v-x|} \geq \frac12 s\alpha^2 \,.\]
Plugging this in~\eqref{eq-Phi1/Phi2} while summing over the $|I|$ values that $z$ can assume gives that
\[ \frac{\Phi_1}{\Phi_2} \leq c'' \ell^{c''+1} \exp(-(\kappa_\beta/2)s \alpha^2)\]
(with room to spare), as required.
\end{proof}

\section*{Acknowledgments}
We are very grateful to Yvan Velenik for valuable suggestions on the
random-line representation and related equilibrium estimates, as well as for pointing out the novelty
of Corollary~\ref{cor-spin-spin-strip}. EL and AS are grateful to  the
Mathematics Department of the University of Roma Tre for hospitality and support.
FM and FLT are similarly grateful to the Theory Group of Microsoft
Research.
FLT acknowledges partial support by ANR grant LHMSHE.

\end{document}